\documentclass[a4paper,11pt]{article}
\usepackage[T1]{fontenc}
\usepackage[utf8x]{inputenc}
\usepackage{lmodern}
\usepackage{amsmath}
\usepackage{amsthm}
\usepackage{amssymb}
\usepackage{authblk}
\usepackage{graphicx}
\usepackage{stmaryrd}
\usepackage{enumitem}
\usepackage{xcolor}
\usepackage{dsfont}
\newcommand\1{\mathds{1}}
\usepackage{mathrsfs}

\usepackage{amsthm}
\usepackage{amssymb}
\usepackage{xparse}
\usepackage{thmtools}
\usepackage{stackrel}
\usepackage{bbold}
\usepackage{relsize}
\usepackage{tikz}

\usetikzlibrary{hobby}

\newcommand{\R}{\mathbb{R}}
\newcommand{\C}{\mathbb{C}}
\newcommand{\Z}{\mathbb{Z}}
\newcommand{\N}{\mathbb{N}}

\newcommand{\dd}{\mathrm{d}}

\DeclareMathOperator{\diag}{diag}

\DeclareMathOperator{\dist}{dist}

\makeatletter
\newtheoremstyle{indented}
{7pt} %
{7pt} %
{} %
{1.5em} %
{\bfseries} %
{.} %
{.5em} %
{} %

\theoremstyle{definition}

\newtheorem{defn}{Definition}[section]

\theoremstyle{plain}
\newtheorem*{theorem*}{Theorem}

\newtheorem{theorem}[defn]{Theorem}

\newtheorem{prop}[defn]{Proposition}

\newtheorem{lem}[defn]{Lemma}

\theoremstyle{definition}
\newtheorem{rem}[defn]{Remark} %

\makeatletter
\renewcommand*\env@matrix[1][*\c@MaxMatrixCols c]{%
  \hskip -\arraycolsep
  \let\@ifnextchar\new@ifnextchar
  \array{#1}}
\makeatother

\newcommand{\re}{{\rm Re}}
\newcommand{\im}{{\rm Im}}
\usepackage{hyperref} \hypersetup{colorlinks=true, citecolor=orange, linkcolor=blue}

\makeatletter
\newcommand{\widetildeoverline}[1]{{%
  \mathpalette\double@widetilde{#1}%
}}
\newcommand{\double@widetilde}[2]{%
  \sbox\z@{$\m@th#1\overline{#2}$}%
  \ht\z@=.9\ht\z@
  \widetilde{\box\z@}%
}
\makeatother

\usepackage{a4wide}
\title{The Szeg\H{o} kernel in analytic regularity and analytic
  Fourier Integral Operators}
\author{Alix
  Deleporte\thanks{alix.deleporte@universite-paris-saclay.fr}}
\affil{Universit\'e Paris-Saclay, CNRS, Laboratoire de math\'ematiques d'Orsay, 91405, Orsay, France.}

\usepackage{scrtime}

\begin{document}

\maketitle

\begin{abstract}
  We build a general theory of microlocal (homogeneous) Fourier Integral Operators in real-analytic
  regularity, following the general construction in the smooth case by
  Hörmander and Duistermaat. In particular, we prove that the Boutet-Sjöstrand parametrix for the Szeg\H{o}
  projector at the boundary of a strongly pseudo-convex real-analytic domain
  can be realised by an analytic Fourier Integral Operator. We then study some applications,
  such as FBI-type transforms on compact, real-analytic Riemannian
  manifolds and propagators of one-homogeneous (pseudo)differential operators.
\end{abstract}
 \tableofcontents

\section{Introduction}
\label{sec:introduction}
This project began as an investigation of the following problem: let
$\Omega\subset \C^n$ be an open set with real-analytic boundary; can we
characterise the analytic singularities of the functions on the boundary $\partial \Omega$ which are
boundary values of holomorphic functions inside $\Omega$?

The problem is well-posed when $\Omega$ is \emph{strongly
  pseudoconvex} (see Definition \ref{def:pseudoconvex}) and a precise
description modulo smooth functions of the Szeg\H{o} projector $S$ (from $L^2(\partial \Omega)$
onto the subspace of boundary values of holomorphic functions) is
well-known in the case where $\partial \Omega$ is smooth
\cite{boutet_de_monvel_sur_1975,fefferman_bergman_1976}. It is then
natural to expect that, if $\partial \Omega$ is analytic, then an even more
precise description of $S$ can be reached. This microlocal problem admits a
semiclassical (small parameter) counterpart, about the description of
spaces of holomorphic sections of high powers of ample line bundles
over Kähler manifolds. In the latter case, a description of the
Szeg\H{o} kernel in analytic regularity was recently obtained
\cite{rouby_analytic_2018,deleporte_toeplitz_2018,hezari_off-diagonal_2017,hezari_property_2021,charles_analytic_2021,deleporte_direct_2022}. 

It quickly became obvious to the author that a natural course of action
to study the analytic singularities of $S$ was to adapt the method of
\cite{boutet_de_monvel_sur_1975}, where $S$ is described as a Fourier
Integral Operator with complex-valued phase function, and where the
proof uses a microlocal normal form procedure (and in particular, a
conjugation by a Fourier Integral Operator). Building a theory of
analytic Fourier Integral Operators that lies as close as possible to the smooth
constructions \cite{hormander_analysis_1985} became our next goal and
constitutes the bulk of this article. 

In the microlocal as well as semiclassical case, the study of the
Szeg\H{o} projector has important applications in geometric
quantization and ``general purpose'' microlocal/semiclassical
analysis. One particular instance of strongly pseudoconvex open set
with real-analytic boundary is the Grauert tube around a real-analytic
Riemannian manifold; the range of $S$ is then the target space of a
natural FBI-type transformation. Thus the analytic counterpart of the
results of \cite{boutet_de_monvel_sur_1975} serves as a tool for other
problems related to the spectrum and dynamics of differential and
pseudo-differential operators. These techniques could be
particularly relevant to the study of non-self-adjoint analytic
(pseudo)differential operators.

\subsection{What this article contains}
\label{sec:what-this-article}

Section \ref{sec:bound-cauchy-riem} reviews some basic properties of
the boundary Cauchy-Riemann problem for pseudoconvex open sets. The
impatient reader only interested in ``off-the-shelf'' properties of analytic
Fourier Integral Operators may skip it, however the pseudoconvex open
set $\R^n\times S^{n-1}$, on which the Szeg\H{o} projector admits
an exact formula, is useful in the construction of the general theory.

Section \ref{sec:analytic-symbols} contains the ``formal analytic''
content of our construction. We describe spaces of analytic amplitudes
(formal or not) in Section \ref{sec:ampl-analyt-micr}, then we study the geometric properties of
homogeneous phase functions and their associated conical
Lagrangians in Section \ref{sec:posit-nond-phas}. We also prove a
stationary phase type theorem in Section \ref{sec:stationary-phase}. Most of
the contents of this section will look familiar to experts, but many
details (in particular, the sense in which we perform Borel summations
of formal analytic amplitudes in the microlocal setting) turn out to
be non-trivial and crucial.

The core of the construction of analytic Fourier Integral Operators is
Section \ref{sec:four-integr-oper}. We properly define analytic
Fourier Integral Operators in Section \ref{sec:first-properties}, in
an ascending level of geometric generality, and study the composition
rule and the action of stationary phase; our results are gathered in Theorem \ref{prop:FIOs}. The algebra of analytic pseudodifferential operators, which of
course is of utmost importance, is described in Section
\ref{sec:pseud-oper-1}; in our framework, to microlocalize the action
of these operators we introduce a FBI transform, which conjugates the
problem to operators on $\R^n\times S^{n-1}$ whose microlocal
structure resembles that of the Szeg\H{o} projector. As in
\cite{boutet_de_monvel_spectral_1981}, we call them ``Toeplitz
operators'', even though they are not best described as $SQS$ for $Q$
a pseudodifferential operators, but rather by a Fourier Integral
Operator formula, akin to the ``covariant Toeplitz operators'' of
\cite{charles_berezin-toeplitz_2003}. These Toeplitz operators turn
out to be a very efficient microlocal replacement of pseudodifferential
operators, allowing us to describe refined properties of Fourier Integral Operators
(notably inversion) microlocally in Section \ref{sec:advanced-properties}. Besides
pseudodifferential operators and Toeplitz operators, important
examples of Fourier Integral Operators are ``quantized contact
transformations'': one can quantize the action of a one-homogeneous
symplectic change of variables into a unitary operator, and we do so
in Proposition \ref{prop:quantized-contact}.

We then turn our attention to the specific problem of the Szeg\H{o}
projector for a general pseudoconvex open set in Section
\ref{sec:szegho-kern-toepl}. We prove a normal form theorem for the
$\overline{\partial}_b$ operator in analytic regularity, and use it to
describe the Szeg\H{o} projector as an analytic Fourier Integral
Operator in Theorem \ref{prop:analytic-Szego}. We
then generalise the properties of Toeplitz operators obtained in
Section \ref{sec:pseud-oper-1} to this more general geometric setting.

Some illustrations of our techniques are presented in Section
\ref{sec:few-applications}. We discuss in particular the case of
Grauert tubes, and how a natural FBI-type transformation on compact
analytic Riemannian manifolds helps to translate ``usual'' questions
concerning differential or pseudodifferential operators to the
Toeplitz framework, where they are more easily studied. We prove in
Proposition \ref{prop:quantum_propag} that propagators of one-homogeneous, self-adjoint analytic
pseudodifferential operators are analytic Fourier Integral
Operators. We also discuss briefly how to obtain semiclassical results from
our microlocal toolbox. To our knowledge, there is no general theory of analytic
Fourier Integral Operators available in the semiclassical case (and
certainly not one which preserves analytic function spaces), and we hope that
our results can be of use in this direction.

The Appendix presents some well-known facts about analytic
semiclassical analysis, which we gather for the comfort of the
reader. We discuss in particular the properties of spaces of analytic
functions, their dual spaces (called \emph{analytic functionals}), and
a convenient quotient space of analytic functionals called
hyperfunctions.

\subsection{Comparison with earlier work}
\label{sec:comp-with-earl}
Microlocal analysis in real-analytic regularity has a long history,
but its early
achievements have been shadowed by the success and popularity of the
$C^{\infty}$ techniques.

Building on an observation that in some
\emph{formal} sense, there are classes of pseudodifferential operators
adapted to real-analytic regularity
\cite{boutet_de_monvel_pseudo-differential_1967}, a general framework was
proposed \cite{sato_microfunctions_1972} but ultimately, this
theory never reached the same foundational status as the $C^{\infty}$
counterpart. In particular, a method of proof of Theorem
\ref{prop:analytic-Szego}
using the tools of \cite{sato_microfunctions_1972} appears in
\cite{kashiwara_analyse_1976}, but the status of this proof and of the
statement itself is disputed; we mention
\cite{kaneko_introduction_1989} for a more detailed, but incomplete,
discussion of Kashiwara's approach. One of the initial goals of this
work was the hope to settle once and for all the status of Theorem \ref{prop:analytic-Szego}.

In contrast, the theory of \emph{semiclassical} analysis in
real-analytic regularity is now well-developed, after the seminal work
\cite{sjostrand_singularites_1982}. In this setting, one is
interested in improving remainders from $O(\hbar^{\infty})$ to
$O(e^{-c\hbar^{-1}})$ for some $c>0$, rather than improving remainders
from smooth to real-analytic. The semiclassical theory makes full use
of complex-valued phase functions, whose manipulation is perhaps even
easier than in the $C^{\infty}$ case because of the naturalness of
extensions into the complex.

It is often said that microlocal analysis is a particular case of
semiclassical analysis. We found out, however, that the usual semiclassical
constructions of pseudodifferential operators, other
Fourier Integral operators, or even Borel summations of analytic
symbols, was not applicable here. This is chiefly due to the fact that,
when presented with an oscillating integral with small parameter, of
the form $\int e^{i\frac{\phi(x)}{\hbar}}a(x)\dd x$, we may remove the
areas where $\im(\phi)>0$ by a simple (smooth) cut-off argument,
leading to $O(e^{-c\hbar^{-1}})$ errors . As a consequence, at its present
stage of development, analytic semiclassical analysis deals with
errors of size $O(e^{-c\hbar^{-1}})$ in $C^{k}$ topology for fixed
$k$. In the microlocal setting, we cannot afford to introduce cut-offs
even in areas of phase space which we presume are away from the
analytic wave front set. Even apart from stationary phase, cut-and-paste arguments, used to pass from local
definitions of objects in coordinate charts to global descriptions,
must be replaced by slightly more refined arguments of cohomological
nature: broadly speaking, we are able to define how objects act
locally and then we can use somewhat abstract gluing tools. This is a
standard approach in the study of holomorphic or real-analytic
functions, but it turns out (see the Appendix) that these arguments
already work in the $C^{\infty}$ setting.

Conversely, at least the elementary results of analytic semiclassical
analysis can be obtained from our microlocal setting by considering
Fourier modes in an auxiliary variable; we describe this briefly in Section
\ref{sec:from-micr-semicl}.

The book \cite{treves_analytic_2022} presents an
overview of the existing literature on general-purpose analytic
microlocal analysis and some of its applications. Sections \ref{sec:analytic-symbols} and
\ref{sec:four-integr-oper} have been
somewhat motivated by this book, as we found out that adapting the
homogeneous FIO framework to the real-analytic case is not as
straightforward as it seems and many statements and constructions of
this book
deserved a more thorough approach.

\subsection{Acknowledgements}
\label{sec:acknowledgements}

We are most grateful to Michael Hitrik and Johannes Sjöstrand for
their encouragement and useful remarks throughout the realisation of
this project.

This work is dedicated to the memory of Steve Zelditch, who motivated
us to give a complete proof of the analytic parametrix of the
Szeg\H{o} kernel in the first place.

\section{The boundary Cauchy-Riemann problem}
\label{sec:bound-cauchy-riem}

This section is a primer on the theory of boundary Cauchy-Riemann
problems; we present some facts which, besides being the basis for
Theorem \ref{prop:analytic-Szego}, will be useful for the general construction of Fourier
Integral Operators in real-analytic regularity.

\subsection{Strongly pseudoconvex domains and CR geometry}
\label{sec:strongly-pseud-doma}

\begin{defn}\label{def:pseudoconvex}Let $Y$ be a (paracompact, boundaryless) complex manifold of
  dimension $n\geq 2$.
  \begin{itemize}
  \item Let $U\subset Y$ open. A function $\rho\in
    C^2(U,\R)$ is \emph{strongly plurisubharmonic} when at every point of
    $U$, in local coordinates, the Hermitian matrix
    $[\frac{\partial^2\rho}{\partial z_j\partial
      \overline{z_k}}]_{j,k}$ is positive definite. (This does not
    depend on the coordinates).
  \item A compact real hypersurface $X\subset Y$ is \emph{strongly
    pseudoconvex} when there exists an open neighbourhood $U$ of
    $X$ and a strongly plurisubharmonic function $\rho:U\to R$, such that $\dd \rho$ never
    vanishes and $X=\{\rho=0\}$.
  \end{itemize}
\end{defn}
The situation of interest is the case where $X$ separates $Y$ into two
connected components; the ``inside'' component $\Omega$ (the one which
contains
$\{\rho<0\}$) is then called a \emph{strongly pseudoconvex domain},
independently on the behaviour of $\rho$ far inside $\Omega$.

Given a general open set $\Omega\subset Y$ with regular boundary, one
may ask how holomorphic functions on $\Omega$ behave near $\partial
\Omega$: whether they can be extended (using, for instance, the
Hartogs theorem) or whether their restrictions to $\partial \Omega$
admit a simple description. As it turns out, these two questions are
related and strongly pseudoconvex domains are exactly those for
which restrictions of holomorphic functions are solutions of a
hypoelliptic linear PDE, the boundary Cauchy-Riemann operator $\overline{\partial}_b$
\cite{kohn_harmonic_1963}. The Szeg\H{o}
projector $S:L^2(X)\to \ker_{L^2}(\overline{\partial}_b)$ allows to describe
the space of solutions, as well as (after application of the Poisson
kernel) the Bergman space of holomorphic functions on $\Omega$. The
description of $S$ modulo smooth functions
\cite{fefferman_bergman_1976,boutet_de_monvel_sur_1975} when $X$ is
smooth has numerous crucial applications in complex and Kähler
geometry as well as geometric quantization. For example, the local
definition of a strongly pseudoconvex hypersurface above, or even an
intrinsic definition mimicking the properties of $X$ gathered in
the next proposition, are equivalent to the realisation of $X$ as the
boundary of a strongly pseudoconvex open set in $\C^n$, with a globally
defined and strongly plurisubharmonic $\rho$, as soon as $n\geq 3$
\cite{boutet_de_monvel_integration_1974}.

\begin{prop}\label{prop:CR-structure}
 A strongly pseudoconvex hypersurface $X\subset Y$ inherits a
 CR-structure as follows: $TX\otimes \C$ decomposes as a direct sum of
 three involutive components $T^{(0,1)}X\oplus T^{(1,0)}X\oplus V$, where
 \begin{itemize}
 \item $T^{(1,0)}X=T^{(1,0)}Y\cap (TX\otimes \C)$ is the
   $n-1$-dimensional bundle of holomorphic tangent vectors.
 \item $T^{(0,1)}X=T^{(0,1)}Y\cap (TX\otimes \C)$ is the
   $n-1$-dimensional bundle of anti-holomorphic tangent vectors.
 \item $V$ is a line bundle over $X$.
 \end{itemize}
 The line bundle $\Sigma=(T^{(0,1)}X\oplus T^{(1,0)}X)^*\subset T^*X\otimes
 \C$ is spanned by the real one-form $\alpha:(x,\xi)\mapsto \dd
 \rho(x)(J\xi)$ whenever $\rho$ is a pseudoconvex defining function
 for $X$. Moreover $\alpha$ is a contact form on $X$, and $\Sigma$ is
 a symplectic submanifold of $T^*X$.
\end{prop}
\begin{proof}
  Let us prove that $\alpha$ is a contact form; this follows from an explicit computation in local coordinates
  \[\alpha\wedge (\dd\alpha)^{n-1}=\|\nabla
  \rho\|
  \det([\partial_{z_j}\partial_{\overline{z_k}}\rho]_{j,k})\dd {\rm
    vol},\]
  where $\dd {\rm vol}$ is the standard volume form in coordinates and
  the Hessian of $\rho$ is restricted to $T^{(0,1)}X\otimes
  T^{(1,0)}X$, where it is still non-degenerate.

  The other claims follow directly from there.
\end{proof}

Given $u\in C^1(X,\C)$, one naturally forms $\overline{\partial}_bu$
as the section of $(T^{(0,1)}X)^*$ obtained by restricting $\dd u$ to
$T^{(0,1)}X$. Restrictions to $X$ of holomorphic functions naturally
satisfy $\overline{\partial}_bu=0$, and reciprocally elements of $\ker_{L^2}
\overline{\partial}_b$ extend into $H^{\frac 12}$ holomorphic
functions on $\Omega$ in the strongly pseudoconvex case \cite{kohn_harmonic_1963}.

In this article we will only be interested in the situation where $X$
(that is, $\rho$) is real-analytic. The symbol of the differential
operator $\overline{\partial}_b$ can then be extended in the complex,
which allows to define several manifolds of interest.

\begin{prop}\label{prop:Z_and_Sigma}
  Let $Z=\{(x,\xi)\in \widetilde{T^*X},
  \widetilde{\sigma(\overline{\partial}_b)}(x,\xi)=0\}$. Then $Z$ is
  a complex submanifold of $\widetilde{T^*X}$ of complex codimension
  $n-1$.

  $Z$ is coisotropic, and defining $\overline{Z}=\{(x,\xi)\in \widetilde{T^*X},
  \widetilde{\sigma(\partial_b)}(x,\xi)=0\}$, then $Z$ and
  $\overline{Z}$ intersect cleanly and 
  \[
    Z\cap \overline{Z}=\widetilde{\Sigma}.
  \]

  Moreover $T\Sigma$ is transverse to the
  null-distribution $(TZ)^{\perp}=\{v\in T(\widetilde{T^*X}), \forall
  v'\in TZ, \omega(v,v')=0\}$.
\end{prop}
\begin{proof}
  Let us consider holomorphic coordinates mapping a neighbourhood in $Y$ of a
  point of $X$ to a neighbourhood of $0$ in $\C^n$, with
  $\partial_{\overline{z_n}}\rho\neq 0$. Then as a subset of
  $\widetilde{T^*\C^n}=\{(x,y,\xi,\eta)\in \C^{4n}\}$,
  $Z$ is the joint zero set of the functions
  \begin{align*}
    p_j:(x,y,\xi,\eta)&\mapsto
                        \xi_j+i\eta_j-\frac{(\partial_{x_j}+i\partial_{y_j})\widetilde{\rho}(x,y)}{(\partial_{x_n}+i\partial_{y_n})\widetilde{\rho}(x,y)}(\xi_n+i\eta_n)
               \qquad \qquad 1\leq j \leq n-1\\
    \widetilde{\rho}:(x,y,\xi,\eta)&\mapsto \widetilde{\rho}(x,y)\\
    \widetilde{\dd \rho}:(x,y,\xi,\eta)&\mapsto \sum_{j=1}^n\xi_j\partial_{x_j}\widetilde{\rho}(x,y)+\eta_j\partial_{y_j}\widetilde{\rho}(x,y)
  \end{align*}
  The differentials of these functions are linearly independent, and
  one can check that the functions $p_j$ Poisson-commute with each
  other. Therefore $Z$ is a coisotropic submanifold of
  $\widetilde{T^*X}$ whose codimension is the number of functions
  $p_j$ (i.e.\ $n-1$).

  To define $\overline{Z}$, we replace $p_j$ with the functions
  \[
    p_j^*:(x,y,\xi,\eta)\mapsto
    \xi_j-i\eta_j-\frac{(\partial_{x_j}-i\partial_{y_j})\widetilde{\rho}(x,y)}{(\partial_{x_n}-i\partial_{y_n})\widetilde{\rho}(x,y)}(\xi_n-i\eta_n)
    \qquad \qquad 1\leq j \leq n-1.
  \]
  These functions again Poisson-commute with each other.

  Now, clearly $\Sigma=Z_{\R}=(Z\cap \overline{Z})_{\R}$. The symplectic
  gradients of the functions $p_1,\ldots,p_{n-1},p_1^*,\ldots,p_{n-1}^*$ are
  linearly independent on $\Sigma$, and in fact the matrix of Poisson brackets
  $\frac{1}{i}[\{p_j,p_k^*\}]_{j,k}$ is everywhere either positive definite (on
  the half-line bundle $\Sigma_+$ of positive multiples of $\alpha$)
  or negative definite (on the half-line bundle $\Sigma_-$ of negative
  multiples of $\alpha$), see the computations in
  \cite{boutet_de_monvel_sur_1975}. This concludes the proof; in
  particular, the Hamiltonian vector fields of $p_1,\ldots,p_{n-1}$, which span $(TZ)^{\perp}$, are transverse to $\Sigma$.
\end{proof}

From now on we decompose $\Sigma=\Sigma_+\cup\Sigma_-\cup\{0\}$ where
$\Sigma_\pm$ is the half-line bundle of positive multiples of $\pm
\alpha$, on which $\pm\frac{1}{i}[\{p_j,p_k^*\}]_{j,k}$ is positive definite.

An important consequence of the geometry of $\Sigma$ and $Z$ is the
following CR-extension result for $\Sigma$-Lagrangians.

\begin{prop}\label{prop:completing_Lagrangians}
  Let $\lambda\subset \Sigma_+\times \Sigma_+$ be a real-analytic conical
  Lagrangian for the twisted symplectic form $\omega_1-\omega_2$.

  There exists a unique $\Lambda\subset Z\times \overline{Z}$,
  Lagrangian in $T^*X\times T^*X$ for the twisted symplectic form,
  which contains $\lambda$. Moreover $\Lambda$ is positive (in the
  sense of \cite{melin_fourier_1975}, Definition 3.3).
\end{prop}
\begin{proof}
 We proceed as in \cite{boutet_de_monvel_sur_1975}. Let $\Lambda\subset Z\times \overline{Z}$ be a Lagrangian which
  contains $\widetilde{\lambda}$. Then the null-distribution $(TZ^{\perp}\times
  \{0\})\oplus(\{0\}\times 
  T\overline{Z}^{\perp})$ belongs to $T\Lambda$. By
  Proposition \ref{prop:Z_and_Sigma}, this
  distribution is involutive (because $Z$ is coisotropic); it is
  transverse to $T\tilde{\lambda}$; its complex dimension is
  $2(n-1)=\dim(\Lambda)-\dim(\widetilde{\lambda})$. Therefore
  $\Lambda$ is exactly the union of the leaves of this distribution
  which pass through $\widetilde{\lambda}$. The positivity of $\Lambda$ is then a
  direct consequence of the positivity of the matrix
  $\tfrac 1i[\{p_j,p_k^*\}]_{j,k}$ on $\Sigma_+$.
\end{proof}

\subsection{Classical normal form}
\label{sec:class-norm-form}
At the level of the operator $\overline{\partial}_b$, the
strong pseudoconvexity of $X$ is reflected in the following facts:
\begin{itemize}
  \item the real
characteristic $\Sigma=\{(x,\xi)\in
T^*X,\sigma(\overline{\partial}_b)=0\}$ is symplectic
\item denoting
$(p_1,\ldots,p_{n-1})$ the components of $\sigma(\overline{\partial}_b)$,
one has $\{p_j,p_k\}=0$ for all $j,k$, while the matrix
$\frac{1}{i}[\{p_j,\overline{p_k}\}]_{j,k}$ is everywhere either positive
definite or negative definite.
\end{itemize}
There is a universal local model (depending only on the dimension) for
the principal symbols of such operators, whose construction is presented in Appendix II of
\cite{boutet_de_monvel_hypoelliptic_1974} in the smooth case. We 
review this proof in the analytic setting. This local model is
an important tool in the description of the Szeg\H{o} projector in the
smooth case, and we will also use it in our setting. The local model
we will use is the vector-valued differential operator $D_0$ acting on
$\R^{n-1}_x\times \R^n_y$:
  \[
  (D_0)_j=-i\frac{\partial}{\partial x_j}-x_j\frac{\partial}{\partial
    y_1}\qquad \qquad 1\leq j \leq n-1.
\]
This differs from the local model considered in
\cite{boutet_de_monvel_sur_1975}. Our choice of $D_0$, with a polynomial total
symbol, will facilitate
handling subprincipal terms in the quantum normal form in Section
\ref{sec:norm-form-overl}.

\begin{lem}\label{prop:extend_zero_bracket}[See also
  \cite{boutet_de_monvel_hypoelliptic_1974}, Proposition 10.10, for
  the $C^{\infty}$ statement] Let $1\leq r\leq m$. Let $\Sigma\subset T^*\R^m$ be an open cone. Let $p_1,\ldots,p_r$ be
  $\frac 12$-homogeneous, real-analytic, complex-valued functions near $\Sigma$, such
  that
  \begin{itemize}
  \item All the Poisson brackets $\{p_j,p_k\}$ and $\{\overline{p_j},\overline{p_k}\}$, for
    $1\leq j,k\leq r$, vanish on $\Sigma$.
  \item All the Poisson brackets $\frac{1}{i}\{p_j,\overline{p_k}\}$, for $1\leq j,k\leq r$,
    are equal to $\delta_{jk}$ on $\Sigma$.
  \end{itemize}
  Then, in a conical neighbourhood of $\Sigma$, there exist
  $p_1',\ldots,p_r'$, homogeneous, real-analytic,
  equal respectively to $p_1,\ldots,p_r$ along with
  their first derivatives on $\Sigma$, and such that
  $\{p'_j,p'_k\}=\{p'_j,\overline{p'_k}\}=\frac{1}{i}\{p'_j,\overline{p'_k}\}-\delta_{jk}$ is equal to
  $0$ everywhere, and the ideal generated by $p_1,\ldots,p_r$
  coincides with the ideal generated by $p_1',\ldots,p_r'$.  
\end{lem}
\begin{proof}
  We proceed by induction on $r$. The claim follows immediately from the following
  two properties:
  \begin{enumerate}
  \item Let $V\subset T^*\R^n$ be a conical open set. Let $p:V\to \C$
    be $\tfrac 12$-homogeneous, real-analytic; suppose that
    $\{p,\overline{p}\}$ bounded away from $0$ on $V$. (In particular,
    $\dd p$ is bounded away from $0$ on $V$.)

    Then there exists a small neighbourhood $W$ of $\{p=0\}$, and $p':W\to \C$, with the same properties, such
    that $\{\widetilde{p}=0\}=\{\widetilde{p'}=0\}$, and
    $\frac{1}{i}\{p',\overline{p'}\}=1$. Moreover, for every function
    $q:V\to \C$ such that $\{p,q\}=\{\overline{p},q\}=0$, one has also
    $\{p,q\}=\{p',\overline{q}\}=0.$
  \item Let $V\subset T^*\R^n$ be a conical open set. Let $p_1,\ldots,p_r,q:V\to \C$
    be $\tfrac 12$-homogeneous and real-analytic; suppose that
    $\{\widetilde{p}_1=\ldots=\widetilde{p}_r=0\}$ is a regular energy
    level near $V$, and that
    $\{p_j,p_k\}=\{\overline{p_j},\overline{p_k}\}=\frac{1}{i}\{p_j,\overline{p_k}\}-\delta_{jk}=0$. Suppose
    also that $\{p_j,q\},\{\overline{p_j},q\}$ vanish on
    $\{p_j=\overline{p_j}=0\}$, on which $\{q,\overline{q}\}$ is
    bounded away from $0$.

    Then there exists a small neighbourhood $W$ of $\{p_j=\overline{p_j}=0\}$ and
    $q':W\to \C$ a $\frac{1}{2}$-homogeneous, real-analytic function,
    which is equal to $q$ along with its first total differential on
    $\{p_j=\overline{p_j}=0\}$ (in particular $\{q',\overline{q'}\}$
    is bounded away from $0$ and the ideal generated by
    $p_1,\ldots,p_r,q$ coincides with the ideal generated by
    $p_1,\ldots,p_r,q'$), and such that
    $\{p_j,q\}=\{\overline{p_j},q\}=0$ everywhere.
  \end{enumerate}
   Let us prove property 1 first. We want to correct
   $p$ into $p'=ap$, where $a$ is real-valued. After this modification, the symplectic bracket
   with the complex conjugate reads
   \[
     \{p',\overline{p'}\}=|a|^2\{p,\overline{p}\}+a\overline{p}\{p,a\}+ap\{a,\overline{p}\}.
   \]
   Letting $A=a^2$, this simplifies into
   \[
     \{p',\overline{p'}\}=A\{p,\overline{p}\}+\frac 12
     p\{A,\overline{p}\}+\frac 12 \overline{p}\{p,A\}.
   \]
   Therefore we want to solve the transport equation
   \[
     Y\cdot A + \tfrac{1}{i}\{p,\overline{p}\}A=1
   \]
   where $Y$ is the (real-valued) vector field
   $\frac{1}{i}(-p\Xi_{\overline{p}}+\overline{p}\Xi_p)$ (where $\Xi_f$ denotes the
  symplectic gradient of $f$).

   This vector field is singular along the codimension 2 set
   $\{Y=0\}=\{p=\overline{p}=0\}$. Moreover it enjoys an absorbing property:
   all trajectories of $\pm Y$ converge exponentially fast to the
   singularity in positive time. Indeed,
   ${\rm div}(Y)=\frac{1}{i}\{p,\overline{p}\}$ is bounded away from
   $0$ (and therefore is either positive everywhere or negative
   everywhere).

   Therefore, by the method of characteristics, there exists a unique
   $A$ in a neighbourhood of $\{Y=0\}$, real-analytic, and such that
   $Y\cdot A + \frac{1}{i}\{p,\overline{p}\}A=1$, and moreover it is real-valued.

   If $q$ is such that $\{p,q\}=\{\overline{p},q\}=0$, then both  $Y$
   and $\{p,\overline{p}\}$ commute with $q$, so that $A$ does as well.

 To prove property 2, we simply observe that the complex submanifold
 $\{\widetilde{p_j}=\widetildeoverline{p_j}=0\}$ is transverse to
 the flows of $\widetilde{p_j}$ and $\widetilde{\overline{p_j}}$. We
 can therefore set $q'=\widetilde{q}$ on this set, then extend it to
 be constant along the flows of $\widetilde{p_j}$ and
 $\widetilde{\overline{p_j}}$.
\end{proof}

\begin{prop}\label{prop:canon_transf_to_local_model}
  Let $P_0\in \Sigma_{\pm}$. There exists a real-analytic contact
  transformation $\kappa$ from a conical neighbourhood of $P_0$ in
  $T^*X\setminus \{0\}$ to a conical neighbourhood of
  $(0,0, 0,(\pm 1,0))$ in $T^*(\R^{n-1}\times \R^{n})$, mapping
  $P_0$ to $(0,0,0,(\pm 1, 0))$, and a real-analytic, $0$-homogenous
  map $C$ from a conical neighbourhood of $x_0$ in
  $T^*X\setminus \{0\}$ to $GL(\C^{n-1},\Omega^{(1,0)}(X))$, such that
  \[
    \sigma(\overline{\partial}_b)(z,\zeta)=C\sigma(D_0)(\kappa(z,\zeta)).
  \]
\end{prop}

\begin{proof}(of Proposition \ref{prop:canon_transf_to_local_model})
  Let us apply Lemma \ref{prop:extend_zero_bracket} to the
  $\overline{\partial}_b$ operator: the components
  $Z_1,\ldots,Z_{n-1}$ of its symbol satisfy, near any point of the
  characteristic set $\Sigma^{\pm}$,
  \[
    \{Z_j,Z_k\}=0 \text{ on }\Sigma \qquad \qquad
    \pm\tfrac{1}{i}\{Z_j,\overline{Z_k}\}>>0.
  \]
  Let $M:\Sigma^+\to GL(n-1)$ be the positive square root of the
  matrix $\pm\tfrac{1}{i}\{Z_j,\overline{Z_k}\}$; extend $M$ arbitrarily into a
  $\frac 12$-homogeneous, real-analytic function from a neighbourhood
  of $\Sigma^+$ into $GL(n-1)$. Let $z=M^{-1}Z$. The ideal generated
  by $z$ coincides with that generated by $Z$, and on particular $z$
  vanishes on $\Sigma^+$, where
  \[
    \{z_j,z_k\}=\sum_{l,m}(M^{-1})_{jl}(M^{-1})_{mk}\{Z_l,Z_m\}=0,
  \]
  whereas
  \[
    \tfrac{1}{i}\{z_j,\overline{z_k}\}=\sum_{l,m}(M^{-1})_{jl}\overline{(M^{-1})_{km}}\tfrac{1}{i}\{Z_l,\overline{Z_m}\}=\pm
    \delta_{jk}.
  \]
  Then, we let $u_j=\re(z_j)$ and $v_j=\pm \im(z_j)$ and directly
  obtain the situation of Lemma \ref{prop:extend_zero_bracket}:
  \[\{u_j,u_k\}=\{v_j,v_k\}=0 \qquad \qquad \{v_j,v_k\}=\delta_{jk}.\]

  Letting $z'_j=u'_j+iv'_j$, then the ideal generated by $z'$
  coincides with that generated by $Z$.

  We are almost ready for the normal form. We let $w$ be any
  non-vanishing, $\frac 12$-homogeneous, real-valued function on a
  neighbourhood of $\Sigma$ which commutes with $z'_1,\ldots,z'_{d}$,
  and we consider the following functions:
  \[
    x_j=u'_jw^{-1} \qquad \xi_j=v_j'w \qquad \eta_1=\pm w^2.
  \]
  By construction, the $x_j$ are $0$-homogeneous while $\eta_1$ and
  the $\xi_j$ are $1$-homogeneous, and
  $\{x_j,x_k\}=\{x_j,\eta_1\}=\{\xi_j,\eta_1\}=\{x_j,\xi_k\}-\delta_{jk}=0$. The
  ideal of real-analytic functions generated by the one-homogeneous
  family $(d_0)_j=x_j\eta_1+i\xi_j=wz'_j$ coincides with the ideal of
  real-analytic functions generated by $z'$, which is the same as the
  ideal generated by $Z$. All in all, we obtain that, for some
  invertible matrix $C$, one has $Z=Cd_0$.

  Completing the coordinates
  $(x_1,\ldots,x_{d},\xi_1,\ldots,\xi_{d},\eta_1)$ into a canonical
  transformation of a neighbourhood of $\Sigma^+$ into a neighbourhood
  of $x=\xi=0,y=0,\pm \eta_1>0,(\eta_2,\ldots,\eta_n)=0$ in $T^*\R^{n-1}\times T^*\R^n$, we can
  exhibit a canonical transform $\kappa$ as requested.
\end{proof}

\subsection{Results from the smooth theory}
\label{sec:results-from-smooth}

Propositions
\ref{prop:Z_and_Sigma}, \ref{prop:completing_Lagrangians},
\ref{prop:extend_zero_bracket} and \ref{prop:canon_transf_to_local_model}
are also valid in the $C^{\infty}$ category, after a suitable
replacement of holomorphic extensions by almost holomorphic
extensions (in fact, our method of proof for the classical normal form applies to this situation as
well, and is much simpler than the one presented in Appendix 2 of \cite{boutet_de_monvel_hypoelliptic_1974}). The main result of \cite{boutet_de_monvel_sur_1975} is
then that, with $\lambda={\rm diag}(\Sigma_+)$, and $\Lambda$ as in
Proposition \ref{prop:completing_Lagrangians}, then the Szeg\H{o}
projector $S$ on $X$ is a Fourier Integral operator associated with
$\Lambda$, modulo an operator which sends $\mathcal{D}'(X)$ into
$\mathcal{D}(X)$. Since $\Lambda$ is a half-line bundle over $X\times
X$, one can represent this Fourier Integral operator, globally, using
a phase function with one phase variable, as follows:
\[
  S(x,y)=\int_0^{+\infty}e^{it\psi(x,y)}s(x,y;t)\dd t.
\]
One can construct $\psi$ using the almost holomorphic extension of a
defining function $\rho$.

The Lagrangian $\Lambda$ is idempotent (and indeed $S$ is a
projector), and a study of Fourier Integral Operators with
Lagrangian $\Lambda$ is  performed in
\cite{boutet_de_monvel_spectral_1981}. This book avoids the notion of
Fourier Integral operators with complex phases and only considers
operators of the form $SPS$ where $P$ is a pseudodifferential operator
on $X$, but it turns out that general
Fourier Integral Operators with Lagrangian $\Lambda$, in the smooth
category, are of this form. We will use the following precise version
of this result.

\begin{prop}\label{prop:repr_BG_smooth}
  Let $A$ be a (smooth) Fourier Integral operator with Lagrangian
  $\Lambda$. Then there exists a
  pseudodifferential operator $Q$ on $X$ such that $[Q,S]$ and $A-SQD$ are
  continuous from $\mathcal{D}'$ to $\mathcal{D}$.
\end{prop}
\begin{proof}
  By induction we only need to prove the result at principal order, and by
  Proposition 2.13 in \cite{boutet_de_monvel_spectral_1981} we may
  drop the condition that $Q$ and $S$ commute. Therefore we are only
  searching for the principal symbol $q_0$ of a pseudodifferential
  operator $Q$, such that the principal symbol of $SQS$ is that of
  $A$. An application of stationary phase shows that $SQS$ has an
  distributional kernel of the form
  \[
    SQS(x,y)=\int_0^{+\infty}e^{it\psi(x,y)}a(x,y;t)\dd t,
  \]
  with principal symbol $a:\Sigma_+\to \R$ equal to $q_0$. Hence,
  given $A$, we can extend arbitrarily its principal symbol to the
  whole of $T^*X$ and obtain an approximation at first order.
\end{proof}

\subsection{The case of the cylinder}
\label{sec:an-exampl-eucl}

A particular case of a pseudoconvex hypersurface (albeit not compact)
is $\{z\in \C^n, |\im(z)|=1\}$; indeed the function $\rho:z\mapsto
|\im(z)|^2$ is strongly plurisubharmonic everywhere.

There is an explicit description of the Szeg\H{o} kernel in this case,
and one can directly establish many properties which we will
generalise only in Section \ref{sec:szegho-kern-param}. This
particular case will be useful in our construction and
characterisations of general Fourier Integral Operators in
real-analytic regularity in Section \ref{sec:pseud-oper-1}.

\begin{prop}\label{prop:flat_Grauert}
  The Szeg\H{o} kernel on $\R^n\times S^{n-1}$ is
  \[
    K(z,w)=\frac{1}{(2\pi)^n}\int_{\R^n}\frac{e^{i(z-\overline{w})\cdot
        \xi}}{\int_{S^{n-1}}e^{2y\cdot \xi}\dd y}\dd \xi,
  \]
  which defines a holomorphic function of $(z,\overline{w})$ on $(\R^n\times \overline{B_{\R^n}(0,1)})^2\setminus {\rm
    diag}(\R^n\times S^{n-1})$.
  
  In particular, the Szeg\H{o} projector $S$ has a distributional
  kernel analytic away from the
  diagonal.
\end{prop}
A related formula concerning the Bergman kernel appears on p. 100 of the manuscript \cite{epstein_shrinking_1990}.
\begin{proof}
  The first step is the study of the function
  \[
    m_n:\xi\mapsto \int_{S^{n-1}}e^{y\cdot \xi}\dd y.
  \]
  This defines an holomorphic function on $\C^n$.
  When restricted to $\R^n$, it is a radial function and, by the
  Laplace method, its asymptotics for large argument are
  \begin{equation}\label{eq:mn_first_term}
    m_n(r\omega)=2^{\frac{n-1}{2}}|S^{n-2}|\Gamma(\tfrac{n+1}{2})r^{-\frac{n+1}{2}}e^r+O_{r\to
      +\infty}(e^rr^{-\frac{n+3}{2}}),
  \end{equation}
  where $|S^{n-2}|$ is the area of the unit $n-2$-sphere. %
  Let $(z,w)\in (\R^n\times \overline{B_{\R^n}(0,1)})^2\setminus {\rm
    diag}(\R^n\times S^{n-1})$. Then $|\im(z-\overline{w})|\leq 2$ and
  furthermore $|\im(z-\overline{w})|=2\Rightarrow
  \re(z-\overline{w})\neq 0$. Thus, by formula \eqref{eq:mn_first_term}, the integral
  \[
    \int_{\R^n}e^{i(z-\overline{w})\cdot \xi}(m_n(2\xi))^{-1}\dd \xi
  \]
  is well-defined, real-analytic on its domain, and holomorphic with
  respect to $(z,\overline{w})$.

  One can directly check that $K(z,w)$ is %
  self-adjoint.%
  
  It remains to prove that, for every $f\in L^2(\R^n\times
  S^{n-1})$ which extends into a holomorphic function inside
  $\R^n\times B(0,1)$, one has $f(z)=\int K(z,w)f(w)\dd w$. First,
  by the Fourier inversion formula,
  for every $x+iy\in \R^n+iS^{n-1}$,
  \[
    f(x+iy)=\frac{1}{(2\pi)^n}\int_{\R^n\times
      \R^n}e^{i\xi\cdot(x-x')}f(x'+iy)\dd x'\dd \xi.
  \]
  we deform the contour of integration over $x'$ by translating by $i(w-y)$ where $w\in S^{n-1}$; we obtain, for every $w\in S^{n-1}$,
  \[
    f(x+iy)=\frac{1}{(2\pi)^n}\int_{\R^n\times
      \R^n}e^{i\xi\cdot(x+iy-(x'+iw))}f(x'+iw)\dd x'\dd \xi.
  \]
  We now average this over $w\in S^{n-1}$ with respect to the probability density
  \[
    w\mapsto \frac{e^{-2w\cdot \xi}}{m_n(2\xi)},
  \]
  and obtain exactly
  \[
    f(x+iy)=\frac{1}{(2\pi)^n}\int_{\R^n\times \R^n\times
      B(0,1)}e^{i\xi\cdot(x+iy-(x'-iw))}(m_n(2\xi))^{-1}f(x'+iw)\dd \xi\dd
    x' \dd w.
  \]
  This concludes the proof.
\end{proof}

\begin{prop}\label{prop:loc-anal-hypoell-Grauert}
   $\overline{\partial}_b$ is locally analytic
  hypoelliptic on $\R^n\times S^{n-1}$, by which we mean the following
  property: if $u\in L^2(\R^n\times
  S^{n-1})$ and $\overline{\partial}_bu$ is real-analytic on
  an open set $V$, then $(1-S)u$ is real-analytic on $V$.
\end{prop}
\begin{proof}
  It suffices to prove the claim in the case where $V$ is bounded. Let $V_1\Subset V$ be a smaller open set and let
  $(\chi_N)_{N\in \N}$ be a family of smooth functions supported on
  $V$, such that $\chi_N=1$ on $V_1$, and
  \[
    \exists \rho, \forall j\leq N, |\nabla^j\chi_N|\leq (\rho N)^j.
  \]
  Such cutoffs will be used systematically in Section
  \ref{sec:ampl-analyt-micr}, and we postpone until Proposition
  \ref{prop:Ehrenpreis} the proof of their existence. Let us
  recall from the Leibniz formula that 
  \[
    u\in C^{\omega}(V)\Rightarrow \exists C,R, \forall n\in \N, \|\nabla^n(\chi_nu)\|_{L^2}\leq
    C(Rn)^n\Rightarrow u\in C^{\omega}(V_1) .
  \]
  Since $V$ is relatively compact, we have more generally the
  following characterisation: let $(Z_1,\ldots,Z_k)$ be a finite family of
  real-analytic vector fields on a neighbourhood of $V$, which span
  $TV$ at each point. Then
  \[
    u\in C^{\omega}(V)\Rightarrow \exists C,R,\forall n\in \N, \forall \mu\in
      \llbracket 1,k\rrbracket^n,\|Z_{\mu_1}Z_{\mu_2}\cdots
    Z_{\mu_n}(\chi_nu)\|_{L^2}\leq C(Rn)^n\Rightarrow u\in C^{\omega}(V_1).
    \]
    The following vector fields are particularly suited to our problem:
    \begin{align*}
      X_j&=\frac{\partial}{\partial x_j}\qquad  & 1\leq j\leq n\\
      Y_{j,k}&=\omega_k\frac{\partial}{\partial
               \omega_j}-\omega_j\frac{\partial}{\partial
               \omega_k}-x_k\frac{\partial}{\partial
               x_j}+x_j\frac{\partial}{\partial x_k} \qquad &
                                                                     1\leq
                                                                     j,k\leq
                                                                     n.
    \end{align*}
    Indeed, these vector fields span $T(\R^n\times S^{n-1})$, and they commute
    with $\overline{\partial}$; hence they also commute with the
    Szeg\H{o} projector $S$. Now, given $u\in L^2$ such that
    $\overline{\partial}u$ is real-analytic on $V$, one has, for every
    $n\in \N$ and every multi-index $\mu\in \llbracket
    1,k\rrbracket^n$,
    \begin{align*}
      \|Z_{\mu_1}\cdots
      Z_{\mu_n}(1-S)(\chi_nu)\|_{L^2}&=\|(1-S)Z_{\mu_1}\cdots
                                       Z_{\mu_n}(\chi_nu)\|_{L^2}\\
      &\leq \|\overline{\partial}Z_{\mu_1}\cdots
        Z_{\mu_n}(\chi_nu)\|_{L^2}\\
      &=\|Z_{\mu_1}\cdots
        Z_{\mu_n}\overline{\partial}(\chi_nu)\|_{L^2};
    \end{align*}
where we used Kohn's hypoellipticity result $\|(1-S)u\|_{L^2}\leq
\|\overline{\partial}w\|_{L^2}$ \cite{kohn_harmonic_1963}.

To conclude, since $\chi_n=1$ on $V_1$ and the kernel of $S$ is analytic away
from the diagonal by Proposition \ref{prop:flat_Grauert}, then for every $V_2\Subset
V_1$, the sequence $((1-S)(1-\chi_n)u)_{n\in \N}$ is bounded in some fixed
space of analytic functions on $V_2$. Therefore we can conclude that
\[
  \exists C,R, \forall \mu\in \llbracket 1,k\rrbracket^n,\|Z_{\mu_1}\cdots
  Z_{\mu_n}\overline{\partial}(1-S)u\|_{L^2(V_2)}\leq C(Rn)^n,
\]
and the proof is complete: $(1-S)u$ is real-analytic on $V_2$.
\end{proof}

\section{Amplitudes and phases}
\label{sec:analytic-symbols}
This section serves as a technical preliminary for the definition and
study of Fourier Integral Operators in Section
\ref{sec:four-integr-oper}. We study in detail spaces of amplitudes
adapted to the real-analytic regularity, and we state and prove a stationary
phase result for these spaces of amplitudes.
\subsection{Amplitudes in analytic microlocal analysis}
\label{sec:ampl-analyt-micr}

This section takes inspiration from \cite{sjostrand_singularites_1982,deleporte_toeplitz_2018,treves_analytic_2022}.

{\bf Convention}: given $U\subset \R^n$ open, $x\in U$, and $u\in
C^j(U,\C)$, we let
\[
  |\nabla^ju(x)|=\left(\sum_{\alpha\in
      \N^n,|\alpha|=j}|\partial^{\alpha}u(x)|^2\right)^{\frac 12}.
  \]

\begin{defn}\label{def:formal-amp}
  Let $\Omega\subset \R^m_x\times \R^n_{\theta}$ be an open cone in
  the second variable, and let
  $d\in \R$. A \emph{(homogeneous classical) formal analytic
    amplitude} of degree $d$ is a sequence $(a_k)$ of real-analytic functions on
  $\Omega$ such that $a_k$ is homogeneous of degree $d-k$, and such
  that there exists $C>0, \rho>0,R>0,m>0$ satisfying the following:
  for every $j,k,\ell\in \N_0$, for every
  $(x,\theta)\in \Omega$,
  \[
    |\nabla^{j}_x\nabla^{k}_{\theta}a_\ell(x,\theta)|\leq
    C\frac{\rho^{j+k}R^\ell(j+k+\ell)!}{(1+j+k+\ell)^m}|\theta|^{d-\ell-k}.
  \]
  For every $m,\rho,R$, the best constant $C$ in this inequality is
  written $\|a\|_{S^{\rho,R}_m(\Omega)}$.
\end{defn}
If $(a_{\ell})_{\ell\in \N}$ is a formal analytic amplitude, then for every
$\ell$, $a_{\ell}$ is a real-analytic function of $(x,\theta)$, which
extends to a fixed size neighbourhood of the real locus, where it
satisfies $\|a_{\ell}(1+|\theta|)^{-d+\ell}\|_{L^{\infty}}\leq
C(R')^{\ell}\ell!$. This turns out to be an equivalent definition of
a formal analytic amplitude. If $\Omega$ is compact with respect to
the first variable,
one-homogeneous analytic changes of
variables in $x,\theta$ preserve
the space of all formal analytic amplitudes, even though they may not
individually preserve the Banach spaces $S^{\rho,R}_m(\Omega)$.

Given a paracompact conical analytic manifold, we may then define
formal analytic amplitudes on $X$ as sequences of real-analytic
functions on $X$ with decreasing homogeneity, which satisfy the
controls of Definition \ref{def:formal-amp} locally in charts.

\begin{defn}\label{def:analytic_amp}
  Let $\Omega\subset \R^m_x\times \R^n_{\theta}$ be an open cone in
  the second variable, and let $d\in \R$. An \emph{analytic amplitude} of degree $d$ on $\Omega$ is a smooth
  function $a$ on $\Omega$ such that  there exists
  $R>0,\rho>0,C>0$ satisfying the following:
  \[
    \forall (x,\theta )\in \Omega,\forall k\leq
    |\theta|/R,\forall j,
    |\nabla_{x}^j\nabla_{\theta}^ka(x,\theta)|\leq
    C\rho^{j+k}(j+k)!|\theta|^{d-k}.
  \]
\end{defn}
An analytic amplitude is real-analytic with respect to $x$, and
extends into a holomorphic function of $x$; an
alternative definition of an analytic amplitude is the estimate
\[\forall k \leq |\theta|/R, |\nabla_{\theta}^ka(x,\theta)|\leq C\rho^kk!|\theta|^{d-k}\]
on a whole neighbourhood of the real locus in $x$. Thus, again, we may speak
of analytic amplitudes on $\Omega\subset X\times \R^n_{\theta}$ when
$X$ is a paracompact analytic manifold.

\begin{rem}\label{rem:name_analytic_ampl}
Despite their names, the analytic amplitudes as in Definition \ref{def:analytic_amp} are \emph{not}
real-analytic functions. The main reason for this
definition is that the Borel summation of a formal analytic amplitude into an
analytic amplitude (Proposition \ref{prop:real_formal_symb}) will involve a lowest term summation, which requires
cut-off functions in the variable $|\theta|$. Trèves calls such functions
``pseudo-analytic'' \cite{treves_analytic_2022} but unfortunately this
term, like many other variations upon the word ``analytic'', already
refer to well-established mathematical objects.
\end{rem}

Analytic amplitudes are stable by
``phasification'' of the variables: denoting $x=(x_1,x_2)$, where
$x_1\in \R^{m_1}$ and $x_2\in \R^{m_2}$, and similarly
$\R^{m_2+n}\ni \theta=(\theta_1,\theta_2)\ni \R^{m_2}\times \R^n$, if some
amplitude $a$ is analytic on $\Omega\subset \R^{m_1+m_2}\times (\R^{n}\setminus \{0\})$, then so is
\[
  (\R^{m_1})\times (\R^{m_2}\times (\R^n\setminus \{0\}))\supset
  \Omega \ni (x_1,\theta)\mapsto a((x_1,\theta_1/|\theta_2|),\theta_2).
\]
The converse, however, is not true. We will say that an amplitude $a$
is \emph{maximally analytic} if, writing $\theta=\lambda\omega$ in
spherical coordinates, the function $((x,\omega),\lambda)\mapsto
a(x,\lambda\omega)$ is an analytic amplitude.

A last immediate consequence of Definition \ref{def:analytic_amp} is that analytic amplitudes belong to the usual
Hörmander symbol classes, even as they are extended holomorphically
with respect to the first variable.

In order to build analytic amplitudes starting from formal analytic amplitudes, the
following construction of cut-off functions will be useful.
\begin{prop}\label{prop:Ehrenpreis} (Ehrenpreis cutoffs) 
  Let $K,L$ be closed sets of $\R^d$ at positive distance from each other. There exists a
  sequence of compactly supported functions
  $(\chi_N)_{N\geq 1}\in C^{\infty}(\R^d,[0,1])$ such that $\chi_N=0$
  on $L$ and $\chi_N=1$ on $K$, and a constant $\rho>0$ such that
  \begin{equation}\label{eq:Ehrenpreis}
    \forall N\in \N, \forall j\in \N, j\leq N\Rightarrow
    |\nabla^j\chi_N|\leq (\rho N)^j.
  \end{equation}
\end{prop}
\begin{proof}
  Let $K\subset K'$ such that
  $K'\cap L=\partial K'\cap K=\emptyset $, and let
  $c=\min(\dist(K',L),\dist(K,\partial K'))>0$.

  Let $\phi\in C^{\infty}(\R^d,[0,+\infty))$ be supported on $B(0,1)$
  and such that $\int_{\R^d}\phi=1$. Given $\alpha>0$, let
  \[\phi_{\alpha}:x\mapsto \alpha^{-d}\phi(\alpha^{-1}x).\]
  The claim follows from the following choice of cut-off functions:
  \[
    \chi_N=\underbrace{\phi_{cN^{-1}}*\ldots
      *\phi_{cN^{-1}}}_{N\text{ times }}*\mathds{1}_{K'}.
  \]
  Indeed, letting $\rho=\|\nabla \phi_c\|_{L^1}$, one has $\|\nabla
  \phi_{cN^{-1}}\|_{L^1}=\rho N$, while $\|\phi_{cN^{-1}}\|_{L^1}=1$ and
  $\|\mathds{1}_{K'}\|_{L^{\infty}}=1$; altogether we find the desired bound.
\end{proof}

\begin{prop}\label{prop:real_formal_symb}%
  Let $\Omega\subset \R^m_x\times \R^n_{\theta}$ be an open cone
  with respect to the second variable. Let $(a_k)_{k\in \N}$ be a degree $d$ formal analytic amplitude on
  $\Omega$. Then there exists a degree $d$ maximally analytic amplitude $a$ on
  $\Omega$, and
  constants $C,\rho$ such that, uniformly on a complex neighbourhood
  of $\Omega$ in the spherical variable,
  for every $N\in \N$,
  \[
    \left|a(x,\theta)-\sum_{k=0}^{N-1}a_k(x,\theta)\right|\leq
    C\rho^NN!|\theta|^{d-N}.
  \]
  
  Any analytic amplitude whose extension satisfies the bound above
  will be called a realisation of the formal amplitude $(a_k)_{k\in
    \N}$. Two extensions of analytic amplitudes realising the same formal amplitude
  have a difference bounded by $e^{-c|\theta|}$, for some $c>0$, on a
  whole conical complex neighbourhood of $\Omega$ with respect to the
  first variable.
\end{prop}
\begin{proof}
  
  Let us first prove the last fact: letting $\widetilde{\Omega}$ be a
  complex neighbourhood with respect to the first variable, suppose that
  \[
    \forall (x,\theta)\in \widetilde{\Omega}, \forall N\in \N,
    |a(x,\theta)|\leq C\rho^NN!|\theta|^{d-N}.
  \]
  Writing $N=\alpha|\theta|$ and applying the Stirling formula gives
  \[
    |(x,\theta)|\leq C|\theta|^d(\alpha\rho/e)^{\alpha|\theta|}
    ;
  \]
  For $\theta$ large enough there is an integer $N$ between $\frac{e|\theta|}{4\rho}$
  and $\frac{e|\theta|}{2\rho}$, and we obtain
  \[
    |a(x,\theta)|\leq C|\theta|^d(1/2)^{\frac{e|\theta|}{4\rho}};
  \]
  hence the claim.

  We now give a direct construction of $a$. 
  Let
    $(\chi_N)_{N\in \N}:[0,+\infty)\to [0,1]$ be a sequence of
    Ehrenpreis cutoffs
    as in Proposition \ref{prop:Ehrenpreis}, equal to $1$ on $[0,1]$
    and to $0$ on
    $\R\setminus [0,2]$.
    
    Given $(a_\ell)_\ell\in S^{\rho,R}_0(\Omega)$, the claim follows
    from the choice
    \[
      a:(x,\theta)\mapsto
      \sum_{\ell=1}^{\infty}a_{\ell}(x,\theta)(1-\chi_{\ell+1}(\tfrac{c|\theta|}{\ell+1})),
    \]where $c$ is chosen small enough depending on $R$.
    
    From now on we pass to spherical coordinates and write $\theta=\lambda\omega,X=(x,\omega)$.
    Let us prove that $a$ is an analytic amplitude. Given
    $(X,\lambda)\in \Omega$, $j\in \N$ and $k\leq \lambda/R$, one has
    \[
      \nabla^j_{X}\partial^k_{\lambda}a(X,\lambda)=\sum_{\ell\leq c\lambda}\nabla^j_{X}\partial^k_{\lambda}(a_{\ell}(X,\lambda))(1-\chi_{\ell+1}(\tfrac{c\lambda}{\ell+1})).
    \]

    Let us first control the terms where no derivatives hit
    $\chi_{\ell+1}$. If $c\leq \frac{1}{4R}$, then
    \begin{align*}
      \left|\sum_{\ell\leq
      c\lambda}(1-\chi_{\ell+1}(\tfrac{c\lambda}{\ell+1}))\nabla^j_{X}\partial^k_{\lambda}a_\ell(X,\lambda)\right| &
      \leq \|a\|_{S^{\rho,R}_0}\sum_{\ell\leq
        c\lambda}\lambda^{d-\ell-k}(j+k+\ell)!\rho^{j+k}R^{\ell}\\
      &\leq
        \|a\|_{S^{\rho,R}_0}(j+k)!(2\rho)^{j+k}\lambda^{d-k}\sum_{\ell\leq
        c\lambda}(2R)^{\ell}\ell!\lambda^{-\ell}\\
      &\leq 2\|a\|_{S^{\rho,R}_0}(j+k)!(2\rho)^{j+k}\lambda^{d-k},
    \end{align*}
    where we used the inequality $(j+k+\ell)!\leq
    2^{j+k+\ell}(j+k)!\ell!$ and the fact that, on $0\leq \ell\leq
    c\lambda$, the family $(2R)^{\ell}\ell!\lambda^{-\ell}$ decreases
    with $\ell$ faster than a geometric sequence of ratio $\frac 12$.

    Now we treat the case where one of the derivatives has hit
    $\chi_{\ell+1}$. The support of $\lambda\mapsto
    \chi_{\ell+1}'(\frac{c\lambda}{\ell+1})$ is included in
    $\frac{c\lambda}{2}\leq \ell+1\leq c\lambda$. If moreover $k\leq
    \epsilon \lambda$ with $\epsilon \leq \frac{c}{2}$, then for every
    $m\leq k$, one has $m\leq \frac{c\lambda}{2}\leq \ell+1$. Therefore
    \begin{multline*}
      \left|\sum_{\ell\leq
      c\lambda}\sum_{m=1}^{k}\binom{k}{m}\partial_{\lambda}^m\chi_{\ell+1}(\tfrac{c\lambda}{\ell+1})\nabla^j_X\partial^{k-m}_{\lambda}a_{\ell}(X,\lambda)\right|\\
      \leq\|a\|_{S^{\rho,R}_0}\sum_{\frac{c\lambda}{2}-1\leq\ell\leq c\lambda-1}\sum_{m=1}^k\binom{k}{m}c^m\rho^{j+k-m}R^\ell(j+k-m+\ell)!\lambda^{d-k-\ell+m}\\
      \leq \|a\|_{S^{\rho,R}_0}(4\rho)^{j+k}(j+k)!\lambda^{d-k}\sum_{\frac{c\lambda}{2}-1\leq\ell\leq c\lambda-1}\sum_{m=1}^kc^m(2\rho)^{-m}(2R)^{\ell}(\ell-m)!\lambda^{m-\ell}.
    \end{multline*}
    Since $m\leq \tfrac{\ell}{2}$, we perform a change of variables and obtain,
    provided $c\leq\frac{1}{4R}$,
    \begin{multline*}
      \left|\sum_{\ell\leq
        c\lambda}\sum_{m=1}^{k}\binom{k}{m}\partial_{\lambda}^m\chi_{\ell+1}(\tfrac{c\lambda}{\ell+1})\nabla^j_X\partial^{k-m}_{\lambda}a_{\ell}(X,\lambda)\right|\\
    \leq
    \|a\|_{S^{\rho,R}_0}k(8\rho^2R)^{j+k}(j+k)!\lambda^{d-k}\sum_{0\leq
      p\leq c\lambda-1}(2R)^p\lambda^{-p}(p+1)!\\
    \leq 
    2\|a\|_{S^{\rho,R}_0}(16\rho^2R)^{j+k}(j+k)!\lambda^{d-k}
  \end{multline*}
  It remains to show that
$a=\sum_{\ell=0}^Na_\ell+O_N(\lambda^{d-N-1})$ with the required growth
in $N\geq
0$. First
\[
  (X,\lambda)\mapsto\sum_{\ell=0}^Na_\ell(X,\lambda)\chi_{\ell+1}(\tfrac{c\lambda}{\ell+1})
\]
has compact support, and is uniformly bounded as follows:
\[\left|\sum_{\ell=0}^Na_\ell(X,\lambda)\chi_{\ell+1}(\tfrac{c\lambda}{\ell+1})\right|
  \leq
  \|a\|_{S^{\rho,R}_0}\sum_{\ell=0}^NR^{\ell}\ell!\lambda^{d-\ell}\chi_{\ell+1}(\tfrac{c\lambda}{\ell+1});\]
for this sum to be nonzero, we need $c\lambda\leq N+1$, moreover
$\ell\mapsto R^{\ell}\ell!\lambda^{d-\ell}$ is log-convex so that
\[\left|\sum_{\ell=0}^Na_\ell(X,\lambda)\chi_{\ell+1}(\tfrac{c\lambda}{\ell+1})\right|
  \leq
  \|a\|_{S^{\rho,R}_0}\lambda^d(N+1)\max(1,R^NN!\lambda^{-N})\1_{\lambda\leq
    \frac{N+1}{c}}.\]
Moreover if $\lambda\leq \frac{N+1}{c}$ then $N!\lambda^{-N}\geq
\left(\frac{c}{e}\right)^N$, so that, all in all, for some $R_1$ large enough,
\[\left|\sum_{\ell=0}^Na_\ell(X,\lambda)\chi_{\ell+1}(\tfrac{c\lambda}{\ell+1})\right|
  \leq
  \|a\|_{S^{\rho,R}_0}R_1^{N+1}(N+1)!\lambda^{d-N-1}\]

It remains to bound
\[
  \left|\sum_{\ell=N+1}^{+\infty}a_\ell(X,\lambda)(1-\chi_{\ell+1}(\tfrac{c\lambda}{\ell+1})\right|\leq \|a\|_{S^{\rho,R}_0}
  \sum_{\ell=N+1}^{c\lambda}R^{\ell}\ell!\lambda^{d-\ell}.
\]
If $c\leq \frac{1}{2R}$, the first term in this sum is
$R^{N+1}(N+1)!\lambda^{d-N-1}$, and the ratio between two consecutive
terms is smaller than $\frac{1}{2}$; thus
\[
\left|\sum_{\ell=N+1}^{+\infty}a_\ell(X,\lambda)\chi_{\ell+1}(\tfrac{c\lambda}{\ell+1})\right|\leq
2\|a\|_{S^{\rho,R}_0}R^{N+1}(N+1)!\lambda^{d-N-1}.
\]
This concludes the proof.
\end{proof}

\begin{rem}\label{rem:Resurgence}
  It is not clear to us which natural conditions one can impose on the
  formal amplitude $a$ to be the expansion at large $|\theta|$ of an
  analytic function; many estimates and considerations in this section
  are much easier in this case. This seems to be an instance of
  \emph{resurgence}. It would be interesting to study if the space of
  such formal amplitudes is stable under the natural Fourier Integral
  Operator manipulations.
\end{rem}

\begin{prop}\label{prop:Cauchy_product}
  Let $(a_j)_{j\in \N}$, $(b_j)_{j\in \N}$ be two formal analytic
  amplitudes on an open cone $\Omega$. Let $a,b$ be respective analytic
  realisations of $(a_j)_{j\in \N},(b_j)_{j\in \N}$. Then $ab$ is an
  analytic amplitude, which realises the Cauchy product
  $((a*b)_j)_{j\in \N}=(\sum_{\ell=0}^ja_{\ell}b_{j-\ell})_{j\in \N}$.
\end{prop}
\begin{proof}
  One can directly check that $ab$ is an analytic amplitude and that
  $((a*b)_j)_{j\in \N}$ is a formal analytic amplitude. Now (counting some terms twice)
  \begin{multline*}
    \left|a(x,\theta)b(x,\theta)-\sum_{k=0}^{N-1}(a*b)_k(x,\theta)\right|\\ \leq
    \left|a(x,\theta)b(x,\theta)-\sum_{j=0}^{\frac{N-1}{2}}\sum_{\ell=0}^{\frac{N-1}{2}}a_j(x,\theta)b_\ell(x,\theta)\right|+\sum_{j+\ell=\frac{N-1}{2}}^{N-1}|a_j(x,\theta)b_{\ell}(x,\theta)|.
  \end{multline*}
  The first term of the right-hand side is equal to
  \[
    a(x,\theta)\left(b(x,\theta)-\sum_{\ell=0}^{\frac{N-1}{2}}b_\ell(x,\theta)\right)+\sum_{\ell=0}^{\frac{N-1}{2}}b_{\ell}(x,\theta)\left(a(x,\theta)-\sum_{j=0}^{\frac{N-1}{2}}a_j(x,\theta)\right),\]
  so that
  \[
    \left|a(x,\theta)b(x,\theta)-\sum_{j=0}^{\frac{N-1}{2}}\sum_{\ell=0}^{\frac{N-1}{2}}a_j(x,\theta)b_\ell(x,\theta)\right|\leq
    C|\theta|^{d_a+d_b-\frac{N-1}{2}}\rho^{\frac{N-1}{2}}(\tfrac{N-1}{2})!;
  \]
  
  To treat the second term, observe that, given $j+\ell\in
  [(N-1)/2,N-1]$,
  \begin{multline}\label{eq:ajbl}
    |a_j(x,\theta)b_{\ell}(x,\theta)|\leq
    C_{a,b}|\theta|^{d_a+d_b-j-\ell}\rho^{j+\ell}(j+\ell)!\\ \leq
    C_{a,b}|\theta|^{d_a+d_b}\max(|\theta|^{-N+1}\rho^{N-1}(N-1)!,|\theta|^{-\frac{N-1}{2}}\rho^{\frac{N-1}{2}}(\tfrac{N-1}{2})!),
  \end{multline}
  where in the last inequality we used the fact that the middle
  quantity is a log-convex function of $j+\ell$. Altogether if
  $|\theta|^{-\frac{N-1}{2}}\rho^{\frac{N-1}{2}}(\tfrac{N-1}{2})!\leq |\theta|^{-N+1}\rho^{N-1}(N-1)!$, the
  proof is complete.

  Similarly, one can write
  \begin{multline*}
    \left|a(x,\theta)b(x,\theta)-\sum_{k=0}^{N-1}(a*b)_k(x,\theta)\right|\\ \leq
    \left|a(x,\theta)b(x,\theta)-\sum_{j=0}^{N-1}\sum_{\ell=0}^{N-1}a_j(x,\theta)b_\ell(x,\theta)\right|+\sum_{j+\ell=N-1}^{2N-2}|a_j(x,\theta)b_{\ell}(x,\theta)|,
  \end{multline*}
  and if $|\theta|^{-2N+2}\rho^{2N-2}(2N-2)!\leq
  |\theta|^{-N+1}\rho^{N-1}(N-1)!$, we can conclude using
  \eqref{eq:ajbl} again.

  The sequence $(v_k)_{k\in \N}=(|\theta|^{-k}\rho^kk!)_{k\in \N}$ is decrasing on
  $\{k\leq |\theta|/\rho\}$ and increasing on $\{k\geq
  |\theta|/\rho\}$. Therefore if both $v_{\frac{N-1}{2}}$ and $v_{2N-2}$
  are greater than $v_{N-1}$, then $N-1\in
  [\frac{|\theta|}{2\rho},2\frac{|\theta|}{\rho}]$. Under these
  circumstances, by the Stirling formula, all of
  $v_{\frac{N-1}{2}},v_{N-1},v_{2N-2}$ can be bounded by
  $e^{-c|\theta|}$ for some $c>0$, and for $\rho'$ large enough,
  $e^{-c|\theta|}\leq (\rho')^{N-1}|\theta|^{-N+1}(N-1)!$; this
  concludes the proof.
\end{proof}

\begin{prop}\label{prop:product_preparation}  
  Let $\Omega_1\subset \R^{n_x}\times \R^{n_{\theta}}\times
  \R^{n_y}$ and $\Omega_2\subset \R^{n_y}\times \R^{n_{\tau}}\times
  \R^{n_z}$ be open cones with respect to their middle variables. Let
  $c>0$ and let
  \[
    \Omega=\{(x,\theta,y,\tau,z)\in \R^{n_x}\times \R^{n_{\theta}}\times
  \R^{n_y}\times \R^{n_{\tau}}\times
  \R^{n_z},(x,\theta,y)\in \Omega_1,(y,\tau,z)\in \Omega_2,
  c|\theta|<|\tau|<c^{-1}|\theta|\}.
\]
Let $a_1:\Omega_1\to \R$ and $a_2:\Omega_2\to \R$ be analytic amplitudes
(real-analytic with respect to their first and third variables). Then
\[
  A:\Omega\ni (x,\theta,y,\tau,z)\mapsto
  a_1(x,\theta,y)a_2(y,\theta,z)
\]
is an analytic amplitude (real-analytic with respect to its first and
last variable). Moreover, if $a_1$ and $a_2$ are respectively
realisations of formal analytic amplitudes $(a_{1;j})_{j\in
  \N},(a_{2;j})_{j\in \N}$, then $A$ is a realisation of the Cauchy
product
\[
  \left(\sum_{\ell=0}^ja_{1;\ell}(x,\theta,y)a_{2;j-\ell}(y,\tau,z)\right)_{j\in
    \N},
\]
the latter being a formal analytic amplitude on $\Omega$.
\end{prop}
\begin{proof}
  The key property is that, on $\Omega$, the quantities
  $|\theta|,|\tau|$ and $(|\tau|^2+|\theta|^2)^{\frac 12}$ are
  comparable. Therefore $(x,\theta,y,\tau,z)\mapsto a_1(x,\theta,y)$
  is an analytic amplitude on $\Omega$, which realises the formal
  analytic amplitude $((x,\theta,y,\tau,z)\mapsto
  a_{1;j}(x,\theta,y))_{j\in \N}$, and a similar property holds for
  $a_2$. By Proposition \ref{prop:Cauchy_product}, the proof is complete.
\end{proof}

We conclude this section with a practical example, related to
Proposition \ref{prop:flat_Grauert}, and which will allow us to prove
that the Szeg\H{o} kernel takes the form of an analytic Fourier
Integral Operator, in the sense of Section \ref{sec:four-integr-oper}.

\begin{prop}\label{prop:m_n_amplitude}Define the following open cone in $\C^d$:
  \[
    \Omega=\{z\in \C^d, |\im(z)|^2<|\re(z)|^2\};
  \]
  recall that the function $r:\Omega\ni z\mapsto
  \sqrt{\sum_{j=1}^dz_j^2}$ is well-defined and holomorphic.

  Then
  \[
    a:z\mapsto e^{-r(z)}\int_{S^{n-1}}e^{z\cdot
      y}\dd y
  \]
  is an analytic amplitude which realises the formal analytic
  amplitude
  \[
    \left(2^{\frac{d-1}{2}}|S^{d-2}|\frac{(d-1)(d-3)\ldots
    (d-2j+1)}{2^jj!}\Gamma(\tfrac{d+1}{2}+j)r(z)^{-\frac{d+1}{2}-j}\right)_{j\in \N}.
  \]
\end{prop}
\begin{proof}
  Let us first study the elementary properties of $a$ on $\Omega$.

  $S^{d-1}$ is a real contour in the complex manifold $\widetilde{S^{d-1}}:=\{\omega\in \C^d,
  \sum_j \omega_j^2=1\}$, on which $\omega\mapsto e^{z\cdot \omega}$ is
  holomorphic. This manifold is preserved by the linear action of
  \[
    \widetilde{O(d)}:=\{M\in M_d(\C), M^TM=I\},
  \]
  whose action is transitive; more generally, for every $\lambda\in
  \C\setminus \R^-$,  $\widetilde{O(d)}$ acts transitively on
  $\{z\in \Omega_1, r(z)^2=\lambda\}$. The function $r$ is itself
  $\widetilde{O(d)}$-invariant.

  To prove that $a$ itself is $\widetilde{O(d)}$-invariant, let $z\in \Omega$. By the previous considerations, there
  exists $M\in \widetilde{O(d)}$ such that $Mz$ is of the form $\alpha
  x$ where $x\in \R^{d}$ and $\re(\alpha)>0$. Since $\widetilde{O(d)}$
  is connected, one can deform the contour of integration into
  $\{M^{-1}\omega, \omega\in S^{d-1}\}$, and then $z\cdot
  M^{-1}\omega=\alpha z\cdot \omega$. Thus $a(z)=a(Mz)$, so $a$ is $\widetilde{O(d)}$-invariant.

  Using this $\widetilde{O(d)}$ invariance, to prove that $a$ is
  the realisation of a formal analytic amplitude, it suffices to prove
  the required estimate for $a(R\alpha e_1)$, where $R\to +\infty$, $|\alpha|=1$,
  $\re(\alpha)>|\im(\alpha)|$, and $e_1$ is the first vector of the base.

  One has
  \[
    a(R\alpha e_1)=|S^{d-2}|\int_{-1}^1 e^{-R\alpha
      (x-1)}(1-x^2)^{\frac{d-1}{2}}\dd x,
  \]
  where $|S^{d-2}|$ is the area of the unit sphere in
  $\R^{d-1}$. Letting $x=1-y$, we rewrite this integral into
  \[
    a(R\alpha
    e_1)=2^{\frac{d-1}{2}}|S_{d-2}|\int_{0}^{2}
    e^{\alpha R y}y^{\frac{d-1}{2}}(1-\tfrac{y}{2})^{\frac{d-1}{2}}\dd y.
  \]
  Let us decompose this integral into two parts. If $y\in [1,2]$,
  then $|e^{-\alpha Ry}|=e^{-\re(\alpha)Ry}\geq
  e^{-\frac{R}{\sqrt{2}}}$ is exponentially small. On $[0,1]$, we will
  replace $(1-\tfrac{y}{2})^{\frac{d-1}{2}}$ by its Taylor series at
  $0$. First
  \[
    \frac{\partial^j}{\partial
      y^j}(1-\tfrac{y}{2})^{\frac{d-1}{2}}=\frac{(d-1)(d-3)\ldots(d-2j+1)}{2^j}(1-\tfrac{y}{2})^{\frac{d-1}{2}-j}.
  \]
  Let $c$ be a small constant (to be determined later on), and let
  $J_R=\lfloor cR\rfloor$. We write, for $y\in [0,1]$,
  \[
    (1-\tfrac{y}{2})^{\frac{d-1}{2}}=\sum_{j=0}^{J_R}\frac{(d-1)(d-3)\ldots(d-2j+1)}{2^jj!}y^j+y^{J_R+1}F_{J_R}(y),
  \]
  where
  \[
    \sup_{y\in [0,1]}|F_{J_R}(y)|\leq
    \frac{(d-1)(d-3)\ldots(d-2J_R+1)}{4^{J_R}J_R!}.
  \]
  Notice that the sequence
  \[
    \left(\frac{(d-1)(d-3)\cdots(d-2j+1)}{2^jj!}\right)_{j\in \N}
  \]
  is bounded; for every $j\geq d/2$, from $j$ to $j+1$ the general
  term is multiplied by $(d-2j-1)(j+1)\in [-1,0]$.
  
 For every $k>0$, $y>0$ and $|\alpha|=1$ with
  $\re(\alpha)>|\im(\alpha)|,$
  \[
    |e^{-\alpha y}y^k|\leq e^{-\re(\alpha)y}y^k\leq
    e^{-\frac{y}{\sqrt{2}}}y^k\leq k^ke^{k(\log(\sqrt{2})-1)}\leq k^k.
  \]
  We apply this inequality at $k=J_R+1+\frac{d-1}{2}$, and obtain
  \[
\left|\int_0^1e^{\alpha R y}y^{\frac{d-1}{2}+J_R}F_{J_R}(y)\dd
  y\right|\leq (cR)^{cR}R^{-cR}\leq Ce^{-c'R}
\]
for some $c'>0$, if $c$ is small enough. It remains
\[
  2^{\frac{d-1}{2}}|S^{d-2}|\sum_{j=0}^{J_R}\frac{(d-1)(d-3)\ldots(d-2j+1)}{2^jj!}\int_0^1e^{\alpha
    R y}y^{\frac{d-1}{2}+j}\dd y.
\]
And to this end we use that, for every $k\leq \frac{R}{\sqrt{2}}$,
\[
  \sup_{y\in [1,+\infty)}|e^{-\alpha Ry}y^k|\leq e^{-R/\sqrt{2}};
\]
applied at $k=j+\frac{d-1}{2}+2$, we obtain
\[
  \left|\sum_0^{J_R}\frac{(d-1)(d-3)\ldots(d-2j+1)}{2^jj!}\int_1^{+\infty}
    e^{\alpha
      R y}y^{\frac{d-1}{2}+j}\dd y\right|\leq Ce^{-c'R}.
\]
Now we can conclude:
\[
  \left|a(R\alpha
  e_1)-\sum_{j=0}^{J_R}2^{\frac{d-1}{2}}|S^{d-2}|\frac{(d-1)(d-3)\ldots
    (d-2j+1)}{2^jj!}\Gamma(\tfrac{d+1}{2}+j)(R\alpha)^{-\frac{d+1}{2}-j}\right|\leq
Ce^{-c'R}
\]
and by $\widetilde{O(d)}$-invariance, for every $z\in \Omega$,
\[
  \left|a(z)-\sum_{j=0}^{J_R}2^{\frac{d-1}{2}}|S^{d-2}|\frac{(d-1)(d-3)\ldots
    (d-2j+1)}{2^jj!}\Gamma(\tfrac{d+1}{2}+j)r(z)^{-\frac{d+1}{2}-j}\right|\leq
Ce^{-c'|z|}.
\]
Using the Cauchy formula, we deduce a similar formula for all
derivatives of $a$, which concludes the proof that $a$ satisfies the
derivative bounds of an analytic amplitude.
\end{proof}

\subsection{Stationary phase}
\label{sec:stationary-phase}
The goal of this section is to obtain a stationary phase theorem
adapted to our notion of analytic amplitudes. We will have to deal
with complex-valued phase functions, and our first step is a study of
the contour deformations necessary to reach the critical points.

\begin{prop}\label{prop:contour}
  Let $\mathcal{K}\Subset \R^n_x\times \R^m_y$. Let $\phi$ be a
  real-analytic function from an open neighbourhood of $\mathcal{K}$
  to $\{z\in \C,{\rm Im}(z)\geq 0\}$. Suppose that $\nabla^2_y\phi$,
  the matrix of second derivatives with respect to the second set of
  variables, is invertible everywhere on $\mathcal{K}$ and that ${\rm
    Im}(\phi)>0$ on $\partial K$.

  Then there exists a constant $c>0$ and a covering of $\mathcal{K}$
  by two open sets $U_1$ and $U_2$ of $\R^{n+m}$ such that the
  following is true.
  \begin{enumerate}
  \item There exists a real-analytic map
    $\kappa_1:U_1\to \R^n_x\times \C^m_y$ with the following
    properties:
    \begin{itemize}
    \item $\kappa_1$ acts as the identity on the first variable.
    \item $\kappa_1$ is a diffeomorphism between $U_1$ and its image
      $\Gamma_1\subset \R^n_x\times \{|{\rm Im}(y)|<c\}$.
    \item $\Gamma_1$ is a totally real submanifold of
      $\C^n\times \C^m$.
    \item ${\rm Im}(\phi)\geq c$ on $\partial \Gamma_1$.
    \item $\kappa_1$ is isotopic to the identity among the maps
      satisfying the four previous properties. (We will henceforth
      call this isotopy a \emph{contour deformation}, and call
      $\Gamma_1$ a \emph{contour}.)
    \item For every $x\in \R^n$, in each connected component of
      $\Gamma_1(y):=\{y\in \C^m, (x,y)\in \Gamma_1\}$, the holomorphic extension
      $\widetilde{\phi}$ of $\phi$ admits exactly one critical point $y^*$ with respect to
      $x$ and the real part of $\widetilde{\phi}$ is constant.
    \item $\nabla^2_y({\rm Im}\widetilde{\phi}\circ \kappa_1)\geq c$, and the
      phase, at the critical points, has non-negative imaginary part.
    \end{itemize}
  \item There exists a contour $\Gamma_2$ deforming $U_2$, on which
    ${\rm Im}(\widetilde{\phi})\geq c$.
  \end{enumerate}
\end{prop}
\begin{proof}
  Let $V$ be a neighbourhood of $\mathcal{K}$ on which $\phi$ is
  defined and $\det\nabla^2_y\phi$ is bounded away from $0$. Let
  $\epsilon>0$ be small enough. We define
  \[
    U_1=\{(x,y)\in V, |\nabla_y\phi(x,y)|< \epsilon\} \qquad \qquad
    U_2=\{(x,y)\in V, |\nabla_y\phi(x,y)|> \epsilon/2\}.
  \]
  Since $\det\nabla^2_y\phi$ is bounded away from $0$, for every
  $(x,y)\in U_1$, there exists a unique critical point $(x,y^*(c))$ of
  $\phi$ closeby (at distance of order $\epsilon$). Moreover, again
  because of the nondegeneracy, a small neighbourhood of $(x,y^*(x))$
  contains the connected component $V$ of
  $U_1\cap (\{x\}\times \R^m)$ to which $(x,y)$ belongs.

  On this neighbourhood of the critical point, we apply the
  holomorphic Morse lemma: there exists a biholomorphism $\sigma$
  (with real-analytic dependence on $x$) such that
  $\widetilde{\phi}(x,y)=\sigma(x,y)^2+\widetilde{\phi}(x,y^*(x))$.

  We are able to deform $V$ into a neighbourhood of $0$ in
  $\{\sigma\in i\R^n\}$, first by straightening $\sigma(V)$ into
  its tangent space at some point $y_0$, then by translation and
  rotation, since the set of the totally real linear subspaces is open
  and connected.

  On $U_2$, where $|\nabla_y\phi(x,y)|>\epsilon/2$, we deform using
  the ascending gradient flow of ${\rm Im}(\widetilde{\phi})$ on $\C^m$. After a
  time $\epsilon$, the imaginary part of the phase is incremented by a
  quantity bounded below everywhere except in the vicinity of real
  points where $\nabla_y{\rm Re}(\widetilde{\phi})$ is small. This can only happen
  if $\nabla_y{\rm Im}(\widetilde{\phi})$ is large, which only happens if
  ${\rm Im}(\phi)$ is large to begin with. Thus, after this gradient
  flow, ${\rm Im}(\widetilde{\phi})$ is bounded from below everywhere.

  To conclude the proof, it remains to show that, among the critical
  points sufficiently close to $U_1$, the imaginary part of the phase
  is never negative. We proceed by contradiction and suppose that
  there is such a critical point close to a point $P\in U_1$. Up to a
  linear transformation by an element of $GL_m(\R)$, near $P$, the
  second derivative of the phase with respect to the second variable
  $y_1$ is nonzero. We can therefore apply the analytic Morse lemma in
  the variable $y_1$, with all other variables as parameters. This
  transforms $U_1$ into a piece of real curve in $\C$, parametrised by
  the other variables. By hypothesis, the imaginary part of the phase
  is strictly smaller on the image of the real locus by the Morse
  biholomorphism than at the critical point. Hence, this piece of
  curve belongs to either one of the two quarter spaces in
  $\C$ where ${\rm Re}(y_1^2)>0$. Since these quarter spaces have a corner at $0$, the image of
  $U_1$ cannot get close to $0$ without the second derivative of the
  phase along $U_1$ getting large. This is not possible because
  $\mathcal{K}$ is compact.
\end{proof}
The contour deformation above allows us to prove a
stationary phase result. Mind that there is no assumption of
uniqueness of the critical point, so that the result of the stationary
phase is a sum of different contributions coming from the different
critical points.

\begin{prop}\label{prop:stat_phase}
  Let $\Omega_1$ be an open subset of $\R^{n_{x}}\times
  \R^{n_{\theta}}$, conic with respect to the
  second variable and relatively compact, and let $\Omega_2$ be an
  open, relatively compact subset of
  $\R^{n_y}$. Let $\phi:\Omega_1\times \Omega_2\to \C$ be real-analytic and
  one-homogeneous with respect to $\theta$, which extends to a small
  neighbourhood of $\Omega_1\times \Omega_2$. Suppose the following:
  \begin{itemize}
  \item $\phi$ takes values in $\{{\rm Im(z)}\geq 0\}$.
  \item $\nabla^2_y\phi$ is nonsingular.
  \item ${\rm Im}(\phi)>0$ on $\partial(\Omega_1\times \Omega_2)$.
  \end{itemize}
  
  Let $a:\Omega_1\times \Omega_2\to \C$ be such that
  $((x,y),\theta)\mapsto a(x,\theta,y)$
  is an analytic amplitude.

  Then there exists a finite collection of open conic subsets
  $V_1,\ldots,V_N$ of $\Omega$, real-analytic zero-homogeneous functions $y^*_j:V_j\to
  \widetilde{\Omega_2}$, analytic amplitudes $b_j$ on $V_j$, and $\epsilon>0$, such
  that, for every $(x,\theta)$ in a small complex neighbourhood of
  $\Omega$ with respect of the first variable,
  \[
    \int_{\Omega_2}e^{i\phi(x,\theta,y)}a(x,\theta,y)\dd y -
    \sum_{j=1}^N\1_{(x,\theta)\in V_i}e^{i\widetilde{\phi}(x,\theta,y^*_j(x,\theta))}b_j(x,\theta)=O(e^{-\epsilon|\theta|}).
  \]
  Moreover ${\rm Im}(\phi)(x,\theta,y^*_j(x,\theta))\geq 0$ on $V_j$ and
  ${\rm Im}(\phi)(x,\theta,y^*_j(x,\theta))>\epsilon|\theta|$ on $\partial V_j$.

  If $a$ realises a formal analytic amplitude,
  then the $b_j$ realise the formal amplitudes obtained by
  formal stationary phase (and in particular, said formal amplitudes are analytic).
\end{prop}
\begin{proof}
  We first apply the contour deformations of Proposition
  \ref{prop:contour}, which allows us to identify the open sets $V_j$
  and the functions $y^*_j$. Recalling the open sets $U_1$ and $U_2$
  of $\Omega_1\times \Omega_2$, for each $(x,\theta)\in \Omega_1$, each
  connected component of $\{y\in \Omega_2,(x,\theta)\in \Omega_1\}$ will
  correspond to a single critical point of $\phi$ with respect to
  $y$. As $x,\theta$ vary, these connected components never merge or split (as
  $|\det(\nabla^2_y\phi)|$ is bounded from below, critical points are
  far away from each other) but simply appear or disappear, as the
  associated critical point moves too far away from the real locus.

  Therefore, if we define the sets $V_j$'s as the projections onto the
  $(x,\theta)$ variables of the connected components of $U_1$, and
  $y^*_j$ as the corresponding critical points, one has indeed ${\rm
    Im}(\phi)(x,\theta,y_j^*(x,\theta))\geq 0$, and positivity on the
  boundary.

  Let $(x,\theta)\in \Omega_1$; decompose $\Omega_2=U_2(x,\theta)\cup
  \bigcup_{j\in \mathcal{J}(x,\theta)}U_{1;j}(x,\theta)$, where
  $U_2(x,\theta)=\{y\in \Omega_2,(x,\theta,y)\in U_2\}$,
  $\mathcal{J}(x,\theta)=\{1\leq j\leq n, (x,\theta)\in V_j\}$ and
  $U_{1;j}(x,\theta)$ is the connected component of $\{y\in
  \Omega_2,(x,\theta,y)\in U_2\}$ whose complex neighbourhood contains
  $y_j^*(x,\theta)$. Since there exists $\epsilon>0$ such that ${\rm
    Im}(\phi)>\epsilon|\theta|$ on $U_2(x,\theta)$, we are left with
  \[
    \sum_{j\in
      \mathcal{J}(x,\theta)}\int_{U_{1;j}(x,\theta)}e^{i\phi(x,\theta,y)}a(x,\theta,y)\dd
    y.
  \]
  By Proposition \ref{prop:contour}, there exists a homotopy
  $\Gamma_t$ of real-analytic
  contours deforming $U_{1;j}(x,\theta):=\Gamma_{0;j}(x,\theta)$ into some
  $\Gamma_{1;j}(x,\theta)$ containing $y^*_j(x,\theta)$ on which ${\rm Re}(\phi)$ is constant; this homotopy
  satisfies ${\rm Im}(\phi)>\epsilon|\theta|$ on $\partial
  \Gamma_{t;j}(x,\theta)$. 

  Extending $a$ and $\phi$ holomorphically on $y$, Stokes' theorem gives
  \[
    \int_{\Gamma_{0;j}(x,\theta)}e^{i\widetilde{\phi}(x,\theta,y)}a(x,\theta,y)\dd y -
    \int_{\Gamma_{1;j}(x,\theta)}e^{i\widetilde{\phi}(x,\theta,y)}\widetilde{a}(x,\theta,y)\dd
    y=\int_{W_1}e^{i\widetilde{\phi}(x,\theta,y)}\widetilde{a}(x,\theta,y)\dd
    y;
  \]
  here, on $W_1=\bigcup_{t\in [0,1]}\partial \Gamma_{t;j}(x,\theta)$, one
  has ${\rm Im}(\widetilde{\phi})>\epsilon |\theta|$. Hence,
  the difference between the integrals on the two contours is
  $O(e^{-c|\theta|})$.

  Applying the Morse lemma on $\widetilde{\phi}$ on $\Gamma_{1;j}$, and
  an analytic change of variables on $y$, it remains to
  prove that, when $a$ is an analytic amplitude of $((x,y),\theta)$, then so is
  \begin{equation}\label{eq:std_phase}
    (x,\theta)\mapsto \int_{B(0,1)}e^{-|\theta|y^2}a(x,\theta,y)\dd y,
  \end{equation}
  and that this stationary phase commutes with the realisation of
  formal analytic amplitudes.
  
  We first show a result ``at fixed order'', then re-sum the
  obtained amplitude in the spirit of Proposition
  \ref{prop:real_formal_symb}.

  Let $N\in \N$ and let $u\in C^N(B_{\R^d}(0,1),\R)$. The Taylor expansion of
  $u$ at $0$ is
  \[
    u(y)=\sum_{|\mu|<N}\frac{\partial^{\mu}u(0)y^{\mu}}{\mu!}+R(y),
  \]
  where
  \[
    |R(y)|\leq \frac{\rho_d^N|y|^N}{N!}\|\nabla^Nu\|_{L^{\infty}};
  \]
  here $\rho_d$ depends only on the dimension $d$.

  Since $\frac{|y|^N}{N!}\leq e^{|y|}$, we obtain
  \[
    \left|    \int_{B(0,1)}e^{-\lambda y^2}R(y)\dd y\right|\leq
    C_d\rho_d^N\lambda^{-\frac{N}{2}}(\tfrac N2)!\frac{\|\nabla^Nu\|_{L^{\infty}}}{N!}.
  \]
  Noting that, on $\R^+$, the function $y\mapsto y^Ne^{-y}$ reaches
  its maximum at $y=N$, then as long as $\frac{N+d+1}{2}<\lambda$,
  \[
    y\notin B(0,1)\rightarrow |y^{\mu}e^{-\lambda y^2}|\leq
    e^{-\lambda}|y|^{-d-1};
  \]
  hence
  \[
    \left|\int_{\R^d\setminus B(0,1)}e^{-\lambda
        y^2}\sum_{|\mu|<N}\frac{\partial^\mu u(0)y^{\mu}}{\mu!}\dd
      y\right|\leq
    C_de^{-\lambda}\sum_{j=0}^{N-1}\frac{\rho_d^j\|\nabla^ju\|_{L^{\infty}}}{j!}.
  \]
  It remains to compute the explicit integral
  \[
    \int_{\R^d}e^{-\lambda y^2}\sum_{|\mu|<N}\frac{\partial^\mu
      u(0)}{\mu!}y^{\mu}\dd y=\pi^{\frac d2}\lambda^{-\frac d2}\sum_{j<\frac{N}{2}}\frac{\Delta^ju(0)}{(4\lambda)^jj!}.
  \]
  Altogether we have proved, for some $\rho_d>0$, for every $N<2\lambda-d-1$, for every $u\in
  C^N$,
  \[
    \left|\int e^{-\lambda y^2}u(y)\dd y-\pi^{\frac d2}\lambda^{-\frac d2}\sum_{j<
        \frac{N}{2}}\frac{\Delta^ju(0)}{(4\lambda)^jj!}\right|\leq
    C_de^{-\lambda}\sum_{j=0}^{N-1}\frac{\rho_d^j\|\nabla^ju\|_{L^{\infty}}}{j!}+C_d\rho_d^N\lambda^{-\frac
      N2}(\tfrac N2)!\frac{\|\nabla^Nu\|_{L^{\infty}}}{N!}.
  \]
  If, on the contrary, $N\geq 2\lambda-d-1$, then
  \[
    \left|\int e^{-\lambda y^2}u(y)\dd y\right|\leq
    C_0\|u\|_{L^{\infty}}\leq C_0\lambda^{-\frac N2}(\tfrac
    N2)!\rho_d^{\frac N2}\|u\|_{L^{\infty}}
  \]
  and, for $j<\frac N2$,
  \[
    \lambda^{-\frac d2}\frac{\Delta^ju(0)}{(4\lambda)^jj!}\leq
    \lambda^{-\frac N2}(\tfrac N2)!\rho_d^{\frac N2}\frac{\|\nabla^{2j}u\|_{L^{\infty}}}{(2j)!}.
  \]
  Therefore, without condition on $N$ and $\lambda$, for every $u\in C^N$,
    \begin{equation}\label{eq:stat_phase_fixed_order}
    \left|\int e^{-\lambda y^2}u(y)\dd y-\pi^{\frac d2}\lambda^{-\frac d2}\sum_{j<
        \frac{N}{2}}\frac{\Delta^ju(0)}{(4\lambda)^jj!}\right|\leq
    C_d(e^{-\lambda}+\rho_d^N\lambda^{-\frac
      N2}(\tfrac N2)!)\sum_{j=0}^N\frac{\|\nabla^ju\|_{L^{\infty}}}{j!}.
  \end{equation}
  
  Plugging $u:y\mapsto a(x,\lambda \omega,y)$, we obtain the second part of the result.

  It remains to prove that \eqref{eq:std_phase} defines an analytic amplitude. The derivatives
  on $x$ and on the spherical variables of $\theta$ play no role, and we are left with proving that, for some
  $\epsilon>0$ and $\rho>0$,
  \[
    \forall k\leq \epsilon \lambda, \left|\frac{\partial^k}{\partial
        \lambda^k}\int_{B(0,1)}e^{-\lambda y^2}a(x,y,\lambda)\dd y
    \right|\leq \rho^{k}k!\lambda^{d-\frac{n_y}{2}-k}.
  \]
  One term of the binomial sum looks like
  \[
    \binom{k}{j}\int_{B(0,1)}e^{-\lambda
      y^2}y^{2j}\partial_\lambda^{k-j}a(x,y,\lambda)\dd y
  \]
  and, applying \eqref{eq:stat_phase_fixed_order} at order $N=2j+n_y$, we obtain
  \begin{multline*}
    \left|\int_{B(0,1)}e^{-\lambda
      y^2}y^{2j}\partial_\lambda^{k-j}a(x,y,\lambda)\dd y\right| \leq
  C_d\left(e^{-\lambda}\sum_{\ell=0}^{N-1}\frac{\rho_d^\ell\|\nabla^\ell_y\partial^{k-j}_{\lambda}(y^{2j}a)\|_{L^{\infty}}}{\ell!}+\right.\\
      +\left.\lambda^{-\frac{n_y}{2}}\sum_{\ell=j}^{j+\frac{n_y}{2}}\frac{\Delta^\ell_y(y^{2j}\partial_{\lambda}^{k-j}a)(0)}{(4\lambda)^\ell \ell!}+\rho_d^{2j+n_y}\lambda^{-\frac{n_y}{2}-j}\frac{\|\nabla^{2j+n_y}_yy^{2j}\partial^{k-j}_{\lambda}a\|_{L^{\infty}}}{(2j+n_y)!}\right).
\end{multline*}
As long as $\frac{k}{\lambda}$ is small enough, if $a$ is an analytic
amplitude, then this satisfies the required claim.
\end{proof}

\subsection{Positive nondegenerate phases and their Lagrangians}
\label{sec:posit-nond-phas}

In this short subsection, we avail ourselves of a few geometrical
results concerning what will be the oscillating phases of the Fourier
Integral Operators considered in this article. We will directly deal
with \emph{complex-valued} phase functions; real-analytic regularity
makes their treatment quite transparent, and every object will be
defined in a complex-geometric setting. For instance, ``Lagrangians''
will always be complex submanifolds (of a larger complex symplectic
space) on which the symplectic form
vanishes.

  We will only use the relatively standard notion of non-degenerate
  phase functions. 
  \begin{defn}\label{def:good_phase}
    A (strictly) \emph{positive non-degenerate} phase function
    $\phi:U_x\times V_{\theta}\times W_y\to \C$, where
    $U\subset \R^{n_x}$, $W\subset \R^{n_y}$ are relatively compact
    open sets and $V\subset \R^{n_{\theta}}$ is an open cone,
    is a one-homogeneous function of $\theta$ such that, with
    $\Sigma_{\phi}=\{(x,\theta,y)\in \widetilde{U}\times
    \widetilde{V}\times \widetilde{W},
    \nabla_{\theta}\widetilde{\phi}=0\}$, the following conditions
    apply:
    \begin{enumerate}
    \item The gradients $\nabla_x\widetilde{\phi}$ and $\nabla_y\widetilde{\phi}$, restricted
      to $\Sigma_{\phi}$, do not vanish.
    \item $\nabla_{\theta}\widetilde{\phi}$ is a defining function for
      $\Sigma_{\phi}$ (that is to say,
      $\nabla_{x,\theta,y}\nabla_{\theta}\widetilde{\phi}$ has rank
      $n_{\theta}$ and its columns span the conormal bundle of
      $\Sigma_{\phi}$).
    \item There exists $c>0$ such that, on the real set $U\times V\times W$, one has
      ${\rm Im}(\phi)\geq c|\theta||\nabla_{\theta}\phi|^2$.
    \end{enumerate}
  \end{defn}

\begin{prop}\label{prop:lagr}
  Let $\phi$ be a positive non-degenerate phase function. Then
  $\Sigma_{\phi}$ is a real-analytic conic submanifold diffeomorphic
  to the conical Lagrangian
  \[\Lambda_{\phi}=\{(x,\dd_x\widetilde\phi(x,\theta,y),y,-\dd_y\widetilde\phi(x,\theta,y)),(x,\theta,y)\in
    \Sigma_{\phi}\}\subset (T^*\C^n\setminus \{0\})^2,\]called the
  canonical relation of $\phi$.
\end{prop}
\begin{proof}
  This is an elementary and well-known fact, whose proof is the same
  whenever $\phi$ is a real-valued phase function, a $C^{\infty}$
  complex-valued phase function with almost analytic extensions, or a
  real-analytic complex-valued phase function. See for instance
  \cite{hormander_analysis_1985}, Proposition 25.4.4.
\end{proof}
Definition \ref{def:good_phase} is independent of a real-analytic
change of variables in $\theta$, as long as it has a real-analytic
dependence in $x$ and $y$. This motivates the next definition.
\begin{defn}\label{def:good_contour}
  Given a real-analytic function $\phi:U\times V\times W\to \C$
  satisfying the first two conditions of Definition
  \ref{def:good_phase}, a totally real, conical submanifold of
  $\widetilde{V}$ with analytic dependence in $x$ and $y$ will be
  called a \emph{good contour} if $\widetilde{\phi}$, restricted to
  this contour, satisfies the third condition.
\end{defn}

\begin{defn}\label{def:composition}
  Two Lagrangians $\Lambda_1\subset (T^*X\setminus \{0\})\times (T^*Y\setminus \{0\})$ and
  $\Lambda_2\subset (T^*Y\setminus \{0\})\times (T^*Z\setminus \{0\})$
  are said to have \emph{transverse composition} at a point of
  $\Lambda_1\times \Lambda_2\cap [(T^*X\setminus \{0\})\times
  \Delta(T^*Y\setminus \{0\})\times (T^*Z\setminus \{0\})]$ if the
  intersection of these two submanifolds is transverse at this
  point.
\end{defn}

\begin{prop}\label{prop:compo_phase}
  Let $\phi_1:\Omega_1\to \C$ and $\phi_2:\Omega_2$ be positive nondegenerate phase functions
  whose Lagrangians have transverse composition. Then for every
  $\epsilon>0$, 
  $\Phi: (x,z,\theta,y,\tau)\mapsto \phi_1(x,\theta,y)+\phi_2(y,\tau,z)$ is a positive
  nondegenerate phase function on
  \[
    \{(x,\theta,y)\in \Omega_1,(y,\tau,z)\in
    \Omega_2,\epsilon|\theta|<|\tau|<\epsilon^{-1}|\theta|\},
  \]up to a small contour homotopy in $y$. Its Lagrangian is
  \[
    \Lambda_1\circ \Lambda_2:=\{(x,\xi,z,\zeta)\in T^*X\setminus
    \{0\}\times T^*Z\setminus \{0\},\exists (y,\eta)\in T^*Y\setminus
    \{0\}, (x,\xi,y,\eta)\in \Lambda_1,(y,\eta,z,\zeta)\in
    \Lambda_2\},\]
  and its real locus is exactly
  \begin{multline*}
    (\Lambda_1\circ \Lambda_2)_\R\\=\{(x,\xi,z,\zeta)\in (T^*X)_{\R}\setminus
    \{0\}\times (T^*Z)_{\R}\setminus \{0\},\exists (y,\eta)\in (T^*Y)_{\R}\setminus
    \{0\}, (x,\xi,y,\eta)\in (\Lambda_1)_{\R},(y,\eta,z,\zeta)\in
    (\Lambda_2)_{\R}\}.
  \end{multline*}
\end{prop}
\begin{proof}
  See \cite{hormander_analysis_1985}, Proposition 25.5.4. Note that
  the notion of positivity used there is more general than ours and we
  have to prove positivity by hand; a
  priori we only have
  \[
    \im(\Phi)\geq
    c|\theta||\nabla_{\theta}\Phi|^2+c|\tau||\nabla_{\tau}\Phi|^2.
  \]
  We now apply a contour deformation on $y$, consisting in the
  positive gradient flow of $(|\theta|^2+|\tau|^2)^{-\frac{1}{2}}\im(\Phi)$. Along this flow, $\im(\Phi)$
  increases at a rate
  $(|\theta|^2+|\tau|^2)^{-\frac{1}{2}}|\nabla_y\Phi|^2$. Therefore, after a small
  time of evolution, we obtain
  \[
    \im(\Phi)\geq
    c|\theta||\nabla_{\theta}\Phi|^2+c|\tau||\nabla_{\tau}\Phi|^2+c(|\theta|^2+|\tau|^2)^{-\frac{1}{2}}|\nabla_y\Phi|^2.
  \]
  This is the desired result since $|\theta|, |\tau|,$ and
  $(|\theta|^2+|\tau|^2)^{\frac 12}$, are comparable.

  The condition on the real locus is an easy consequence of the fact
  that the imaginary part of the phases are always nonpositive, and
  vanish exactly on the real locus of their Lagrangians.
\end{proof}
Phase functions are seldom more than continuous at the
origin (unless they are linear), so that it is necessary to remove
conical neighbourhoods of $\theta=0$ and $\tau=0$ above. When actually
composing Fourier Integral operators in the next section, it will be necessary to deal
with these neighbourhoods in a different way. This will be facilitated
by the fact that $\im(\Phi)+(|\theta|^2+|\tau|^2)^{\frac
  12}|\nabla_y\Phi|^2$ is bounded away from $0$
on this set.

  In the $C^{\infty}$ theory, Fourier Integral Operators are
  completely described by the Lagrangians of their phase functions;
  the usual proof consists in a term-by-term identification or
  construction (see for instance \cite{hormander_analysis_1985},
  Proposition 25.1.5). The analytic theory is a bit subtler, and we
  first describe how to remove ``spurious'' phase variables; we will
  prove in effect that Fourier Integral Operators associated with
  ``complicated'' phase functions can also be written using simpler
  phases, which will be enough for our needs.
  
  \begin{prop}\label{prop:excess}Let $\phi$ be a positive
    nondegenerate phase. Then on every point of $\Sigma_{\phi}$,
    \[
      n_{\theta}-{\rm rank}\nabla^2_{\theta}\phi=n_x+n_y-{\rm
        rank}(\pi:\Lambda_{\phi}\to\C^{n_x+n_y})\]
  \end{prop}
  \begin{proof}
    The proof is the same as in the real-valued case, see
    \cite{treves_analytic_2022}, Proposition 18.4.1.
  \end{proof}

    This proposition interacts very well with the condition of
    positivity in our definition of phase functions. Indeed, on
    $\Sigma_{\phi}^{\R}$, in the directions where
    $\nabla^2_{\theta}\phi$ is non-degenerate, it has positive
    imaginary part; we will therefore be able to apply the stationary
    phase without difficulty.

\section{Fourier Integral Operators}
\label{sec:four-integr-oper}

\subsection{Basic properties}
\label{sec:first-properties}

The main goal of this subsection is to give a suitable definition of
Fourier Integral Operators in a real-analytic setting; they will be
associated with positive phase functions in the sense of Definition
\ref{def:good_phase} and analytic amplitudes in the sense of
Definition \ref{def:analytic_amp}. Among other properties, we wish
Fourier Integral Operators to preserve real-analytic functions and to
behave well under composition.

The main obstacle to a direct and general definition is the fact that
most of our objects are only defined locally, in small neighbourhoods
of the subset where the imaginary part of the phase vanishes. The
zones where the phase has positive imaginary part should not
contribute to the integrals modulo real-analytic functions, but
without the possibility of introducing cut-offs, it is difficult to
prove a priori that it is the case.

We first describe a particular case of oscillatory integrals, in
coordinates, then we proceed by successive generalisations and
patching arguments.

\begin{defn}\label{def:osc_int}
  Let $U\subset \R^{n_x}$ be an open set, and let
  $\Gamma\subset \R^{n_{\theta}}$ be a conical open set.  Let $\phi$
  be a one-homogeneous function on $U\times \Gamma$, real-analytic
  with respect to $x$, and
  such that $\im(\phi)\geq 0$ everywhere.
  
  Then, given an analytic amplitude $a$ on $U\times \Gamma$, we define the
  following distribution on $U$:
  \[
    K_{\Gamma,\phi}(a):x\mapsto \int_{\Gamma}
    e^{i\phi(x,\theta)}a(x,\theta)\dd \theta.
  \]
\end{defn}

\begin{prop}\label{prop:pos_phase_reg_int}
  Let $U$ and $\Gamma$ be two open sets as above, let $\phi$ be a one-homogeneous function on a neighbourhood of
  $U\times \Gamma$, real-analytic with respect to $x$ and such that $\im(\phi)\geq 0$ everywhere. Let
  $a$ be an analytic amplitude on a neighbourhood of $U\times \Gamma$. Then $K_{\Gamma,\phi}(a)$
  is real-analytic on
  \[
    \{x\in U, \forall \theta\in \Gamma,
    \im(\phi)(x,\theta)>c|\theta|\}\] as well as on
  \[
    \{x_0\in U, \exists C,c>0,\forall \theta\in \Gamma, \forall x\in
    {\rm Neigh}(x_0,\widetilde{U}),
    |\widetilde{a}|(x,\theta)<Ce^{-c|\theta|}\}.\]
\end{prop}
\begin{proof}
  Under either of the conditions above, the integrand
  $(x,\theta)\mapsto e^{i\phi(x,\theta)}a(x,\theta)$ extends holomorphically
  into an integrable function of $\theta$ for $x$ close to the real
  locus.
\end{proof}

A particular case of oscillatory integral above consists in
singular integral kernels of operators acting as follows:
\begin{equation}\label{eq:local_FIO}
  x\mapsto \int_{\Gamma} e^{i\phi(x,\theta,y)}a(x,\theta,y)v(y)\dd \theta \dd y,
\end{equation}
that is, the singular integral kernel of a Fourier Integral Operator (with phase
$\phi$ and amplitude $a$). If $\phi$ is a positive
phase function and $a$ is also real-analytic with respect to $y$, then
these integral operators preserve real-analytic functions, as long as
$\im(\phi)>0$ on the boundary of $\Gamma$.

\begin{prop}\label{prop:reg_to_reg}
  Let $U\subset \R^{n_x}$, $\Gamma\subset
  \R^{n_{\theta}+n_y}$ be open sets, with $\Gamma$ conical in its
  first variable and relatively compact in its second variable. Let $\phi$ be
  a positive phase function on a neighbourhood of $U\times \Gamma$ such that
  \[
    \exists c>0, \forall (x,\theta,y)\in U\times \partial \Gamma, \im(\phi)(x,\theta,y)\geq c|\theta|,
  \]
  and let
  $a$ be an analytic amplitude on a neighbourhood of $U\times \Gamma$. Then
  for every $v\in C^{\omega}(\R^{n_y})$, the function
  \[
    x\mapsto \int_{\Gamma} e^{i\phi(x,\theta,y)}a(x,\theta,y)v(y)\dd \theta \dd y
  \]
  is real-analytic on $\overline{U}$.
\end{prop}
\begin{proof}
  Let $v\in C^{\omega}(\R^{n_y})$ and consider
  \[
    u:x\mapsto \int_{\Gamma}e^{i\phi(x,\theta,y)}a(x,\theta,y)v(y)\dd \theta \dd y.
  \]
  We will change the contour of integration in the variable
  $y$. Since $\phi$ is a positive phase function, one has
  \[
    \exists c>0,\forall (x,\theta,y)\in \overline{U}\times \Gamma,
    \im(\phi)+|\nabla_y\phi|^2\geq c|\theta|.
  \]
  We now deform contours by following the flow of
  $\nabla_y\im(\widetilde{\phi})$. After a small time, we obtain a
  contour $\Gamma_1$ on which $\im(\widetilde{\phi})>c|\theta|$. Now,
  by the Stokes formula,
  \[
    u(x)=\int_{\Gamma_1}e^{i\widetilde{\phi}(x,\theta,y)}\widetilde{a}(x,\theta,y)\widetilde{u}(y)\dd
    \theta \dd y + \int_{\Gamma_2}e^{i\widetilde{\phi}(x,\theta,y)}\widetilde{a}(x,\theta,y)\widetilde{u}(y)\dd
    \theta \dd y,
  \]
  where $V_2$ lies within a small neighbourhood of $\partial \Gamma$ in
  $\R^{n_{\theta}}\times \C^{n_y}$, so that $\im(\widetilde{\phi})>c|\theta|$ on
  $V_2$. Thus, by Proposition \ref{prop:pos_phase_reg_int}, $v$ is real-analytic.
\end{proof}

Over the course of this article, the open sets on which phases and
amplitudes are well-defined will seldom be of a product form $U\times
\Gamma$, unless restricted to small neighbourhoods of a point. We have
to patch together expressions of the form \eqref{eq:local_FIO}, which
is possible thanks to Proposition \ref{prop:patching_analytic_func_or_distros}. Before
doing so we introduce a relevant function space.
\begin{defn}\label{def:spaces}
  If $K$ is a compact set of a real-analytic paracompact manifold, we denote by $\mathcal{E}(K)$ the
  space of smooth functions on $K$ in the sense of Whitney
  \cite{whitney_analytic_1934}, and by $\mathcal{O}(K)$ the space of
  functions which extend to real-analytic functions on a neighbourhood
  of $K$.

  Let $U$ be a relatively compact open set of a real-analytic
  paracompact manifold. We denote
  \[
    \mathcal{F}(U)=\mathcal{E}'(\overline{U})/(\mathcal{E}'(\partial
    U)\cup \mathcal{O}(\overline{U})).
  \]
\end{defn}
We refer to Appendix \ref{sec:funct-spac-analyt} for basic properties
of these spaces.

\begin{defn}\label{def:FIO_fixed_fibre}
  Let $X$ and $Y$ be two paracompact analytic manifolds. Let $U\subset
  X$ and $V\subset Y$ be open and relatively compact. Let
  $\Omega\subset X\times \R^{n_{\theta}}\times V$ be open and conical
  in its second variable.

  Let $\phi$ be a positive phase function on a neighbourhood of
  $\Omega$, such that $\im(\phi)\geq 0$ everywhere and
  \begin{equation}\label{eq:boundary_positivity}
    \exists c>0, \forall x\in U, \forall (\theta,y)\in
    \partial(\{(\theta',y')\in \R^{n_{\theta}}\times V,(x,\theta,y)\in \Omega\}), \im(\phi)(x,\theta,y)\geq c|\theta|.
  \end{equation}
  Let $a$ be an analytic amplitude on a neighbourhood of
  $\Omega$, and let $v\in \mathcal{F}(V)$.
  We define $I_{\phi,\Omega,U}(a)v$ as the element of
  $\mathcal{F}(U)$ defined by the following recipe.
  \begin{enumerate}
  \item 
   For every $x_0\in U$, for some small neighbourhood
  $\Gamma_0\times V_0$ of
  \[\{(\theta,y)\in \R^{n_{\theta}}\times V,(x_0,\theta,y)\in
    \Omega,\im(\phi)(x,\theta,y)=0\},\] consider $\chi\in
  C^{\infty}(X,\R)$ be equal to $1$ on a small neighbourhood $U_0$ of $x_0$ and supported on a
  small neighbourhood of $U_0$. Then
  \[
  x\mapsto
  \int_{\Gamma}e^{i\phi(x,\theta,y)}a(x,\theta,y)\chi(x)v(y)\dd \theta
  \dd y
\]
is the action of a ``usual'' Fourier Integral Operator on $v$, and in
particular it belongs to $\mathcal{E}'(X)$ (in fact it is supported
near $U_0$). This element of $\mathcal{E}'(X)$ can then be restricted
to an element of
$\mathcal{E}'(U_0)/\mathcal{E}'(\partial U_0)$, see Proposition
\ref{prop:presheaf}. By Proposition \ref{prop:reg_to_reg}, if $v\in
C^{\omega}(\overline{V_0})$ then this produces an element of $C^{\omega}(\overline{U_0})$.
  
\item Perform the last item on a finite family of open sets covering
  $U$.
\item By Proposition \ref{prop:pos_phase_reg_int}, any two such
  distributions agree on the intersection of their defining supports
  modulo a real-analytic function, and by Proposition
  \ref{prop:reg_to_reg}. Therefore, by Proposition \ref{prop:patching_analytic_func_or_distros}, we can patch these distributions together into
  a uniquely, well-defined element of
  $\mathcal{F}(U)$, which is $0$ if $v\in C^{\omega}(\overline{V})$. In conclusion, we've built
  $I_{\phi,\Omega,U}(a):\mathcal{E}'(\overline{V})/C^{\omega}(\overline{V})\to\mathcal{F}(U)$.
\item One has then, by Proposition \ref{prop:reg_to_reg},
  \[
    SS_a(I_{\Omega,U,\phi}(a)v)\subset \{x\in U, \exists (\theta,y)\in
    \R^{n_{\theta}}\times SS_a(v), (x,\theta,y)\in \Omega,
    \im(\phi)(x,\theta,y)=0\},
  \]
  and in particular, if $v\in \mathcal{E}'(\partial V)$, then
  $SS_a(I_{\Omega,U,\phi}(a)v)=0$. Therefore, we can conclude:
  \[
    I_{\phi,\Omega,U}(a):\mathcal{F}(V)\to \mathcal{F}(U).
\]
\end{enumerate}
\end{defn}

\begin{rem}\label{rem:EtoO}
  It is natural to ask whether this construction works when $v$ is a
  \emph{singularity hyperfunction}, that is, an element of
  $\mathcal{O}'(\overline{V})/(\mathcal{O}'(\partial
  V)+\mathcal{O}(V))$. The answer is yes, but to do so we need to study the
  analytic wave front set of the kernel $K_{\Gamma,\phi}$; we will do
  so in Section \ref{sec:advanced-properties}, after having defined
  the analytic wave front set as the set of analytic singularities
  after conjugation by a Fourier Integral Operator; for the moment we
  only invoke the fact that the smooth wave front set of
  $K_{\Gamma,\phi}$ is known and therefore analytic Fourier Integral
  Operators act nicely on distributions. After making sure that
  Definition \ref{def:FIO_fixed_fibre} makes sense as acting on
  singularity hyperfunctions, the rest of Section \ref{sec:four-integr-oper} will extend to this
  case without any difficulty.
\end{rem}

As in the smooth case, these operators behave well under composition
and stationary phase 
under natural geometric hypotheses, and particularly so when the
analytic amplitudes in question are realisations of formal analytic amplitudes.

\begin{prop}\label{prop:compo_well_def}Let $X,Y,Z$ be paracompact
  real-analytic manifolds, and let $U\subset X,V\subset Y,W\subset Z$ be relatively compact open
  sets. 
  Let $\Omega_1\subset X\times \R^{n_{\theta}}\times V$
  and $\Omega_2\subset Y\times \R^{n_{\tau}}\times W$ be
  open and conic with respect to their middle variables; let
  $\phi_1$ and $\phi_2$ positive
  non-degenerate phase functions on open neighbourhoods of,
  respectively, $\Omega_1$ and $\Omega_2$. Suppose that $(\phi_1,\Omega_1,U)$ and
  $(\phi_2,\Omega_2,V)$ both satisfy
  \eqref{eq:boundary_positivity}. Let
  $a_1:\Omega_1\to \C$ and $a_2:\Omega_2\to \C$ be analytic
  amplitudes.

  Let $\Phi:(x,\theta,y,\tau,z)\mapsto
  \phi_1(x,\theta,y)+\phi_2(y,\tau,z)$ and
  $A:(x,\theta,y,\tau,z)\mapsto a_1(x,\theta,y)a_2(y,\tau,z)$.

  There exists a contour deformation $\Gamma\in \widetilde{Y}$, with
  real-analytic dependence on $x,\theta,\tau,z$, and a
  neighbourhood $\Omega$ of
  \[
    \{(x,\theta,\tau,z)\in
    X\times \R^{n_{\theta}+n_{\tau}}\times Z,y\in
    \Gamma(x,\theta,\tau,z),\nabla_{y,\theta,\tau}\Phi=0\},
  \]
  contained in $\{\epsilon|\theta|<|\tau|<\epsilon^{-1}|\theta|\}$ for
  some $\epsilon>0$,
  such that $(\phi_1+\phi_2,\Omega,U_1)$ satisfies
  \eqref{eq:boundary_positivity}, %
  and
  \[
    I_{\phi_1,\Omega_1,U_1}(a_1)\circ
    I_{\phi_2,\Omega_2,U_2}(a_2)=I_{\Phi,\Omega,U_1}(A)
  \]
  (as maps from $\mathcal{F}(W)$ to $\mathcal{F}(U)$).
   Moreover, if $a_1$ and $a_2$ respectively realise formal analytic
  amplitudes, then there exists a maximally analytic amplitude $A'$ on
  $\Omega$ such that
  \[
    I_{\Phi,\Omega,U_1}(A-A')=0.
  \]

\end{prop}
\begin{proof}
  The distributional kernel of $I_{\phi_1,\Omega_1,U_1}(a_1)\circ
  I_{\phi_2,\Omega_2,U_2}(a_2)$ is
  \[
    (x,z)\mapsto \int
    e^{i\Phi(x,\theta,y,\tau,z)}A(x,\theta,y,\tau,z)\dd y \dd \theta
    \dd \tau.
  \]
  We begin with a contour deformation on $y$, which follows the
  positive gradient flow of $(|\theta|^2+|\tau|^2)^{-\frac
    12}\nabla_y\im(\Phi)$. After a small time, along this contour
  $\Gamma_1$ one
  has
  \[
    \im(\Phi)\geq c|\theta||\nabla_{\theta}\Phi|^2+c|\tau||\nabla_{\tau}\Phi|^2+c(|\theta|^2+|\tau|^2)^{-\frac
      12}|\nabla_y\Phi|^2.
  \]
  Compared to the integral on the original contour, the difference is,
  by Stokes' theorem, an integral of the form
  \[
    \int_{V}  e^{i\widetilde{\Phi}(x,\theta,y,\tau,z)}\widetilde{A}(x,\theta,y,\tau,z)\dd y \dd \theta
    \dd \tau
  \]
  where on $V= \bigcup_{t\in [0,1]}\partial
  \Gamma_t(x,\theta,\tau,z)$, the phase $\widetilde{\Phi}$ has
  positive imaginary part, due to the fact that
  $(\phi_1,\Omega_1,U_1)$ satisfies
  \eqref{eq:boundary_positivity}. Hence, by Definition
  \ref{def:FIO_fixed_fibre}, the difference maps $\mathcal{E}'(\overline{W})$
  into $C^{\omega}(\overline{U})$ (note that, in this application of Definition \ref{def:FIO_fixed_fibre}, $\Phi$ is not a positive
  phase function, because it is singular at $\theta=0$ and $\tau=0$;
  for the same reasons $A$ is not an analytic amplitude).

  We now use Proposition \ref{prop:pos_phase_reg_int} again to remove
  the complement set of a conical neighbourhood $\Omega$ of
  $\{\im(\Phi)=0\}$.

  Let us prove that, on $\Omega$, $\theta$ and $\tau$ are bounded away
  from $0$. Indeed, if $\theta=0$, then $\Phi=\phi_2$, and
  $|\tau||\nabla_\tau\phi_2|^2$ is bounded away from $0$ except in the
  vicinity of $\Sigma_2$, where $|\tau|^{-1}|\nabla_y\phi|^2$ is
  bounded away from $0$. Therefore $\im(\Phi)$ is bounded away from
  $0$ on $\{\theta=0\}$, and on $\{\tau=0\}$ as well, by a symmetrical
  argument. This concludes the proof.

  To conclude, if $a_1$ and $a_2$ realise formal analytic amplitudes,
  then so does $A$, by Proposition \ref{prop:product_preparation};
  letting $A'$ be a maximally analytic realisation of the same
  analytic amplitude, one has $|A'-A|\leq Ce^{-c(|\theta|+|\tau|)}$ so
  that, by Proposition \ref{prop:pos_phase_reg_int}, the difference
  between the associated Fourier Integral Operators maps
  $\mathcal{E}'(\overline{W})$ into
  real-analytic functions.
\end{proof}

\begin{prop}\label{prop:stat_phase_in_FIO}
  Let $X$ and $Y$ be paracompact real-analytic manifolds and let $U,V$
  be relatively compact sets of $X$ and $Y$. Let $\Omega\subset X\times \R^{n_{\theta'}}\times
  \R^{n_{\omega}}\times V$ be open and a relatively compact cone with respect
  to the second variable; let
  $\phi$ be a positive phase function on a neighbourhood of $\Omega$ such that
  $(\phi,\Omega,U)$ satisfies condition
  \eqref{eq:boundary_positivity}. Suppose further that
  $\nabla_\omega^2\phi$ is nonsingular everywhere.

  Let $a:\Omega\to \C$ be such that $((x,y,\omega),\xi)\mapsto
  a(x,\xi,\omega,y)$ is an analytic amplitude.

  Then there exists $J\in \N$ and, for every $1\leq j\leq J$, a conic
  open set $\Omega_j\subset X\times
  \R^{n_{\theta'}}\times V$, a non-degenerate phase
  function $\phi_j$ on a neighbourhood of $\Omega_j$ such that
  $(\phi_j,\Omega_j,U)$ satisifes condition
  \eqref{eq:boundary_positivity}, and an analytic amplitude
  $b_j:\Omega_j\to \C$, such that
  \[
    I_{\phi,\Omega,U}(a)-\sum_{j=1}^JI_{\phi_j,\Omega_j,U}(b_j)
  \]
  maps $\mathcal{E}'(\overline{V})$ into $C^{\omega}(\overline{U})$. Moreover, the union of the
  Lagrangians of $\phi_j$ is the Lagrangian of $\phi$, and if $a$
  realises some formal analytic amplitude $(a_k)_{k\in \N}$, then each
  $b_j$ realises a formal analytic amplitude $(b_{j;k})_{k\in \N}$,
  which can be obtained from $(a_k)_{k\in \N}$ by the formal
  stationary phase.
\end{prop}
\begin{proof}
  Choosing local expressions for $I_{\phi,\Omega,U}(a)$ as in
  Definition \ref{def:FIO_fixed_fibre}, this is a restatement of
  Proposition \ref{prop:stat_phase}.
\end{proof}

Many
natural Fourier Integral Operators make sense of the fibre variable
$\theta$ as belonging to the fibre of a bundle over the base $X\times
Y$ rather than a fixed space $\R^{n_{\theta}}$. Similarly, in
practical applications of Proposition \ref{prop:stat_phase_in_FIO},
the decomposition of the phase variables into the directions where
$\nabla^2\phi$ is non-degenerate and a complement set is seldom
independent on $x,y$.

Both issues are formally taken care of by a one-homogeneous real-analytic change of
variables $(x,\theta,y)\rightsquigarrow
(x,\xi(x,\theta,y),y)$; however, analytic amplitudes do not have
real-analytic dependence on $\theta$, and therefore, these change of
variables destroy real-analyticity with respect to $x$ and
$y$. We can remediate to this problem in the case of realisations of formal analytic amplitudes.

\begin{prop}\label{prop:change_phase_variables}Let $X$ and $Y$ be
  paracompact analytic manifolds, let $U\subset X$ and $V\subset Y$ be
  relatively compact. Let $\Omega\subset
  X\times \R^{n_{\theta}}_{\theta}\times V$ suppose that $\Omega$ is conical
  in its second variable.

  Let $\xi_t:\Omega\to X\times \C^{n_{\theta}}\times V$ be a family of contours
  (they extend into complex-analytic diffeomorphisms, are
  one-homogeneous in the second variable, act as identity on the first and third
  variables, depending smoothly on the parameter $t$) with
  $\xi_0=id$. Let $\phi$ be a phase function on a neighbourhood of $\Omega$ and suppose that $(\phi_t\circ
  \xi_t,\Omega,U)$ satisfies \eqref{eq:boundary_positivity} for all
  $t\in [0,1]$. Let $a$ be the realisation of a formal analytic
  amplitude $(a_k)_{k\in \N}$ on a neighbourhood of $\Omega$.

  Then, any realisation $b$ of $(a_k\circ \xi_1)_{k\in \N}$ is such
  that (as maps from $\mathcal{F}(V)$ to $\mathcal{F}(U)$)
  \[
    I_{\phi\circ \xi_1,\Omega,U}(b\det(D_{\theta}\xi_1))=I_{\phi,\Omega',U}(a).
  \]
\end{prop}
\begin{proof}
  Let $\xi_t$ be a homotopy of complex-analytic diffeomorphisms that are
  one-homogeneous and act as identity on the first and third
  variables, with $\xi_0=id$. Then $(a_k\circ \xi_t)_{k\in \N}$ is, uniformly for $t\in
  [0,1]$, a formal analytic amplitude; as in the proof of Proposition
  \ref{prop:real_formal_symb}, one can realise it into
  \[
    b(t,x,y,\theta)=\sum_{k\in \N}a_k\circ
    \xi_t(x,\theta,y)(1-\chi_k(\tfrac{c\theta}{k}))
  \]
  for some $c>0$ small, where $\chi_k$ is a radial function, supported
  on $B(0,2)$ and equal to $1$ on $B(0,1)$, and such that
  $|\nabla^j\chi_k|\leq (Rk)^j$ whenever $j\leq k$.

  Let now $x_0\in U$. Following Definition
  \ref{def:FIO_fixed_fibre}, let $U_0$ be a small neighbourhood of
  $x_0$ in $U$, let $\Gamma_0,V_0$ an open neighbourhood of
  $\{\im(\phi)(x_0,y,\theta)=0\}$, and let $v\in \mathcal{E}'(V_0)$,
  then for $t$ small, we are considering
  \[
    J(t):x\mapsto \int_{\Gamma_0\times V_0 }
    e^{i\phi(x,\xi_t(x,\theta,y),y)}b(t,x,y,\theta)\det(D_{\theta}\xi_t)v(y)\dd
    \theta\dd y,
  \]
  For some $t_0>0$,
  one has still
  $\im(\phi)\circ \xi_t>0$ on $U_0\times \partial \Gamma_0\times V_0$ for all $t\in [0,t_0]$.

  For every holomorphic function $f$,
  \[
    \frac{\dd}{\dd
      t}(f(x,\xi_t(x,\theta,y),y)\det(D_{\theta}\xi_t))={\rm div}_{\theta}\left(f(x,\xi_t(x,\theta,y))\partial_t\xi_t\right),
    \]
    and therefore, by Stokes' formula,
    \begin{multline}\label{eq:Jprime}
      \frac{\dd J(t)}{\dd t}(x)=\int_{\Gamma_0\times
        V_0}e^{i\phi(x,\xi_t(x,\theta,y),y)}\det(D_{\theta}\xi_t)v(y)\sum_{k\in
        \N}a_k\circ\xi_t(x,\theta,y) \partial_t\xi_t\cdot
      \nabla_{\theta}\chi_k(\tfrac{c\theta}{k})\dd
      \theta \dd y\\+\int_{\partial \Gamma_0\times V_0}e^{i\phi(x,\xi_t(x,\theta,y),y)}\det(D_{\theta}\xi_t)v(y)\sum_{k\in
        \N}b(t,x,y,\theta)\dd
      \theta \dd y.
    \end{multline}
    We can apply Proposition \ref{prop:pos_phase_reg_int} to the
    second term of the right-hand side, since $b$ is an analytic
    amplitude and $\phi\circ\xi_t$ has positive imaginary part on
    $\partial \Gamma_0\times V_0$; therefore this term is an analytic
    function of $x$. As for the first term, let us prove that there
    exists $C>0, \epsilon>0$ such that, 
    uniformly for $(x,y)$ in a complex neighbourhood,
    \[
\left|      \sum_{k\in
        \N}a_k\circ\xi_t(x,\theta,y) \partial_t\xi_t\cdot
      \nabla_{\theta}\chi_k(\tfrac{c\theta}{k})\right|\leq
    Ce^{-\epsilon|\theta|}.\]
  Since $(a_k)_{k\in \N}$ is a formal analytic amplitude, one has
  \[
    |a_k\circ\xi_t(x,\theta,y)|\leq CR^kk!|\theta|^{-k}.
  \]
  Moreover $\nabla_{\theta}\chi_k(\tfrac{c\theta}{k})$ is only nonzero
  whenever $\frac{c|\theta|}{2}\leq k \leq c\theta$. In particular,
  \[
    |a_k\circ\xi_t(x,\theta,y)\nabla_{\theta}(\chi_k(\tfrac{c\theta}{k})|\leq
    C(cR)^{\frac{c|\theta|}{2}}.\]
  Moreover for fixed $\theta$ the number of nonzero terms is of order
  $|\theta|$, so that we finally obtain, whenever $c<\frac{1}{K}$, for
  some $\epsilon>0$, the required estimate.

  In particular, we can apply Proposition \ref{prop:pos_phase_reg_int}
  to the first term of the right-hand-side of \eqref{eq:Jprime} as
  well; to conclude, for $t$ small, $\frac{\dd J(t)}{\dd t}$
  continuously sends $\mathcal{E}'(Y)$ into $\mathcal{O}(U)$.
  Integrating this fact
  on $[0,t_0]$ for $t_0$ small, we obtain the desired identity with $\xi_1$ replaced
  with $\xi_{t_0}$. However the statement is clearly of a local nature, and we can cover
  $[0,1]$ with a finite number of open sets on which we can apply the
  argument above. This concludes the proof.
\end{proof}

From now on we will only consider analytic Fourier Integral Operators
whose amplitudes are realisations of formal analytic amplitudes; by
Proposition \ref{prop:pos_phase_reg_int}, they do not depend on the
realisation, and we will simply pass formal amplitudes as arguments of $I_{\phi,\Omega,U}$.

Assembling the previous facts, we obtain the following general result.
\begin{theorem}\label{prop:FIOs}
  Let $X,Y$ be paracompact real-analytic manifolds. Let $E$ be a real-analytic
  cone bundle over $X\times Y$, and denote the base projection by $\pi=(\pi_X,\pi_Y)$. Let $U\subset X$ and $V\subset Y$ be
  open and relatively compact. Let $\Omega\subset \pi^{-1}(X\times V)$ be open and
  conical, and let $\phi$ be a positive phase function on a
  neighbourhood of $\Omega$. Let $a$ be a formal analytic amplitude on
  a neighbourhood of $\Omega$.

  \begin{enumerate}
    \item 
  Given $U\subset X$ such that $(\phi,\Omega,U)$ satisifes
  \eqref{eq:boundary_positivity}, any choice of local trivialisations
  of $E$ leads (via Definition \ref{def:FIO_fixed_fibre}) to the same, well-defined
  operator $I_{\phi,\Omega,U}(a)$ mapping $\mathcal{F}(V)$ to
  $\mathcal{F}(U)$. When further quotiented by smooth
  functions, $I_{\phi,\Omega,U}(a)$ coincides with the
  usual formal construction of the Fourier Integral operator with
  phase $\phi$ and formal amplitude $a$. Furthermore, for all $v\in
  \mathcal{E}'(Y)$,
  \[
    SS_a(I_{\phi,\Omega,U}(a)v)\subset \{x\in U, \exists \Theta \in \Omega,
    \pi_X (\Theta)=x, \pi_Y(\Theta)\in SS_a(v), \im(\phi(\Theta))=0\}.
  \]
  We call \emph{analytic Fourier Integral operator} such an
  $I_{\phi,\Omega,U}(a)$.
  
    \item
  If $U$ is relatively compact, $\Omega$ is a relatively compact
  open cone, and ${\rm rank} \nabla^2_{\theta}\phi\geq 1$ on a
  neighbourhood of $\Omega$, then
  there exists a finite number of real-analytic vector bundles
  $(E_j)_{1\leq j\leq J}$, such that ${\rm rank}(E_j)={\rm rank}(E)-1$, open conical sets $\Omega_j\subset E_j$,
  positive phase functions $\phi_j$  and formal analytic amplitudes $b_j$
  on neighbourhoods of $\Omega_j$, such that
  \[I_{\phi,\Omega,U}(a)=\sum_{j=1}^JI_{\phi_j,\Omega_j,U}(b_j).\] As
  a formal amplitude, $b_j$ is obtained from $a$ by
  stationary phase.

  \item Let $Z$ be a paracompact real-analytic manifold, $F$ a
  real-analytic cone bundle over $Y\times Z$, $W\subset Z$ open and
  conical, $\Omega'\subset \pi_F^{-1}(Y\times W)$
  open and conical, $\phi'$ a positive phase function on a
  neighbourhood of $\Omega'$  such
  that $(\phi',F,V)$ satisfies \eqref{eq:boundary_positivity}, $a'$ a formal analytic amplitude on a
  neighbourhood of $\Omega'$. Then the composition
  $I_{\phi,\Omega,U}(a)I_{\phi',\Omega',V}(a')$ is an analytic Fourier Integral
  operator (constructed as in Proposition \ref{prop:compo_well_def}).
\end{enumerate}
\end{theorem}

To some extent, we have reproduced the $C^{\infty}$ theory in the
analytic case. One thing still missing, and a crucial property in the
$C^{\infty}$ case, is the fact that Fourier Integral operators only
depend on the positive Lagrangian associated with their phases: given
two phases $\phi_1$ and $\phi_2$ with same Lagrangian, we expect any analytic
Fourier Integral operator associated with $\phi_1$ to be equal to some
analytic Fourier Integral operator associated with $\phi_2$. The proof
in the $C^{\infty}$ case (see \cite{hormander_analysis_1985},
Proposition 25.1.5) relies on a construction of the amplitude
order by order, which is difficult to translate to the analytic case
(essentially, one would need to prove by hand that the constructed
formal amplitude is analytic). We can, however, prove that some
Lagrangians of interest possess a more or less canonical phase, so
that all Fourier Integral operators associated with said Lagrangians
can be rewritten with this specific case. The first of these
Lagrangians is the diagonal in $T^*\C^n$; we will prove in
Proposition \ref{prop:Kuranishi} the \emph{Kuranishi trick}: all
such Fourier Integral operators are pseudodifferential operators in
the sense that they can be rewritten using the usual phase and any good
contour.

The second family of Lagrangians contains the Szeg\H{o} kernel
parametrix on the boundary of any strictly pseudoconvex domain (the
Lagrangian, of course, depends on the domain). These Lagrangians are
idempotent, and Fourier Integral Operators with same Lagrangian are
called ``covariant Toeplitz operators'' and form an algebra. To
describe this algebra in practice, one writes all these Fourier
Integral Operators with the same phase. We can do so, more generally, for Lagrangians whose base projection has corank 1.

\begin{prop}\label{prop:FIO-rank1-phase}
  Let $X,Y$ be paracompact analytic manifolds. Let $\Lambda\subset
  (T^*\widetilde{X}\setminus \{0\})\times (T^*\widetilde{Y}\setminus
  \{0\})$ be a conical, real-analytic, positive
  Lagrangian and suppose that the projection onto the base
  $\pi:\Lambda\to \widetilde{X}\times \widetilde{Y}$ satisfies $\dim
  \ker \dd\pi=1$ everywhere (so that $\Lambda$ is a half-line bundle
  over its projection $\pi(\Lambda)=Z$).

  There exists a positive phase function on a half-line bundle over
  a neighbourhood of $Z$ such that $\Lambda_{\phi}=\Lambda$. For any such
  phase function, and any analytic Fourier Integral Operator with
  phase $\phi$, there exists a formal analytic amplitude $a$ on a neighbourhood
  of $Z$ such that $I_{\phi}(a)$ coincides with the analytic Fourier
  Integral Operator.
\end{prop}
\begin{proof}
  By Theorem
  \ref{prop:FIOs} and Propositions \ref{prop:stat_phase_in_FIO} and \ref{prop:excess},
  possibly after application of stationary phase, an analytic Fourier
  Integral Operator with Lagrangian $\Lambda$ can be written with a
  number of phase variables equal to $1$, which means that the phase
  function is defined on a half-line bundle over a neighbourhood of $Z$. This proves the existence of such a phase.

  Any such phase $\phi$ satisfies $Z=\{\phi=0\}$ and near $Z$ one has
  $\nabla_x\phi\neq 0,\nabla_y\phi\neq 0$. Since $Z$ has codimension
  $1$, this means that $\phi$ is a defining function for $Z$. In particular, given another
  function $\phi_2$ satisfying the same properties, locally after
  identification of the half-line bundles over a neighbourhood of $Z$ on
  which the phases are defined, one has $\phi_2=f\phi$ for some
  non-vanishing function $f$. Thus, locally, Fourier Integral
  Operators with phase $\phi_2$ have integral kernels of the form
  \[
    (x,y)\mapsto \int_0^{+\infty}e^{iuf(x,y)\phi(x,y,1)}b(x,y,u)\dd
    u
  \]
  where we trivialised the half-line bundle and $b$ is a realisation
  of a formal analytic amplitude.

  Now $u\mapsto f(x,y)u$ is a well-defined change of variables, so that by
  Proposition \ref{prop:change_phase_variables}, there exists a
  formal analytic amplitude $a$ such that $I_{\phi}(a)$ is the same
  Fourier Integral Operator. This concludes the proof.
\end{proof}

\subsection{Pseudodifferential operators, FBI transforms, and Toeplitz
operators}
\label{sec:pseud-oper-1}

Pseudodifferential operators are a particular case of Fourier Integral
operators, with singular kernels of the form
\begin{equation}\label{eq:kernel_pseudos}
  U\times V\ni (x,y)\mapsto \int_{\Gamma(x,y)} e^{i(x-y)\cdot (\xi+i\eta(x,\xi,y))}a(x,\xi,y)\dd \xi,
\end{equation}
where $U,V$ are open sets of $\R^n$, $\Gamma(x,y)$ is
a conical contour, with real-analytic dependence in $(x,y)$, and which
is \emph{positive}, meaning that
\begin{equation}\label{eq:good_ctr_pseudo}
  \exists c, \forall (x,y)\in U\times V, \forall \xi\in \Gamma(x,y),\im((x-y)\cdot \xi)>c|\xi||x-y|^2;
\end{equation}
moreover $a$ realises a formal analytic amplitude on a neighbourhood of
$\{(x,\xi,x),x\in U\}$ in $U\times \R^n\times V$.

\begin{rem}\label{rem:pseudo_not_microlocal}
By the previous results, operators defined by \eqref{eq:kernel_pseudos} do not depend on the
realisation, nor do they depend on the choice of $\Gamma$. However, there is an important caveat at this stage: in order for condition
\eqref{eq:boundary_positivity} to hold, we need
$\Gamma(x,x)=\R^n$. Indeed, otherwise the phase vanishes
when $x=y\in U,\xi\in \partial \Gamma(x,x)$. This is, of course, a
severe hindrance because we would like to study the action of
real-analytic amplitudes defined near a point of phase space. This issue be
addressed later on through the FBI transform. For the moment, we make note of this limitation
and define pseudodifferential operators as follows.
\end{rem}

\begin{defn}\label{def:pseudos}
  Let $U\Subset V\subset \R^n$, and let $a$ be a formal analytic amplitude on a
  neighbourhood of $\{(x,\xi,x),(x,\xi)\in U\times \R^n\}$ in
  $U\times \R^n\times V$. The pseudodifferential operator
$Op(a)$ is defined as the operator with distributional kernel
\[
  U\times V\ni (x,y)\mapsto
  \int_{\Gamma(x,y)} e^{i(x-y)\cdot\xi}a^{\Gamma}(x,\xi,y)\dd \xi,
\]
where  $\Gamma(x,y)$ is any
contour satisfying \eqref{eq:good_ctr_pseudo} and $a^{\Gamma}$ is such
that $(x,\xi,y)\mapsto a^{\Gamma}(x,\theta(x,\xi,y),y)$ is a realisation of
$(a_k(x,\theta(x,\xi,y),y))_{k\in \N}$, for
a contour deformation $\theta$ which sends $\R^n$ to $\Gamma(x,y)$.

These operators
preserve analytic functions, and they are uniquely defined as
operators from $\mathcal{F}(V)$ to $\mathcal{F}(U)$.
\end{defn}
The Lagrangian associated with the phase of a pseudodifferential
operator is always the diagonal of $(T^*U\setminus \{0\})^2$. Our first
result is the \emph{Kuranishi trick}, which allows us to prove a
reciprocal statement.

\begin{prop}\label{prop:Kuranishi}
  Let $U\subset V\Subset \R^n$ and let $\pi:E\to \R^n\times \R^n$ be a
  cone bundle and let $\Omega\subset \pi^{-1}(E)$ be an open compact cone.. Let $\phi$
  be a positive phase function on a neighbourhood of  $\Omega$ such that $(\phi,\Omega,U)$
  satisfies \eqref{eq:boundary_positivity} and $\Lambda_{\phi}$ is the
  diagonal of $(T^*U\setminus \{0\})^2$. Let $a$ realise
  a formal analytic amplitude on a neighbourhood of $\Omega$

  There exists a formal analytic amplitude $b$
  on a neighbourhood of $\{(x,\xi,x),(x,\xi)\in U\times \R^n\}$ in $U\times \R^n\times V$, such that
  \[
    I_{\phi,\Omega,U}(a)=Op(b)
  \]
  (as always, in the sense of operators from
  $\mathcal{F}(V)$ to $\mathcal{F}(U)$).
\end{prop}
\begin{proof}
  We first apply Propositions \ref{prop:excess} and
  \ref{prop:stat_phase_in_FIO} to reduce ourselves to the case where
  $N=n$, since the right-hand side of the formula in Proposition
  \ref{prop:excess} is equal to $n$. Note that when doing so, we obtain a finite sum of
  Fourier Integral Operators.

  Since
  $\Lambda_{\phi}$ is the diagonal of $(T^*U\setminus \{0\})^2$, then
  $\phi(x,\theta,y)$ vanishes when $x=y$, and $\theta\mapsto
  \partial_{\theta}\phi(x,\theta,x)\in \R^n$ is a surjective local
  diffeomorphism. Therefore it is a global diffeomorphism, and it
  extends to a neighbourhood of $x=y$ into a map $(x,\theta,y)\mapsto
  \zeta(x,\theta,y)$ which is a global diffeomorphism between the real domain
  of its second variable and a totally real contour of $\C^n$, and
  such that
  \[
    \phi(x,\theta,y)=(x-y)\cdot \zeta(x,\theta,y).
  \]
  The
  contour is $\R^n$ when $x=y$, and it satisfies
  \eqref{eq:good_ctr_pseudo}, because $\phi$ is positive. Removing a
  neighbourhood of $x=y$ (away from which the phase has positive
  imaginary part) and applying
  Proposition \ref{prop:change_phase_variables}, we obtain a
  pseudodifferential operator.
\end{proof}

Given a compact real-analytic manifold $M$, one can equivalently
define analytic pseudodifferential operators on $M$ as
\begin{itemize}
\item gluings of Definition \ref{def:pseudos} in coordinate charts;
\item Fourier Integral Operators on $M$ (or on open subsets of $M$)
  with Lagrangian ${\rm diag}(T^*M)$.
\item Fourier Integral Operators defined with a suitable global phase
  function. An example is given by the function $\phi:T^*M\times M\to
  \C$ defined as follows: $\iota:M\to \R^N$ is an embedding of $M$ in
  Euclidian space of sufficiently large dimension, $\pi_x:\R^n\to
  T_xM$ is the pull-back of the orthogonal projection from $\R^n$ to
  $\iota(T_xM)$, and
  \[
    \phi(x,\xi,y)=\xi(\pi_x(\iota(y)))+i|\xi|\dist(\iota(x),\iota(y))^2.
    \]
  \end{itemize}

  Pseudodifferential operators are not uniquely determined by their
  amplitudes; there is ``one variable too many'' among
  $(x,\xi,y)$. The notion of \emph{symbol} is better suited to their analysis.
\begin{prop}\label{prop:ampl_to_symb_pseudos}
  Let $U\subset \R^n$ open and relatively compact and let $a$ be a formal analytic symbol on a
  neighbourhood $\Omega$ of $\{(x,\xi,x)\in \R^{3n},x\in U\}$. Let $V$ be a
  small neighbourhood of $U$.

  Given $j\in \N$ define
  \[
    b_j:(x,\xi)=\sum_{|\mu|\leq j}\frac{i^{|\mu|}}{\mu!}(\partial_y^{\mu}\cdot
    \partial_{\xi}^{\mu}a_{j-|\mu|})(x,\xi,x).
  \]
  Then $(b_j)$ is a formal analytic amplitude and
  $Op((a_j))=Op((b_j))$. $(b_j)$ is called the (left)
  \emph{total symbol} of $(a_j)$.
\end{prop}
\begin{proof}It follows from a direct computation that $b_j$ is a formal analytic
  amplitude.

  Let $V$ be a small neighbourhood of $U$ such that $(b_j)_{j\in \N}$ is
  well-defined on a complex neighbourhood of $U\times \R^n\times
  V$. Let $\theta:U\times \R^n\times V\to \C^n$ be
  such that $\kappa:(x,\xi,y)\mapsto (x,\theta(x,\xi,y),y)$ extends to a
  biholomorphism with $\kappa(x,\xi,x)=\xi$ and $\phi:(x,\theta,y)\mapsto \phi\circ \kappa^{-1}$ is a
  positive phase function.

  Let $a$ and $b$ be respective realisations of $a_j\circ \kappa^{-1}$
  and $b_j\circ \kappa^{-1}$ as built in Proposition
  \ref{prop:real_formal_symb}. Namely, with $\chi_j\in
  C^{\infty}(\R^n,\R)$ radial and such that
  \begin{align*}
    &\1_{\R^n\setminus B(0,1)}\leq \chi_j\leq \1_{\R^n\setminus
      B(0,\frac 12)}\\
    &\exists \rho_0,\forall N\leq j,
      \|\nabla^N\chi_j\|_{L^{\infty}}\leq (\rho_0j)^N,
  \end{align*}
  we have defined, for $c>0$ small,
  \begin{align*}
    a(x,\theta,y)&=\sum_{j\in
                   \N}a_j(x,\xi(x,\theta,y),y)(1-\chi_{j+1}(\tfrac{c\theta}{j+1}))\\
    b(x,\theta,y)&=\sum_{j\in
                   \N}b_j(x,\xi(x,\theta,y),y)(1-\chi_{j+1}(\tfrac{c\theta}{j+1})),
  \end{align*}
  and then
  \begin{align*}
    Op((a_j))(x,y)&=\int e^{i\phi(x,\theta,y)}a(x,\theta,y)\dd \theta\\
    Op((b_j))(x,y)&=\int e^{i\phi(x,\theta,y)}b(x,\theta,y)\dd \theta
  \end{align*}
  modulo a regularising operator. Define the contour
  $\Gamma(x,y)=\{\xi\in \C^n, \theta(x,\xi,y)\in \R^n\}$, so that
  \begin{align*}
    Op((a_j))(x,y)&=\int_{\Gamma(x,y)} e^{i(x-y)\cdot \xi}a\circ \kappa(x,\xi,y)\dd \xi\\
    Op((b_j))(x,y)&=\int_{\Gamma(x,y)} e^{i(x-y)\cdot \xi}b\circ \kappa (x,\xi,y)\dd \xi.
  \end{align*}
  Now we use the standard manipulation: for every fixed $N\in \N$ and
  smooth function $f$ on $V$
  \[
    \forall x\in U, \forall y \in B(x,\epsilon), f(y)=\sum_{|\mu|<
      N}\frac{\partial^\mu
      f(x)}{\mu!}(y-x)^{\mu}+\sum_{|\mu|=N}R_\mu(f)(x,y)(y-x)^{\mu},
  \]
  where
  \[
    |R_\mu(x,y)|\leq C\frac{\|\partial^{\mu}f\|_{L^{\infty}}}{N!}.
    \]
  Let $\psi_j=1-\chi_j$, and let $c_1\ll c$. To apply the above
  formula at an order $N$ which grows like $c'|\xi|$, we write
  \begin{multline*}
    a\circ \kappa(x,\xi,y)=a\circ \kappa(x,\xi,y)\chi_1(c'\theta(x,\xi,y))\\+\sum_{N\in
      \N}\left(\sum_{|\mu|<N}\frac{\partial^{\mu}_y}{\mu!}a\circ
      \kappa(x,\xi,y)|_{y=x}(y-x)^{\mu}+\sum_{|\mu|=N}R_\mu(a\circ \kappa)(x,\xi,y)(y-x)^{\mu}\right)(\psi_{N+1}(\tfrac{c'\theta(x,\xi,y)}{N+1})-\psi_{N}(\tfrac{c'\theta(x,\xi,y)}{N})).
  \end{multline*}
  This amounts to
  \begin{multline*}
    a\circ \kappa(x,\xi,y)=\sum_{\mu\in
      \N^d}\frac{\partial_y^{\mu}}{\mu!}(a\circ
    \kappa(x,\xi,y))|_{y=x}(y-x)^{\mu}\chi_{|\mu|}(\tfrac{c'\theta(x,\xi,y)}{|\mu|})\\+\sum_{\mu\in
      \N^d}R_{\mu}(a\circ \kappa)(x,\xi,y)(y-x)^{\mu}(\psi_{|\mu|}(\tfrac{c'\theta(x,\xi,y)}{|\mu|})-\psi_{|\mu|-1}(\tfrac{c'\theta(x,\xi,y)}{|\mu|-1})).
  \end{multline*}
  Before proceeding further, we integrate by parts in $Op(a)$ to
  trade the powers of $x-y$ for derivatives in $\xi$:
  \begin{multline*}
    Op(a)(x,y)=\int_{\Gamma(x,y)}\sum_{\mu\in
      \N^d}\partial_{\xi}^{\mu}\left[\frac{i^{|\mu|}\partial_y^{\mu}}{\mu!}(a\circ
    \kappa(x,\xi,y))|_{y=x}\chi_{|\mu|}(\tfrac{c'\theta(x,\xi,y)}{|\mu|})\right]\dd
    \xi\\
    +\int_{\Gamma(x,y)}\sum_{\mu\in
      \N^d}|i|^{|\mu|}\partial_{\xi}^{\mu}\left[R_{\mu}(a\circ
    \kappa)(x,\xi,y)(\psi_{|\mu|}(\tfrac{c'\theta(x,\xi,y)}{|\mu|})-\psi_{|\mu|-1}(\tfrac{c'\theta(x,\xi,y)}{|\mu|-1}))\right]\dd \xi.
  \end{multline*}
  
  Each term in this expansion is real-analytic with respect to $(x,y)$
  in the variables $(x,\theta,y)$. Indeed, given $1\leq k\leq d$,
  \begin{align*}
    \partial_{y_j}(a\circ \kappa)=&\partial_{y_j}\kappa\cdot (\nabla
    a)\circ \kappa=(\partial_{y_j}a)\circ \kappa+\partial_{y_j}\theta\cdot
    (\nabla_{\xi}a)\circ \kappa\\
    \partial_{\xi_j}(a\circ \kappa)=&\partial_{\xi_j}\theta\cdot (\nabla_{\xi}
    a)\circ \kappa
  \end{align*}
  and both $(\nabla_{\xi,y}) \theta)\circ \kappa^{-1}$ and $a$ are
  real-analytic with respect to $(x,y)$. Therefore, by induction, for
  every $\mu\in \N^{d}$
  $[\partial_y^{\mu}(a\circ \kappa)|_{y=x}]\circ \kappa^{-1}$ is
  real-analytic with respect to $(x,y)$. The term $R_{\mu}$ only
  depends on these kind of derivatives, as evidenced by the general remainder
  formula for Taylor series
  \[
    R_{\mu}(f)(x,y)=\frac{|\mu|}{\mu!}\int_0^1\partial^{\mu}f(tx+(1-t)y)(1-t)^{N-1}\dd
    t.
  \]
  Let us now prove that the remainders are small:
  \[
    \left|\sum_{\mu\in
        \N^d}\partial_{\xi}^{\mu}\left[R_{\mu}(a\circ
      \kappa)(x,\xi,y)(\psi_{|\mu|}(\tfrac{c'\theta(x,\xi,y)}{|\mu|})-\psi_{|\mu|-1}(\tfrac{c'\theta(x,\xi,y)}{|\mu|-1}))\right]\right|\leq
    Ce^{-c_1|\theta(x,\xi,y)|},
  \]
  thereby obtaining (by Proposition \ref{prop:reg_to_reg}) that
  this term in $a\circ \kappa$ contributes to $Op(a)$ by a
  regularising operator. Given $\mu\in \N^d$, one has
  \[
    \psi_{|\mu|}(\tfrac{c'\theta}{|\mu|})-\psi_{|\mu|-1}(\tfrac{c'\theta}{|\mu|-1})\neq
    0 \Rightarrow \frac 12 (|\mu|-1)\leq c'|\theta|\leq |\mu|.
  \]
  Now, for some $\rho>0,R>0$,
  \[
    \forall j+k\leq \xi/R, |\nabla^j_y\nabla^k_{\xi}(a\circ \kappa)(x,\xi,y)|\leq
    C\rho^{j+k}(j+k)!|\xi|^{d-k},
  \]
  so that, given $\mu\in \N^d$ and $\theta$ such that $ \frac 12
  (|\mu|-1)\leq c'|\theta|\leq |\mu|$, given $\N^{d}\ni \nu\leq
  \mu$,
  \begin{multline*}
    \frac{\mu!}{\nu!(\mu-\nu)!}\partial_{\xi}^\nu R_{\mu}(a\circ
    \kappa)(x,\xi,y)\partial_{\xi}^{\mu-\nu}(\psi_{|\mu|}(\tfrac{c'\theta(x,\xi,y)}{|\mu|})-\psi_{|\mu|-1}(\tfrac{c'\theta(x,\xi,y)}{|\mu|-1}))\\
    \leq C(a)|\mu|\frac{|\mu|!}{(|\mu|-|\nu|)!}\rho^{2|\mu|}(c')^{|\mu|-|\nu|}|\xi|^{d-|\nu|}.
  \end{multline*}
  With respect to $|\nu|$, the right-hand side is a log-convex
  function; if $c'$ is small enough, at both endpoints $|\nu|=0$ and
  $|\nu|=|\mu|$, it is smaller than $Ce^{c_1|\theta|}$. Therefore
  \[
    \left|\partial_{\xi}^{\mu}\left[R_{\mu}(a\circ
        \kappa)(x,\xi,y)(\psi_{|\mu|}(\tfrac{c'\theta(x,\xi,y)}{|\mu|})-\psi_{|\mu|-1}(\tfrac{c'\theta(x,\xi,y)}{|\mu|-1}))\right]\right|\leq
    Ce^{-c_1|\theta|},
  \]
  and for $\theta$ fixed the number of non-zero terms in the sum over
  $\mu$ is polynomial in $|\theta|$.
  
  Similarly, within
  \[
    \sum_{\mu\in
      \N^d}\partial_{\xi}^{\mu}\left[\frac{i^{|\mu|}\partial_y^{\mu}}{\mu!}(a\circ
      \kappa(x,\xi,y))|_{y=x}\chi_{|\mu|}(\tfrac{c'\theta(x,\xi,y)}{|\mu|})\right],
  \]
  the terms where at least one derivative with respect to $\xi$ hits
  $\chi_{|\mu|}$ are exponentially small: indeed, given $\mu\in \N^d$, the support of $\nabla
  \chi_{|\mu|}(\tfrac{c'\theta(x,\xi,y)}{|\mu|})$ lies within $\{\frac
  12 |\mu|\leq c'|\theta|\leq |\mu|\}$, and the estimates on the
  derivatives are the same as previously.
  
  To conclude the proof, we inject the expression of $a\circ \kappa$
  in
  \[
    \sum_{\mu\in
      \N^d}\partial_{\xi}^{\mu}\frac{i^{|\mu|}\partial_y^{\mu}}{\mu!}(a\circ
    \kappa(x,\xi,y))|_{y=x}\chi_{|\mu|}(\tfrac{c'\theta(x,\xi,y)}{|\mu|})
  \]
  to prove that this is close to $b\circ \kappa$. We obtain
  \[
    \sum_{\mu\in \N^d}\sum_{j\in
      \N}\frac{i^{|\mu|}\partial_{\xi}^{\mu}\partial_y^{\mu}}{\mu!}[a_j(x,\xi,y)\chi_j(\tfrac{c\theta(x,\xi,y)}{j})]|_{y=x}\chi_{|\mu|}(\tfrac{c'\theta(x,\xi,y)}{|\mu|}),
  \]
  and again one can prove that this is exponentially small unless no
  derivative hits $\chi_j$: the support of
  $\nabla\chi_j(\tfrac{c\theta(x,\xi,y)}{j})$ lies in $\{\frac 12
  j\leq c|\theta|\leq j\}$, and one must also have
  $|\mu|<2c'|\theta|$; as long as $c'<\frac{c}{4}$, this ensures $2|\mu|\leq j$, hence for all
  $\nu_y\leq \mu$ and $\nu_{\xi}\leq \mu$,
  \begin{multline*}
    \left|\frac{\mu!}{\nu_y!(\mu-\nu_y)!\nu_{\xi}!(\mu-\nu_{\xi})!}\partial^{\nu_y}_y\partial^{\nu_{\xi}}_\xi
      a_j(x,\xi,y)\partial^{\mu-\nu_y}_y\partial^{\mu-\nu_{\xi}}_\xi\chi_j(\tfrac{c\theta(x,\xi,y)}{j})\right|\\
    \leq \rho^{j+2|\mu|}\frac{j!|\mu|!}{(|\mu|-|\nu_y|)!(|\mu|-|\nu_\xi|)!}c^{2|\mu|-|\nu_y|-|\nu_{\xi}|}|\xi|^{d-j-|\nu_{\xi}|};
  \end{multline*}
  for $c$ small this log-convex function of $|\nu_y|$ and $|\nu_{\xi}|$ is smaller
  than $e^{-c|\theta|}$ at all four endpoints
  $(0,0)$, $(0,|\mu|)$, $(|\mu|,0)$, $(|\mu|,|\mu|)$.

  It remains
  \[
    \sum_{\mu\in \N^d}\sum_{j\in
      \N}\frac{i^{|\mu|}\partial_{\xi}^{\mu}\partial_y^{\mu}a_j}{\mu!}(x,\xi,x)\chi_j(\tfrac{c\theta(x,\xi,x)}{j})\chi_{|\mu|}(\tfrac{c'\theta(x,\xi,y)}{|\mu|}),
  \]
  where we recall that
  \[
    b(x,\xi,y)=\sum_{\mu\in \N^d}\sum_{j\in
      \N}\frac{i^{|\mu|}\partial_{\xi}^{\mu}\partial_y^{\mu}a_j}{\mu!}(x,\xi,x)\chi_{j+|\mu|}(\tfrac{c\theta(x,\xi,x)}{j+|\mu|});
  \]
  therefore we want to prove that the following is exponentially small
  \[
    \sum_{\mu\in \N^d}\sum_{j\in
      \N}\frac{i^{|\mu|}\partial_{\xi}^{\mu}\partial_y^{\mu}a_j}{\mu!}(x,\xi,x)[\chi_j(\tfrac{c\theta(x,\xi,x)}{j})\chi_{|\mu|}(\tfrac{c'\theta(x,\xi,y)}{|\mu|})-\chi_{j+|\mu|}(\tfrac{c\theta(x,\xi,x)}{j+|\mu|})].
  \]
  We first observe that this is zero whenever $j\geq 2c|\theta|$, and
  also whenever $j+|\mu|\leq c|\theta|$ and $|\mu|\leq
  2c'|\theta|$. Therefore we are only interested in situations where
  $j\leq 2c|\theta|$ and either $|\mu|\leq 2c'|\theta|$, $j\geq
  (c-2c')|\theta|$, or $|\mu|\geq 2c'|\theta|$; since $c'\leq \frac
  c4$, this means that we are only considering cases where either $j$
  or $|\mu|$ is larger than $2c'|\theta|$. In this setting,
  \[
    \left|\frac{\partial_{\xi}^{\mu}\partial_y^{\mu}a_j}{\mu!}(x,\xi,x)\right|\leq
    C(a)\mu!j!|\xi|^{d-j-|\mu|}\leq Ce^{-c_1|\theta|}.
  \]
  This concludes the proof.
\end{proof}

\begin{prop}\label{prop:inversion_formal_pseudos}
  Define the Moyal product as follows: given two formal analytic symbols
  $a(x,\xi),b(x,\xi)$, let
  \[
    (a\sharp
    b)_k=\sum_{n=0}^k\sum_{l=0}^{k-n}\sum_{|\beta|=n}\frac{1}{\beta!}\partial^\alpha_xa_l\partial_{\xi}^{\beta}b_{k-n-l}.
  \]
  \begin{enumerate}
    \item The symbol product of two formal analytic symbols is a formal
  analytic symbol; in fact the topology of formal analytic symbols is
  induced by a countable family of Banach spaces of analytic
  symbols that are Banach algebras for the Moyal product.

  \item 
  If both $a$ and $b$ are defined on
  $U\times \R^n\times V$ where $U\Subset V$, then 
  \[
    Op(a)Op(b)-Op(a\sharp b)
  \]
  is regularising.

  \item
  Moreover, for every formal analytic amplitude $a$ such that $a_0$ is bounded
  away from $0$, there exists a formal analytic amplitude $b$ such that
  $(a\sharp b)=1$, and there exists a formal analytic amplitude $c$ such that
  $c\sharp c=a$.
  
\item Define the adjoint $a^*$ of a formal analytic amplitude $a$ as
  \[
    (a^*)_k:(x,\xi)=\sum_{|\mu|\leq k}\frac{i^{|\mu|}}{\mu!}\partial_x^{\mu}\partial_{\xi}^{\mu}\overline{a}_{k-|\mu|}(x,\xi).
  \]
  Then, given a formal analytic amplitude $a$ with $a^*=a$ and $a_0>0$, there
  exists a formal analytic amplitude $b$ such that
  $b^*\sharp b=a$.
\end{enumerate}
\end{prop}
\begin{proof}
  Claims 1 and 3 only depend on the formal calculus and
  are well-known; see \cite{sjostrand_singularites_1982}. A related
  fact, proved in Appendix \ref{sec:algebra-symb-pseud}, is the following:
  with the symbol norms introduced in Definition \ref{def:formal-amp},
  for every $m$ large enough (depending on the dimension) and every
  $R\geq 2^{d+1}\rho^2$,
  \[
    \|a\sharp b\|_{S_m^{\rho,R}}\leq
    12\|a\|_{S_m^{\rho,R}}\|b\|_{S_m^{\frac{\rho}{2},\frac{R}{4}}}.
  \]

  Claim 2 is a particular case of
  Proposition \ref{prop:stat_phase_in_FIO}.

  To prove claim 4, fix a Banach space $S^{\rho,R}_m$ of symbols as in Definition \ref{def:formal-amp}, containing
  $a,\sqrt{a_0},(\sqrt{a_0})^{-1}$ (where we invert with respect to
  the Moyal product) and satisfying the requirements of Proposition \ref{prop:composition-formal}. Letting $b_0=\sqrt{a_0}$, by claim 3,
  \[
    (b_0^*)^{-1}\sharp a\sharp (b_0)^{-1}=1+r
  \]
  where the inverse is taken with respect to the Moyal product and
  where $r$ is a formal analytic symbol of degree $-1$.

  Up to increasing $R$, we can assume
  $\|r\|_{S^{\frac{\rho}{2},\frac{R}{4}_m}}<\frac{1}{12}$. Then the power
  series
  \[
    \sqrt{1+r}:=1+\frac{r}{2}-\frac{r\sharp r}{8}+\frac{r\sharp r
      \sharp r}{16}+\ldots
  \]
  converges to an element $c$ of $S^{\rho,R}_m$ such that $c^*=c$ and
  $c\sharp c = 1+r$. Setting $b=c\sharp b_0$ then concludes the proof.
\end{proof}

Our next move is to translate the pseudodifferential algebra into one
that is more suited to our microlocal needs, as a workaround to Remark
\ref{rem:pseudo_not_microlocal}. 
\begin{defn}\label{def:FBI}
  The FBI transform from $L^2(\R^n)$ to $L^2(\R^n\times S^{n-1})$ is the operator with singular kernel
  \[
    T:\R^n\times S^{n-1}\times \R^n\ni (x,\omega,y)\mapsto \int_0^{+\infty} e^{it((x-y)\cdot
      \omega+i\frac{|x-y|^2}{2})}t^{\frac{n+1}{4}}\dd t.
  \]
\end{defn}
The power of $t$ is chosen such that $T$ is continuous from
$L^2(\R^n)$ to $L^2(\R^n\times S^{n-1})$ and $T^*T$ is a degree $0$
pseudodifferential operator.

\begin{prop}\label{prop:FBI_holom_ext}
  The FBI transform of a compactly supported distribution extends into an holomorphic
  function of $x-i\omega$ on $\R^n\times B(0,1)$.
\end{prop}
\begin{proof}
  It suffices to remark that, writing the phase as
  \[
    t((x-y)\cdot \omega+i\tfrac{|x-y|^2}{2}+i\tfrac{(1-|\omega|^2)}{2}),
  \]
  we obtain a holomorphic function of $x-i\omega$ on $\R^n\times
  B(0,1)$, whose imaginary part is non-negative.
\end{proof}
Before proceeding further, we introduce a helpful notation: given
$z\in \C^n$ we set $z^2=\sum_{j=1}^nz_j^2$; this is the holomorphic
extension of $\R^n\ni x\mapsto |x|^2$. We will reserve $|z|^2$ for the
Hermitian norm of $z$.

\begin{prop}\label{prop:T*T}~
  \begin{enumerate}
    \item 
  The Lagrangian associated with $T$ is
  \[
    \Lambda_T=\{(x,t(\omega+i(x-y)),\omega,t(x-y),y,t(\omega+i(x-y))),t>0,(x-y)\cdot \omega=-i\tfrac{(x-y)^2}{2}.\}
  \]
  \item $\Lambda_{T^*}\circ
  \Lambda_T={\rm diag}(T^*\C^n\setminus \{0\})$, so that
  $T^*T$ is an elliptic pseudodifferential operator.

  \item The
  projection of 
  $\Lambda_{T}\circ\Lambda_{T^*}$ onto the
  base has corank $1$ everywhere. Moreover the intersection of
  $\Lambda_{T}\circ\Lambda_{T^*}$ and the real locus is a half-line
  bundle over $\diag(\R^{n}\times S^{n-1})$, so that the integral
  kernel of any operator of the form $T^*Op(a)T$,
  for a formal analytic amplitude $(a_j)_{j\in \N}$, or more generally of
  any Fourier Integral Operator with the same Lagrangian as $T^*T$, is
  real-analytic away from the diagonal, and near the diagonal takes the form,
  \[
    (x_1,\omega_1,x_2,\omega_2)\mapsto
    \int_0^{+\infty}e^{it\psi(x_1,\omega_1,x_2,\omega_2)}b(x_1,\omega_1,x_2,\omega_2,t)\dd
    t+r(x_1,\omega_1,x_2,\omega_2).
  \]
  In this formula,
  \begin{itemize}
   \item \begin{equation}\label{eq:phase_Toep}
    \psi(x_1,\omega_1,x_2,\omega_2):=i(1+(\tfrac{x_1-i\omega_1-x_2-i\omega_2}{2})^2),
  \end{equation}
  and in particular $\psi$ extends into a holomorphic function of
  $x_1-i\omega_1$ and $x_2+i\omega_2$ on $\R^n\times B(0,1)\times
  \R^n\times B(0,1)$, and $\im(\psi)$ is positive away from
  $\diag(\R^n\times S^{n-1})$.
\item $b$ is the realisation of a formal analytic amplitude.
  \item $r$ and $b$ extend into holomorphic functions of
  $x_1-i\omega_1$ and $x_2+i\omega_2$ on a neighbourhood of
  $\diag(\R^n\times S^{n-1})$ in $\R^{4n}$.
\item 
  $b$  realises a formal analytic amplitude obtained from $a$ by usual
  stationary phase. The first term $b_0$ is the only function
  such that $b_0(x,\omega,x,\omega)=a_0(x,\omega)$ and which extends
into 
  holomorphic function of $\omega_1+ix_1$ and $\omega_2-ix_2$. More
  generally, for any $j,\ell\in \N$, $b_j(x,\omega,x,\omega)$ only depends on $a_l$ and its
  derivatives up to degree $2j-\ell$ at $x,\omega$.
\end{itemize}
\end{enumerate}
\end{prop}
\begin{proof}
  The description of $\Lambda_T$ follows by an immediate computation. Now
  \begin{align*}
    \Lambda_{T^*}\circ \Lambda_T=\{(y_1,\eta_1,y_2,\eta_2),\exists
    (x,t_1,t_2,\omega),\,
    &\eta_2=t_2(\omega+i(x-y_2))\\
    &\eta_1=t_1(\omega-i(x-y_1))\\
    &t_2(\omega+i(x-y_2))=t_1(\omega-i(x-y_1))\\
    &t_2(x-y_2)=t_1(x-y_1)\\
    &(x-y_2)\cdot \omega=-i\tfrac{(x-y_2)^2}{2}\\
    &(y_1-x)\cdot \omega=-i\tfrac{(x-y_1)^2}{2}\}.
  \end{align*}
  The three first conditions ensure $\eta_1=\eta_2$. In particular,
  $\eta_1^2=\eta_2^2$, so that (using the last two conditions)
  $t_1=t_2$; finally the fourth condition gives
  $y_1=y_2$. Reciprocally, given $(y,\eta)\in T^*\R^d\setminus \{0\}$,
  setting $x=y,t_1=t_2=|\eta|,\omega=\eta/|\eta|$ yields
  $(y,\eta,y,\eta)\in \Lambda_{T^*}\circ \Lambda_T$.

  We now consider
  \begin{align*}
    \Lambda_{T}\circ \Lambda_{T^*}=\{(x_1,\xi_1,\omega_1,v_1,x_2,\xi_2,\omega_2,v_2),\exists
    (y,t_1,t_2),\,
    &\xi_1=t_1(\omega_1+i(x_1-y))\\
    &v_1=t_1(x_1-y)\\
    &\xi_2=t_2(\omega_2+i(y-x_2))\\
    &v_2=t_2(x_2-y)\\
    &t_1(\omega_1+i(x_1-y))=t_2(\omega_2+i(y-x_2))\\
    &(x_1-y)\cdot \omega_1=-i\tfrac{(x_1-y)^2}{2}\\
    &(y-x_2)\cdot \omega_2=-i\tfrac{(y-x_2)^2}{2}\}.
  \end{align*}
  Given $(x_1,\omega_1,x_2,\omega_2)$ fixed, conditions 1, 3, and 5 give $\xi_1=\xi_2$; together with the
  two last conditions, $\xi_1^2=\xi_2^2$ yields $t_1=t_2$. Using
  again condition 5, we can write
  $y=\frac{x_1-i\omega_1+x_2+i\omega_2}{2}$, and then either the
  last two conditions amounts to
  \[
    (\frac{x_1-i\omega_1-x_2-i\omega_2}{2})^2=-1;
  \]
  on this codimension 1 manifold, whose intersection with
  the real locus is exactly the diagonal, the space of solutions is a
  half-line described by $t_1$.

  The rest of the claim follows from there by applying
  Propositions \ref{prop:compo_well_def},
  \ref{prop:stat_phase_in_FIO}, and \ref{prop:FIO-rank1-phase}. The property of holomorphic extension
  of $b$ comes from that of the kernel of $T$ itself; the fact that
  holomorphic extension of real-analytic functions prescribed on the diagonal is unique comes from the fact that $\R^n\times
  B(0,1)$ is strongly pseudoconvex.
\end{proof}
The microlocal structure of operators of the form $TOp(a)T^*$ is
exactly that of the candidate for the Szeg\H{o} projector at the
boundary of $\R^n\times B(0,1)$, which is strongly pseudoconvex
since its natural defining function $|\omega|^2-1$ is strongly
p.s.h. Indeed,
\begin{align*}
  \psi(x_1,\omega_1,x_2,\omega_2)&=\frac{i}{4}\left((x_1-x_2)^2+(\omega_1-\omega_2)^2+2i(x_1-x_2)\cdot(\omega_1+\omega_2)\right)
  \\ &=-i\left(\frac
  14[(x_1+i\omega_1)^2+(x_2-i\omega_2)^2-2(x_1+i\omega_1)\cdot (x_2-i\omega_2)]-1\right)
\end{align*}
is the holomorphic extension of the function equal to
$i(1-|\omega|^2)$ on the diagonal $(x_1,\omega_1)=(x_2,\omega_2)$.

\begin{prop}\label{prop:Szego_is_FIO}
  The Szeg\H{o} kernel on $\R^n\times S^{n-1}$ is a Fourier Integral Operator with phase $\psi$.
\end{prop}
\begin{proof} Recall from Proposition \ref{prop:m_n_amplitude} that $a:\xi\mapsto
    e^{-|\xi|}\int_{S^{n-1}}e^{\xi\cdot y}\dd y$ is, up to an exponential factor,
    the realisation of an elliptic formal analytic amplitude. Thus
    \[
      b:\R^n\ni \xi \mapsto \frac{e^{-|\xi|}}{\int_{S^{n-1}}e^{\xi\cdot
          y}\dd y}
    \]
    is a radial function and the realisation of an elliptic formal analytic
    amplitude. Passing to spherical coordinates, the Szeg\H{o} kernel
    computed in Proposition \ref{prop:flat_Grauert} now reads
    \[
      K(z,w)=\frac{1}{(2\pi)^n}\int_0^{+\infty}b(2r)\left[e^{-2r}\int_{S^{n-1}}e^{ir(z-\overline{w})\cdot
          \nu}\dd \nu\right]r^{n-1}\dd r,
    \]
    which we simplify into
    \[
      K(z,w)=\frac{1}{(2\pi)^n}\int_0^{+\infty}e^{r\left(-2+\sqrt{-\sum_{j=1}^n(z_j-\overline{w_j})^2}\right)}b(2r)a(ir(z-\overline{w}))r^{n-1}\dd
      r.
    \]
    If $(z,w)$ belongs to a neighbourhood to $\{z=w\in \R^n\times
    S^{n-1}\}$, then $i(z-\overline{w})$ lies in a neighbourhood of
    $2S^{n-1}$ in $\C^n$, so that $ir(z-\overline{w})\in \Omega$.

    To conclude, near the diagonal, the Szeg\H{o} kernel
    takes the form of a Fourier integral operator whose phase is the
    holomorphic extension of 
    \[
      \psi_2(z,\overline{z})=-2+2|\im(z)|
    \]
    which is a defining function for $\R^n\times S^{n-1}$. By
    Proposition \ref{prop:FIO-rank1-phase}, one can rewrite it as a
    Fourier Integral Operator with phase $\psi$. Away from the
    diagonal we already know by Proposition \ref{prop:flat_Grauert}
    that the Szeg\H{o} kernel is real-analytic.
  \end{proof}

\begin{prop}\label{prop:invesion_Toeplitz}Given $U\subset \R^n\times
  S^{n-1}$, open and relatively compact, we let
  \[\mathcal{H}(U)=\ker_{\mathcal{E}'(\overline{U})}(\partial_b)/(C^{\omega}(\overline{U})+\mathcal{E}'(\partial
  U)),\]
  and $S(U)$ the space of formal analytic amplitudes on
  $\overline{U}\times \R_+$.
  \begin{enumerate}
    \item 
  Given $V\subset \R^n\times S^{n-1}$ and $a\in S(V)$, there exists a
  neighbourhood $W$ of $\overline{V\times V}$ and a unique formal analytic amplitude $\widetilde{a}$
  on $W\times \R_+$, holomorphic in the second variable and
  anti-holomorphic in the first variable.
\item 
  Given $U\Subset V$ open and
  $\epsilon>0$ small enough, letting $\Omega=\{(x,t,y)\in U\times \R^+\times V,
  \dist(x,y)<\epsilon\}$, then $(\phi,\Omega,U)$ satisfies
  \eqref{eq:boundary_positivity} and
  $I_{\psi,\Omega,U}(\widetilde{a})$ maps
  $\mathcal{F}(V)$ into $\mathcal{H}(U)$. 
 \item 
  There exists a sequence $(B_j)_{j\in \N}$ of bi-differential operators
  on $\R^n\times S^{n-1}$ such that for every $j$, $B_j$ has total
  degree $2j$, and for every relatively compact $V\subset \R^n\times
  S^{n-1}$, $S(V)$ is a unit algebra for the product
  \[
    (S(V))^2\ni (a,b)\mapsto (a\sharp
    b)_k=\left(\sum_{j+\ell=k}B_{k-j-\ell}(a_j,b_\ell)\right)_k.
  \]
  \item Invertible elements for the product above are exactly the formal amplitudes
    $a=(a_j)_{j\in \N}$ such that $a_0$ never vanishes.
    \item 
  For every $U_1\Subset U_2\Subset V$ and $a,b\in S(V)$,
  letting $\epsilon$ small enough and
  \begin{align*}
    \Omega_1&=\{(x,t,y)\in U_1\times \R^+\times U_2,
              \dist(x,y)<\epsilon\}\\
    \Omega_2&=\{(x,t,y)\in U_2\times \R^+\times V,
              \dist(x,y)<\epsilon\},
  \end{align*}
  one has
  \[
    I_{\psi,\Omega_1,U_1}(\widetilde{a})I_{\psi,\Omega_2,U_2}(\widetilde{b})=I_{\psi,\Omega_1,U_1}(\widetilde{a\sharp
      b}).
  \]
  \item 
  The unit for $\sharp$ acts as the identity on
  $\mathcal{H}$.
\end{enumerate}
\end{prop}
\begin{proof}
  Since $\R^n\times S^{n-1}$ is strongly pseudoconvex, the diagonal of
  $(\R^n\times S^{n-1})^2$ is a totally real submanifold, where we
  reverse the CR structure on the left factor. Therefore any
  real-analytic function on an open set $V\subset \R^n\times S^{n-1}$
  can be extended CR-holomorphically to a neighbourhood of $\diag(V)$ in
  $(\R^n\times S^{n-1})^2$.

  Given a formal analytic amplitude $\widetilde{a}$ near the diagonal such
  that $\partial_{b;1}\widetilde{a}=0$ (by which we mean: it is anti-CR-holomorphic in
  the first variable), since $\partial_{b;1}\psi=0$ as well, local
  realisations of $I_{\psi,\Omega,U}(\widetilde{a})$ as in Definition
  \ref{def:FIO_fixed_fibre} will indeed map $\mathcal{E}'(\overline{V})$ into
  functions $u$ such that $\partial_{b;1}u$ is real-analytic on a
  neighbourhood of $\overline{U}$. Therefore, by Proposition
  \ref{prop:loc-anal-hypoell-Grauert}, if $S$ denotes the
  (antiholomorphic) Szeg\H{o}
  projector one has indeed $u=(1-S)u+Su$ where
  $\partial_{b}Su=0$ and $(1-S)u$ is real-analytic on
  $\overline{U}$. In particular, $I_{\psi,\Omega,U}(\widetilde{a})$ indeed maps
  $\mathcal{F}(V)$ into $\mathcal{H}(U)$.

  Recalling from Proposition \ref{prop:T*T} that
  $\Lambda_{\psi}=\Lambda_{T}\circ \Lambda_{T^*}$ and that
  $\Lambda_{T^*}\circ \Lambda_T$ is the diagonal of $(T^*\C^n)^*$, we
  obtain that $\Lambda_{\psi}\circ \Lambda_{\psi}=\Lambda_{\psi}$. We
  also proved in Proposition \ref{prop:T*T} that any analytic Fourier Integral
  Operator with Lagrangian $\Lambda_{\psi}$ can be written as a
  Fourier Integral Operator with phase $\psi$. Therefore, given
  $V,a,b, U_1,U_2,\Omega_1,\Omega_2$ as in the claim, there exists a formal analytic
  amplitude $c$ such that
  $I_{\psi,\Omega_1,U_1}(\widetilde{a})I_{\psi,\Omega_2,U_2}(\widetilde{b})=I_{\psi,\Omega_1,U_1}(c)$. The
  amplitude $c$ is obtained from $a$, $b$ and $\psi$ by stationary phase and
  change of variables in the middle variables, so that $c$ is
  CR-holomorphic on $\Omega_1$. By the usual $C^{\infty}$ calculus of
  these operators (see
  for instance \cite{boutet_de_monvel_sur_1975}), as a formal amplitude, the restriction of
  $c$ to the diagonal is indeed of the form $a\sharp b$ as in the
  claim.

  We know by Proposition \ref{prop:Szego_is_FIO} that the Szeg\H{o}
  projector has a kernel real-analytic from the diagonal and of the
  form $I_{\psi,\Omega,U}(a)$, where $U=\R^n\times S^{n-1}$,
  $\Omega=(\R^n\times S^{n-1})^2$ and $a$ depends only on the fibre
  variable. By definition, $S$ acts as identity on $\mathcal{H}$; its
  amplitude is a unit for the algebra above.

  It remains to prove that elliptic amplitudes can be inverted, which is
  a claim of formal nature (by the product formula, it suffices to
  check it at the level of expansion of amplitudes, and not
  a priori as operators making sense from some function space to
  another). To this
  end, we first suppose that amplitudes are globally defined, so that for every $a\in
  S(\R^n\times S^{n-1})$, by Proposition \ref{prop:T*T},
  the operator $T^*I_{\psi,\Omega,V_1}(a)T$ is a pseudodifferential
  operator on $\R^n$. Moreover if $a$ is elliptic then so
  is this pseudodifferential operator, so that by Proposition
  \ref{prop:inversion_formal_pseudos} it admits an inverse with amplitude
  $b$. 

  Now, $M=I_{\psi}(a)TOp(b)T^*$ satisfies, for every $c\in
  S_{\mathcal{H}}(\R^n\times S^{n-1})$,
  \[
    T^*MI_{\psi}(c)T=TI_{\psi}(c)T^*,
  \]
  and since $T$ and $T^*$ are invertible, we conclude that $M$ acts as
  identity on $\mathcal{H}$. Moreover, by the previous considerations,
  $M$ can be written in the form $I_{\psi}(d)$ for some $d\in
  S_{\mathcal{H}}(\R^n\times S^{n-1})$.

  As a formal amplitude, $d$ is obtained from $a$ by a formula of the
  form $d=\sum_{k=0}^{+\infty}C_k(a,a_0^{-1})$, for some
  sequence of bidifferential operators $C_k$ (this is a consequence of
  Propositions \ref{prop:T*T} and \ref{prop:inversion_formal_pseudos};
  one can also directly prove it using the $C^{\infty}$ calculus). Applying this formula to
  a locally defined amplitude $a$ leads necessarily its inverse for
  the formal product.
\end{proof}

In the last proof and from now on, given $V_0\subset \R^n\times S^{n-1}$ and given a
formal analytic amplitude $a$ on $V_0\times (0,+\infty)$, we will
sometimes drop the domains and consider $I_{\psi}(a)$ as the collection of
operators of the form
$I_{\psi,\Omega,U}(\widetilde{a}):\mathcal{F}(V)\to \mathcal{H}(U)$ where $U\Subset
V\Subset V_0$ and $\Omega$ is a small neighbourhood of $U\times
V\times (0,+\infty)$ on which $a$ extends into $\widetilde{a}$,
holomorphic in the first variable and anti-holomorphic in the second
variable.
These operators will always strictly decrease the definition set of
the function on which they act.

The operator $T^*$ sends $\mathcal{O}(\R^n\times S^{n-1})$ to
$\mathcal{O}(\R^n)$, and therefore $T$ is well-defined from
$\mathcal{O}'(\R^n)$ to $\mathcal{O}'(\R^n\times S^{n-1})$. Thus we
can define the \emph{analytic wave front set} of an analytic
functional as follows.

\begin{defn}\label{def:WFa}
  The \emph{analytic wave front set} of $u\in \mathcal{O}'(\R^n)$ is
  defined as the analytic singular support of $Tu$.
\end{defn}
This is equivalent to the usual definition of analytic singular
support by means of the FBI transform, which asks for exponential
decay of a related quantity. In particular, it is equivalent to all
of the classical notions of analytic wave front set
\cite{bony_equivalence_1977}.

We conclude this section with a few comments. Our algebra of operators
defined in Proposition \ref{prop:invesion_Toeplitz} is inspired from
``covariant Toeplitz operators'', originally defined in the semiclassical
setting in \cite{charles_berezin-toeplitz_2003} as an alternative
description of operators of the ``contravariant Toeplitz''
form $SaS$, where $S$ is the Szeg\H{o} projector and $a$ acts by
multiplication; these operators also form an algebra in analytic
regularity \cite{deleporte_toeplitz_2018,rouby_analytic_2018}.

It it a bit awkward that
the natural range of $T$, and therefore the space on which our algebra
acts, extends into a space of \emph{anti}holomorphic functions on
$\R^n\times B(0,1)$. Alternative conventions, however, are not more
satisfactory; one can set the FBI phase as $(y-x)\cdot
\omega+i\frac{|x-y|^2}{2}$ and obtain holomorphic functions, but at
the end we
obtain a characterisation of the analytic wave front set of $u$ as the
analytic singular support of $(x,\omega)\mapsto Tu(x,-\omega)$. One
can also simply swap the variables $x$ and $\omega$, but this is at odds with the interpretation of $\R^n\times
S^{n-1}$ as the cosphere bundle over $\R^n$ with a natural CR
structure obtained by identifying $T^*\R^n$ and $\widetilde{\R^n}=\C^n$;
this ``Grauert tube'' vision of the FBI transform extends naturally to
manifolds with a real-analytic Riemannian structure
\cite{guillemin_grauert_1991,guillemin_grauert_1992,leichtnam_intrinsic_1996},
see also the semiclassical approach \cite{sjostrand_density_1996}.

The Grauert tube approach is
essentially contained in our toolbox; the FBI transform as introduced
in \cite{leichtnam_intrinsic_1996}, for instance, can be written as an
analytic Fourier Integral Operator whose
Lagrangian has similar properties to the flat case, and the definition
of the wave front set is equivalent to the one above, via analytic
charts. We will say more on this in Section \ref{sec:appl-grau-tubes}.

\subsection{Advanced properties of Fourier Integral Operators}
\label{sec:advanced-properties}

With help of the FBI transform and the structure of Toeplitz
operators, we can prove that general Fourier Integral Operators push
the analytic wave front set of a distribution as expected, and also an
microlocal ellipticity result for pseudodifferential operators.

\begin{prop}\label{prop:WF_Pu}
  Let $a$ be a formal analytic symbol on $\R^n$. Then
  \[
    WF_a(u)\subset \{a_0=0\}\cup WF_a(Op(a)u).
  \]
\end{prop}
\begin{proof}
  Let $(x,\omega)\in \R^n\times S^{n-1}$ and suppose that $TOp(a)u$ is
  real-analytic near $(x,\omega)$ and $a_0$ is bounded away from $0$
  near $(x,\omega)$. Modulo an analytic function, one has
  \[
    TOp(a)u=TOp(a)Op(r)T^*Tu=I_{\psi}(b)Tu,
  \]
  where $b$ is a formal analytic amplitude such that $b_0$ is bounded
  away from $0$ near $(x,\omega)$. By Proposition
  \ref{prop:invesion_Toeplitz}, on a neighbourhood $V$ of $(x,\omega)$,
  $b$ admits a formal inverse for the $I_{\psi}$ calculus. Letting $d$
  denote the formal analytic amplitude of its inverse, given $U\Subset V$, one has, modulo an
  analytic function,
  \[
    I_{\psi,U\times V,U}(d)(TOp(a)u) = I_{\psi,U}(d)I_{\psi}(b)Tu=u;
  \]
  now, by Proposition \ref{prop:reg_to_reg}, since $TOp(a)u$ is
  real-analytic on $V$, the left-hand side is real-analytic on $U$. This concludes the proof.
\end{proof}

\begin{prop}\label{prop:FIO-WF}
  Let $(\phi,\Omega,U)$ satisfy \eqref{eq:boundary_positivity} and let
  $a$ be a formal analytic amplitude on $\Omega$. Then
  \[
    WF_a(I_{\phi,\Omega,U}(a)u)\subset \{(x,\xi)\in T^*\R^{n_x}\setminus \{0\},
    \exists (y,\xi)\in WF_a(u), (x,\xi,y,\eta)\in (\Lambda_{\phi})_{\R}\}.
  \]
\end{prop}
\begin{proof}
  The operator $T\circ I_{\phi,\Omega,U}(a)\circ T^*$ is a Fourier
  Integral Operator whose real locus of the Lagrangian projects on the
  base onto
  \[
    \{(x,\omega,y,v)\in (\R^{n_x}\times S^{n_x-1}\times \R^{n_y}\times
    S^{n_y-1})\cap (\Lambda_{\phi})_{\R}\};
  \]
  thus, we are left with the claim
  \[
    SS_a(I_{\phi,\Omega,U}(a)u)\subset \{x\in \R^{n_x}, \exists
    \xi,\eta\neq 0, \exists y\in SS_a(u), (x,\xi,y,\eta)\in
    (\Lambda_{\phi})_{\R}\},
  \]
  which was proved in Theorem \ref{prop:FIOs}.
\end{proof}

Another important application of the FBI transform is the construction
of unitary Fourier Integral operators associated with arbitrary real
symplectic maps, also called ``quantized contact transformations''.

\begin{prop}\label{prop:quantized-contact}
  Let $\kappa$ be a one-homogeneous symplectic transformation
  between respective neighbourhoods of two open cones $U,V$ of $T^*\R^n$. Let $[U]=U\cap
  (\R^n\times S^{n-1})$ and $[V]=V\cap (\R^n\times S^{n-1})$. Then there
  exists $\hat{\kappa}:\mathcal{H}([V])\to \mathcal{H}([U])$ and
  $\widehat{\kappa^{-1}}:\mathcal{H}([U])\to \mathcal{H}([V])$ such that
  the following is true.
  \begin{itemize}
  \item for every formal analytic amplitude $a\in
    S([V])$, there exists $b\in S([U])$, such
    that, for every $U_0\Subset U$, there exists $V_0\Subset V$, a neighbourhood $\Omega$ of
    $\diag([U_0])$ in $(\R^n\times S^{n-1})\times [U]\times (0,+\infty)$, a neighbourhood
    $\Omega'$ of $\diag([V_0])$ in $(\R^n\times S^{n-1})\times
    [V]\times (0,+\infty)$ with
    \[
      \hat{\kappa}I_{\psi,\Omega',[V_0]}(b)\widehat{\kappa^{-1}}=I_{\psi,\Omega,[U_0]}(a):\mathcal{H}([U])\to \mathcal{H}([U_0]),
    \]
    where $b_0=a_0\circ \kappa$; if $a$ is the amplitude of the
    Szeg\H{o} projector then so is $b$.
    \item for every formal analytic amplitude $b\in
    S([U])$, there exists $a\in S([V])$ such
    that, for every $V_0\Subset V$, there exists $U_0\Subset U$, a neighbourhood $\Omega$ of
    $\diag([U_0])$ in $(\R^n\times S^{n-1})\times [U]\times (0,+\infty)$, a neighbourhood
    $\Omega'$ of $\diag([V_0])$ in $(\R^n\times S^{n-1})\times [V]\times (0,+\infty)$with
    \[
      \widehat{\kappa^{-1}}I_{\psi,\Omega,[U_0]}(a)\hat{\kappa}=I_{\psi,\Omega',[V_0]}(b):\mathcal{H}([V])\to \mathcal{H}([V_0]),
    \]
    where $a_0=b_0\circ \kappa^{-1}$; if $b$ is the amplitude of the
    Szeg\H{o} projector then so is $a$.
  \item For every $U_0\Subset U$ and $V_0\Subset V$, the distribution \[(x,y)\mapsto
    \widehat{\kappa^{-1}}(x,y)-\overline{\hat{\kappa}(y,x)}\] is
    real-analytic on $V_0\times U_0$.
  \end{itemize}
\end{prop}
\begin{proof}
  For convenience we drop the bracket notations and identify open sets of
  $\R^n\times S^{n-1}$ with open cones in $T^*\R^n$. Thus $\kappa$ is
  a contact transformation on $\R^n\times S^{n-1}$. We
  let $\phi_{\kappa}$ be any function from a
  neighbourhood $W$ of the graph of $\kappa$ in $V_1\times U_1$,
  where $U_1\Subset U$ and $V_1\Subset V$, which is
  holomorphic in $x_1-i\omega_1$ and in $x_2+i\omega_2$, and such that
  $\im(\phi_{\kappa})\asymp
  \dist((x_1,\omega_1),\kappa(x_2,\omega_2))$. Such a
  function exists thanks to Proposition
  \ref{prop:completing_Lagrangians}: it is a phase function for the
  only Lagrangian containing the graph of $\kappa$.

  We now set
  \[
    \widehat{\kappa_0}=I_{\psi_{\kappa},W,V_1}\1_{(x_1,\omega_1,x_2,\omega_2)\in
      W}\int_0^{+\infty}e^{it\psi_{\kappa}(x_1,\omega_1,x_2,\omega_2)}t^{n-1}\dd
    t
  \]
  and we similarly define $\widehat{\kappa_0^{-1}}$ based on
  $\kappa^{-1}$.

  Given $a$ as in the claim, the operator
  $\widehat{\kappa_0^{-1}}I_{\psi,\Omega,U_0}(a)\widehat{\kappa_0}$ is, by
  Theorem \ref{prop:FIOs}, a Fourier integral operator with the
  same Lagrangian as the Szeg\H{o} projector (indeed, the real locus
  of the Lagrangian is the diagonal, and this Lagrangian is
  holomorphic for the skew CR structure on $(\R^n\times
  S^{n-1})^2$). Therefore it is of the form $I_{\psi,\Omega',[V_0]}(b)$
  for another analytic formal amplitude $b$ obtained from $a$ by stationary
  phase. In particular, letting $S$ denote the Szeg\H{o} projector,
  \[
    \widehat{\kappa_0^{-1}}\widehat{\kappa_0}=\widehat{\kappa_0^{-1}}S\widehat{\kappa_0}=I_{\psi,\Omega',V_1}(f),
  \]
  where $f_0\neq 0$ and $V_1$ is any relatively compact open set of
  $V$. By Proposition \ref{prop:invesion_Toeplitz}, $f$ admits an
  inverse for the product of analytic Fourier Integral Operators with
  phase $\psi$, and therefore there exists a formal analytic amplitude $g$
  on a neighbourhood of $V$, such that
  \[
    I_{\psi,\Omega',V_1}(g)\widehat{\kappa_0^{-1}}\widehat{\kappa_0}=S:\mathcal{H}(V)\to
    \mathcal{H}(V_1).
  \]
  We now set $\hat{\kappa}=\widehat{\kappa_0}I_{\psi,
    \Omega',V_1}(\sqrt{g})$ and
  $\widehat{\kappa^{-1}}=I_{\psi,\Omega',V_1}(\sqrt{g}^*)\widehat{\kappa_0^{-1}}$. In
  this algebra,
  the square root and adjoint of an amplitude are well-defined because, as
  in the proof of
  Proposition \ref{prop:invesion_Toeplitz}, the calculus of these
  amplitudes is conjugated via $T$ to the symbol calclulus of
  pseudo-differential operators. With this correction, we obtain the
  last claim.

  We also obtain that the second part of the claim holds with the quantifiers on $a$ and
  $b$ exchanged (for every $a$, there exists $b$, \ldots).

  Given $b$, the operator
  \[
    \hat{\kappa}I_{\psi,\Omega',V_0}(b)\widehat{\kappa^{-1}}
  \]
  is also an analytic Fourier Integral Operator whose Lagrangian is
  that of $\psi$, so that it is of the form
  $I_{\psi,\Omega,U_0}(a)$, and we obtain the first part of the claim
  with the quantifiers on $a$ and $b$ exchanged (for every $b$, there
  exists $a$, \ldots)

  The particular case
  $\hat{\kappa}S\widehat{\kappa^{-1}}$ is a projector, because we
  already made sure that $\widehat{\kappa^{-1}}\hat{\kappa}$ acts as
  identity; however it is also an invertible operator (because it is a
  Fourier Integral Operator with phase $\psi$ and nonvanishing
  principal amplitude). Therefore it is the identity.

  To conclude, given $a\in S(V)$, one recovers $b$ so that the first
  point of the proposition is true by letting $b$ be the amplitude of
  $\widehat{\kappa^{-1}}I_{\psi,\Omega,U_0}(a)\hat{\kappa}$, and
  conversely.

\end{proof}

\begin{rem}\label{rem:replacement}
To conclude this section, %
we can go back to Remark
\ref{rem:EtoO}. Now that we have characterised the analytic wave front
set of the singular kernel of an analytic Fourier Integral Operator
(Proposition \ref{prop:FIO-WF}), we observe that these operators are
well-defined when acting on singularity hyperfunctions, that is, elements of
$\mathcal{F}_{\omega}(V):=\mathcal{O}'(\overline{V})/(\mathcal{O}'(\partial V)+ \mathcal{O}(\overline{V}))$.
Indeed, as in the smooth case the wave front set conditions (see
\cite{hormander_analysis_2003}, Theorems 8.5.1 and 8.5.2) are
satisfied. Thus, in all of Section \ref{sec:four-integr-oper}, we can now replace $\mathcal{F}$ with
$\mathcal{F}_\omega$, and $\mathcal{H}$ with
$\mathcal{H}_\omega:U\mapsto
\ker_{\mathcal{O}'(\overline{U})}(\overline{\partial}_b)/(\mathcal{O}(\overline{U})+\mathcal{O}'(\partial
U))$. This replacement is crucial in the proof of Theorem \ref{prop:analytic-Szego}.
\end{rem}
\section{The Szeg\H{o} kernel and Toeplitz operators for general
  pseudoconvex open sets}
\label{sec:szegho-kern-toepl}

\subsection{Normal form for the $\overline{\partial}_b$ operator}
\label{sec:norm-form-overl}
The features of $\overline{\partial}_b$ on the boundary of a strongly
pseudoconvex manifold are the following: locally, it is a
vector-valued, degree $1$ analytic pseudodifferential operator
$(D_1,\ldots,D_d)$ on a $2d-1$-dimensional manifold, such that
$[D_j,D_k]=0$ for every $j\neq k$ and the matrix of principal symbols $\sigma(i[D_j,D_k^*])$
is either positive definite or negative definite on
$\Sigma=\{\sigma(D_j)=0\forall 1\leq j\leq d\}$.

All of this is also true of the following model operator $D_0$, acting
on $\R^{d-1}_x\times \R^{d}_y$:
\[
  (D_0)_j=-i\frac{\partial}{\partial x_j}-x_j\frac{\partial}{\partial
    y_1}.
\]

Our goal is to prove the following result attributed to Sato:
microlocally near any point of $\Sigma$, $\overline{\partial}_b$ is
conjugated with $D_0$ by an analytic Fourier Integral Operator (as
always, modulo
an operator continuously mapping $\mathcal{E}'$ into
real-analytic-functions). Here, as before, ``microlocally'' will mean
``locally after conjugation with the FBI transform''.

The first step (Proposition \ref{prop:canon_transf_to_local_model}) is to solve the classical problem, and find a
symplectic transformation which maps the symbol of
$\overline{\partial}_b$ to the symbol of $D_0$, multiplied by an
invertible matrix; this was described in Section
\ref{sec:class-norm-form}. W move on to an exact conjugation in the
algebra of analytic formal symbols of pseudodifferential operators
(Proposition \ref{prop:correct_D0_error_formal}), and we conclude by realising
all requested transformations on the FBI side (Proposition \ref{prop:Sato_conj}).

Given a unitary Fourier integral operator $\hat{\kappa}$ which quantizes $\kappa$, the
principal symbol of $\widehat{\kappa^{-1}}\overline{\partial}_b\hat{\kappa}$ is precisely that of
$Op(C)D_0$. It remains to correct the subprincipal (degree $0$) remainder, which we can do by
conjugating with an elliptic, diagonal matrix-valued pseudodifferential operator
\begin{prop}\label{prop:correct_D0_error_formal}
  Let $r=(r_1,\ldots,r_d)$ be a formal analytic degree 0 vector-valued symbol in the neighbourhood of a point
  where $\sigma(D_0)=0$. Then there exists a smaller neighbourhood of
  this point and a vector of degree 0 elliptic symbols $a=(a_1,\ldots,a_d)$ such
  that, on this neighbourhood,
  \[
    (\sigma(D_0)_j+r_j)\sharp a_j=a_j\sharp \sigma(D_0)_j \qquad
    \forall 1\leq j\leq d.
    \]
\end{prop}
\begin{proof}We first remove the order 0 term in $r$, then all other terms.

  At order 0, the following equation links the principal symbol of $a_j$
to the principal symbol of $r_j$
\[
  \partial_{x_j}a_j^0+ix_j\partial_{y_1}a_j^0-i\eta_1\partial_{\xi_j}a_j^0=ir_j^0a_j^0.
\]
This transport equation admits a solution, together with the Cauchy data $a_j^0|_{x_j=0}=1$, written as
\[
  a_j^0(x,y,\xi,\eta)=\exp\left[i\int_0^{x_j}r_j^0(x-(1-t)e_j,y+i\tfrac{t^2}{2}e_1,\xi+i(1-t)\eta_1e_j,\eta)\dd
    t\right].
\]
Since $r_j^0$ is a real-analytic function, this solution is
well-defined for $x_j$ sufficiently close to $0$, and is again
real-analytic. Hence, we can solve for $a$ at principal
order. Thus
\[
  (a^0_j)^{\sharp -1}\sharp (\sigma(D_0)_j+r_j)\sharp
  a_j^0=\sigma(D_0)_j+r_j',
  \] where $r_j'$ is a degree $-1$ analytic symbol.

Consider now the family
\[
  [0,1]\ni t\mapsto \sigma(D_0)_j+tr_j'.
\]
We wish to find a degree $-1$ elliptic symbol $a_j'(t)$ such that
\begin{equation}\label{eq:rprimeaprime}
  (\sigma(D_0)_j+tr_j')\sharp(1+a_j'(t))=(1+a_j'(t))\sharp (D_0)_j.
\end{equation}

We will find $a_j'(t)$ as a solution of the differential equation
\[\partial_ta_j'(t)=ib_j(t)\sharp (1+a_j'(t)).\]
First, differentiating \eqref{eq:rprimeaprime} with respect to $t$ yields
\[
  r_j'=-i[\sigma(D_0)_j+tr_j',b_j(t)],
\]
and we want to solve this equation for $b_j(t)$.

One has
\begin{equation}\label{eq:exact_conj_D0}
  [(D_0)_j,Op(b_j(t))]=Op(-i\partial_{x_j}b_j(t)-x_j\partial_{y_1}b_j(t)+\eta_1\partial_{\xi_j}b_j(t)).
\end{equation}
From now on we drop the index $j$ and the parameter $t$ on $r'$ and $b$. We
write their full symbols as
\[
  r'=\sum_{k\geq 1} r_k'\qquad \qquad b=\sum_{k\geq 1}b_k,
\]
where $b_k$ and $r_k'$ are homogeneous of degree $-k$.

The equation above yields the following family of equations on $b_k$:
\[
  -i\partial_{x_j}b_k
  +\eta_1\partial_{\xi_j}b_k-x_j\partial_{y_{1}}b_k=ir_k'+t(r'\sharp
  b-b\sharp r')_k.
\]
Note that the right-hand-side only involves $b_l$ for
$1\leq l\leq k-1$. Thus, we can find the $b_k$ by induction,
recursively solving a transport equation of the form
\[
  i\partial_{x_j}u-\eta_{1}\partial_{\xi_j}u+x_j
  \partial_{y_{1}}u=g,
\]
using again the method of characteristics: the solution (with Cauchy data
$u=0$ on $x_j=0$) is
\[
  u:(x,y,\xi,\eta)\mapsto
  \int_0^{x_j}g(x-(1-t)e_j,y+i\tfrac{t^2}{2}e_1,\xi+i(1-t)\eta_1e_j,\eta)\dd t.
\]

We have to prove that the successive solutions $(b_k)_{k\in \N}$ form an
analytic symbol. It turns out that a convenient analytic class is
given by the infinite jet of the $S^{\rho,R}_m$ class at $x_j=0$; namely
we let
\[
  \|a\|_{JS^{\rho,R}_m}=\sup_{\alpha,k}\frac{\sup_\Omega(|\nabla^{\alpha}a_k|)(1+|\alpha|+k)^m}{(|\alpha|+k)!\rho^{|\alpha|}R^k},
\]
where
\[
  \Omega=\{|\eta_{1}|\in [\tfrac 14, \tfrac
  12],|x|+|y|+|\xi|+|(\eta_2,\ldots,\eta_{d-1})|<\epsilon,x_j=0\}.
\]
Since the product of symbols is an infinitesimal operation, we have
(see Proposition \ref{prop:composition-formal})
\[
  \|(b\sharp r'-r'\sharp b)\|_{JS^{\rho,R}_m}\leq
  24\|r'\|_{JS^{\frac{\rho}2, \frac R4}_m}\|b\|_{JS^{\rho,R}_m}.
\]

Suppose by induction that, for all $l\leq k-1$, one has
\[
  \sup_{\Omega}\|\nabla^{\alpha}b_l\|\leq
  C_b\frac{\rho^{|\alpha|}R^l(|\alpha|+l)!}{(|\alpha|+l+1)^m}.
\]
Then one has readily, for all $t\in [0,1]$, if $m\geq C(d)$,
\[
  \sup_{\Omega}|\nabla^\alpha(r'_k+t(r'\sharp b - b \sharp r')_k)|\leq
  (24C_b+1)\|r\|_{JS^{\frac{\rho}2,\frac
      R4}_m}\frac{\rho^{|\alpha|}R^k(|\alpha|+k)!}{(|\alpha|+k+1)^m}.
\]
Denote $g_k=ir'_k+t(r'\sharp b - b \sharp r')_k$, and let
\[
  C(g_k,k)=\sup_{\alpha}\frac{\sup_\Omega(|\nabla^{\alpha}g_k|)(1+|\alpha|+k)^m}{(|\alpha|+k)!\rho^{|\alpha|}R^k}
\]
and
\[
  C(b_k,k)=\sup_{\alpha}\frac{\sup_\Omega(|\nabla^{\alpha}b_k|)(1+|\alpha|+k)^m}{(|\alpha|+k)!\rho^{|\alpha|}R^k}.
\]
We just proved that
\[C(g_k,k)\leq \|r'\|_{JS^{\frac{\rho}2, \frac R4}_m}(1+24\sup_{ \ell
    \leq k-1}C(b_\ell,\ell)).
\]
Let us now prove that
\[
  C(b_k,k)\leq C(g_k,k).
\]
To this end, we write the transport equation
\[
  i\partial_{x_j}b_k-\eta_{1}\partial_{\xi_j}b_k+x_j
  \partial_{y_{1}}b_k=g_k
\]
in terms of the power series at $x_j=0$
\[
  b_k=\sum_{n=1}^{+\infty} \frac{b_k^{(n)}}{n!}x_j^n \qquad \qquad g_k=\sum_{n=1}^{+\infty}
  \frac{g_k^{(n)}}{n!}x_j^n.\] One has, for all $n\geq 1$,
\begin{equation}\label{eq:transport}
  ib_{k}^{(n+1)}-\eta_{1}\partial_{\xi_j}b_{k}^{(n)}+n\partial_{y_{1}}b_{k}^{(n-1)}=g_k^{(n)},
\end{equation}
and for $n=0$,
\[
  b_{k}^{(0)}=0.
\]

Let us prove by induction on $n$ that
\[
  |\nabla^{\alpha}b_k^{(n)}|\leq
  C(g_k,k)\frac{\rho^{n+|\alpha|}(n+|\alpha|+k)!}{(n+|\alpha|+k+1)^m}.
\]
This is true at $n=0$. If it is true at $n-1$ and $n$, then there are
three terms in
\begin{equation}\label{eq:ctrl_b_g}
  |\nabla^{\alpha}b_k^{(n+1)}|\leq|\nabla^{\alpha}g_k^{(n)}|+|\nabla^{\alpha}(\eta_{1}\partial_{\xi_j}b_k^{(n)})|-|n\nabla^{\alpha}\partial_{y_{1}}b_k^{(n-1)}|.
\end{equation}
First, one has directly
\begin{align*}
  |\nabla^{\alpha}g_k^{(n)}|&\leq
                          C(g_k,k)\frac{\rho^{(n+|\alpha|)}R^k(n+|\alpha|+k)!}{(n+|\alpha|+k+1)^m}\\
                        &\leq
                          C(g_k,k)\frac{\rho^{(n+1+|\alpha|)}R^k(n+1+|\alpha|+k)!}{(n+|\alpha|+k+2)^m}\times
                          \frac{\left(1+\frac{1}{n+|\alpha|+k+1}\right)^m}{\rho(n+1+|\alpha|+k)}.
\end{align*}
If $\rho\geq 6\left(\frac{3}{2}\right)^m$, then we obtain
\[
  |\nabla^{\alpha}g_k^{(n)}|\leq
  \frac{1}{6}C(g_k,k)\frac{r^{(n+1+|\alpha|)}R^k(n+1+|\alpha|+k)!}{(n+|\alpha|+k+2)^m}.
\]
We can also directly control the last term in \eqref{eq:ctrl_b_g} in the same way, using the
induction hypothesis:
\begin{align*}
  |n\nabla^{\alpha}\partial_{y_{d+1}}b_k^{(n-1)}|&\leq
                                                 nC(g_k,k)\frac{\rho^{n+|\alpha|}R^k(n+|\alpha|+k)!}{(n+|\alpha|+k+1)^m}\\
                                               &\leq
                                                 C(g_k,k)\frac{\rho^{n+1+|\alpha|}R^k(n+1+|\alpha|+k)!}{(n+|\alpha|+k+2)^m}\times
                                                 \frac{\left(1+\frac{1}{n+|\alpha|+k+1}\right)^m}{\rho}\frac{n}{n+1+|\alpha|+k}.
\end{align*}
As previously, if $\rho\geq 6\left(\frac 32\right)^m$, we obtain
\[
  |n\nabla^{\alpha}\partial_{y_{1}}b_k^{(n-1)}|\leq
  \frac{1}{6}C(g_k,k)\frac{\rho^{(n+1+|\alpha|)}R^k(n+1+|\alpha|+k)!}{(n+|\alpha|+k+2)^m}.\]

In the second term of \eqref{eq:ctrl_b_g}, we have to distinguish between two cases,
depending on whether or not one of the differentials has hit
$\eta_1$. Thus
\[
  |\nabla^{\alpha}(\eta_{1}\partial_{\xi_j}b_k^{(n)})|\leq
  |\alpha||\nabla^{\alpha-\gamma}\partial_{\xi_j}b_k^{(n)}|+|\eta_{1}||\nabla^{\alpha}\partial_{\xi_j}b_k^{(n)}|,
\]
where $\gamma$ is the length $1$ polyindex corresponding to the
direction of $\eta_{1}$.

As before,
\begin{align*}
  |\alpha||\nabla^{\alpha-\gamma}\partial_{\xi_j}b_k^{(n)}|&\leq
                                                            |\alpha|C(g_k,k)\frac{\rho^{n+|\alpha|}R^k(n+|\alpha|+k)!}{(n+|\alpha|+k+1)^m}\\
                                                          &\leq
                                                            C(g_k,k)\frac{\rho^{n+1+|\alpha|}R^k(n+1+|\alpha|+k)!}{(n+|\alpha|+k+2)^m}\times
                                                            \frac{\left(1+\frac{1}{n+|\alpha|+k+1}\right)^m}{\rho}\frac{|\alpha|}{n+1+|\alpha|+k},
\end{align*}
and if $\rho\geq 6\left(\frac 32\right)^m$, we obtain
\[
  |\alpha||\nabla^{\alpha-\gamma}\partial_{\xi_j}b_k^{(n)}|\leq
  \frac{1}{6}C(g_k,k)\frac{\rho^{(n+1+|\alpha|)}R^k(n+1+|\alpha|+k)!}{(n+|\alpha|+k+2)^m}.\]

The last term is
\[
  |\eta_{1}||\nabla^{\alpha}\partial_{\xi_j}b_k^{(n)}|\leq
  |\eta_{1}|C(g_k,k)\frac{\rho^{(n+1+|\alpha|)}R^k(n+1+|\alpha|+k)!}{(n+|\alpha|+k+2)^m}.
\]
Recall that we are controlling derivatives on the set $\Omega$ where
$|\eta_1|\leq \frac 12$. We obtain
\[
  |\eta_{1}||\nabla^{\alpha}\partial_{\xi_j}b_k^{(n)}|\leq
  \frac 12 C(g_k,k)\frac{\rho^{(n+1+|\alpha|)}R^k(n+1+|\alpha|+k)!}{(n+|\alpha|+k+2)^m}.\]
Altogether, the induction is complete, and we obtain that
\[
  C(b_k,k)\leq C(g_k,k)\leq (24\max_{\ell\leq
    k-1}C(b_{\ell},\ell)+1)\|r'\|_{JS^{\frac{\rho}2, \frac R4}_{m}},
\]
provided
\[
  m\geq C(d)\qquad \qquad \rho\geq 6\left(\frac 34\right)^m\qquad
  \qquad R\geq 2^{d+1}\rho.
\]
It remains to find $\rho,R,m$ as above such that
$\|r'\|_{S^{\rho,R}_m}<\frac{1}{24}$. Since $r'$ is of degree $-1$, this can
be done by fixing $m$ and $\rho$, then choosing $R$ large enough.

Having found $b$, it remains to build $a'$ such that
\[
  (1+a')(0)=0 \qquad \qquad \partial_t (1+a')(t) = b(t)\sharp (1+a')(t).
\]
We know that $b(t)\in JS^{\rho,R}_m$ for all $t\in [0,1]$, with uniformly
bounded norm. Thus, by Proposition \ref{prop:composition-formal}, this Cauchy problem satisfies
the hypotheses of the Picard-Lindelöf theorem on the Banach space
$JS^{2\rho,4R}_m$. In particular, we can find a solution in this space,
and $a'(1)$ is the requested operator.
\end{proof}

Grouping together Proposition \ref{prop:correct_D0_error_formal},
Proposition \ref{prop:canon_transf_to_local_model}, Proposition
\ref{prop:quantized-contact}, Theorem
\ref{prop:FIOs}, and Remark \ref{rem:replacement}, we obtain the following quantum normal
form.

\begin{prop}\label{prop:Sato_conj}
  Let $(z,\zeta)\in \Sigma^{\pm}$. Identify a neighbourhood of $z$ in
  $X$ with a neighbourhood of $0$ in $\R^{2d-1}$. There exists an open
  neighbourhood $U$ of $(z,\zeta)$ in $T^*\R^{2d-1}$ such that, for
  every pair of open sets $z\in U_2\Subset U_1 \Subset U$, there
  exists
  \begin{itemize}
  \item open neighbourhoods $V_2\Subset V_1$ of $\{x=y=\xi=0,\eta=(\pm 1, 0,\ldots, 0)\}$ in
    $T^*(\R^d_x\times \R^{d-1}_y)$
  \item an analytic Fourier Integral
    Operator $\hat{\kappa}:\mathcal{H}_{\omega}(U_1)\to\mathcal{H}_{\omega}(V_1)$
  \item an analytic Fourier Integral Operator
    $\widehat{\kappa^{-1}}:\mathcal{H}_{\omega}(V_2)\to \mathcal{H}_{\omega}(U_2)$
  \item an elliptic matrix-valued formal symbol
    $C\in S(V_1)$
  \item a neighbourhood $\Omega$ of $\diag(V_1)\times
    (0,+\infty)$ in $(\R^{2d-1}\times S^{2d-2})^2\times (0,+\infty)$
  \end{itemize}
  such that
  \[
    T\overline{\partial}_bT^*=\widehat{\kappa^{-1}}I_{\psi,\Omega,V_2}(C)TD_0T^*\widehat{\kappa} \qquad
    \qquad S=\widehat{\kappa^{-1}}\hat{\kappa}
  \]
  as operators mapping $\mathcal{F}_{\omega}(U_1)$ into $\mathcal{F}_{\omega}(U_2)$.

  Reciprocally, there exists an open neighbourhood $V$ of $\{x=y=\xi=0,\eta=(\pm 1, 0,\ldots, 0)\}$ in
    $T^*(\R^d_x\times \R^{d-1}_y)$ such that, for every pair of open
    sets $V_2\Subset V_1\Subset V$ containing this point, there exists
      \begin{itemize}
  \item open neighbourhoods $U_2\Subset U_1$ of $(z,\zeta)$ in
    $T^*R^{2d-1}$
  \item an analytic Fourier Integral
    Operator $\hat{\kappa}:\mathcal{H}_{\omega}(U_2)\to\mathcal{H}_{\omega}(V_2)$
  \item an analytic Fourier Integral Operator
    $\widehat{\kappa^{-1}}:\mathcal{H}_{\omega}(V_1)\to \mathcal{H}_{\omega}(U_1)$
  \item an elliptic matrix-valued formal symbol
    $C\in S(V_1)$
  \item a neighbourhood $\Omega$ of $\diag(V_1)\times
    (0,+\infty)$ in $(\R^{2d-1}\times S^{2d-2})^2\times (0,+\infty)$
  \end{itemize}
  such that
  \[
    \widehat{\kappa}T\overline{\partial}_bT^*\widehat{\kappa^{-1}}=I_{\psi,\Omega,V_2}(C)TD_0T^* \qquad
    \qquad S=\hat{\kappa}\widehat{\kappa^{-1}}
  \]
  as operators mapping $\mathcal{F}_{\omega}(V_1)$ into $\mathcal{F}_{\omega}(V_2)$.
\end{prop}

\subsection{The Szeg\H{o} kernel parametrix and the algebra of Toeplitz operators}
\label{sec:szegho-kern-param}

In this section, we broadly follow the method proposed in
\cite{boutet_de_monvel_sur_1975} to obtain a parametrix for the
Szeg\H{o} kernel starting with the normal form of Proposition
\ref{prop:Sato_conj}.

In this section, the index $_0$ refers to the ``model'' case
$\R^d\times S^{2d-1}$. In particular $S_0$ is the Szeg\H{o} projector
of Proposition \ref{prop:flat_Grauert}, and $\Sigma^{\pm}_0$ is the
positive, resp. negative, characteristic set of
$(\overline{\partial}_b)_0$, the boundary Cauchy-Riemann operator on
$\R^d\times S^{2d-1}$. By contrast, the notations $\overline{\partial}_b, S,
\Sigma^{\pm},$ refer to the boundary of a general, relatively compact, strongly
pseudo-convex open set whose boundary is denoted by $X$.

\begin{theorem}\label{prop:analytic-Szego}The Szeg\H{o} projector
  admits a (singular) integral kernel which is real-analytic away from
  the diagonal. Given a defining function $\rho$
  and its polarisation $\psi$, well-defined on a small neighbourhood
  $\Omega$ of the diagonal in $X\times X$, there exists a realisation
  of a formal analytic
  amplitude $a$ on $X\times X\times (0,+\infty)$ (see Definitions
  \ref{def:formal-amp}, \ref{def:analytic_amp} and Proposition
  \ref{prop:real_formal_symb} for details about the nature of $a$) such that
  \[
    (x,y)\mapsto S(x,y)-\int_0^{+\infty}e^{it\psi(x,y)}a(x,y;t)\dd t
  \]
  is real-analytic.
\end{theorem}
\begin{proof}The proof consists first in exhibiting a parametrix for
  $S$ in the form of a Fourier Integral Operator as above, then
  proving that this parametrix indeed approximates $S$ in the sense
  that the difference is a continuous operator from $\mathcal{O}'(X)$
  to $\mathcal{O}(X)$. This requires several steps, as the mere fact
  that $S$ acts nicely on $\mathcal{O}'(X)$ requires a proof.
\begin{enumerate}
\item 
Given $(z,\zeta)\in \Sigma^{\pm}$ and
$(z_0,\zeta_0)\in \Sigma^{\pm}_0$, applying Proposition
\ref{prop:Sato_conj} twice, we obtain a pair of Fourier integral
operators $\hat{\kappa}$ and $\widehat{\kappa^{-1}}$ between
neighbourhoods of $(z,\zeta)$ and $(z_0,\zeta_0)$ (such that
$\hat{\kappa}\widehat{\kappa^{-1}}$ and
$\widehat{\kappa^{-1}}\hat{\kappa}$ slightly decrease the definition
set and are equal to the restriction map) such that, for every
pair of small neighbourhoods $V_1\Subset V$ of $(z,\zeta)$ and every
pair of small neighbourhoods $U_1\Subset U$ of $(z_0,\zeta_0)$, 
\begin{align*}
  \hat{\kappa}T\overline{\partial}_bT^*\widehat{\kappa^{-1}}&=I_{\psi}(C)T(\overline{\partial}_b)_0T^*:\mathcal{F}_{\omega}(U)\to \mathcal{H}_{\omega}(U_1)\\
  T\overline{\partial}_bT^*&=\widehat{\kappa^{-1}}I_{\psi}(C)T(\overline{\partial}_b)_0T^*\hat{\kappa}:\mathcal{F}(V)\to \mathcal{H}(V_1).
\end{align*}
We now present the following local candidate for $TST^*$:
\[
  \mathcal{S}=\widehat{\kappa^{-1}}TS_0T^*\hat{\kappa}.
\]
One readily checks that
\[
  T\overline{\partial}_bT^*\mathcal{S}=0:\mathcal{F}_{\omega}(V)\to
  \mathcal{H}_{\omega}(V_1).
\]
Moreover $\mathcal{S}$ is self-adjoint, because of the constraints on
$\widehat{\kappa}$ and $\widehat{\kappa^{-1}}$ in Proposition
\ref{prop:quantized-contact} (modified by Remark \ref{rem:replacement}).

\item 
From the definition of $\mathcal{S}$, we see that $T^*\mathcal{S}T$ is an elliptic Fourier Integral Operator
whose Lagrangian $\Lambda$ is a half-line bundle over its projection onto the
base;%
it is also a projector. When writing $T^*\mathcal{S}T$ in a
form given by Proposition \ref{prop:FIO-rank1-phase}, once the phase
is chosen, the fact that it is a self-adjoint projector determines
every term of the formal amplitude. Therefore different
pieces of $T^*\mathcal{S}T$ can be glued together into a global
Fourier Integral operator $\widetilde{S}$ on $X$; in fact, this Fourier Integral
Operator takes the form proposed in the statement,
with a formal analytic amplitude.
\item Let us prove that $(\widetilde{S}-1)S$ is continuous from
  $L^2(X)$ to $\mathcal{O}(X)$. Given $u\in L^2(X)$, the wave front of
  $Su$ is included in $\Sigma$ by Proposition \ref{prop:WF_Pu} because
  $\overline{\partial}Su=0$. Now,
  given $(z,\zeta)\in \Sigma$ and small open neighbourhoods
  $V_1\Subset V$, letting $v$ be the restriction of $TSu$ to $V$, we can apply the previous technique and
  obtain that
  $\widehat{\kappa^{-1}}T(\overline{\partial}_b)_0T^*\hat{\kappa}v=0$
  as an element of $\mathcal{F}_{\omega}(V_1).$ By Proposition
    \ref{prop:loc-anal-hypoell-Grauert}, we deduce that
    $\widehat{\kappa^{-1}}T(S_0-1)T^*\hat{\kappa}v=0$ as an element of
    $\mathcal{F}_{\omega}(V_1)$; but this is equal to the restriction of
    $(\widetilde{S}-1)S u$ to $V_1$.
  \item Let us similarly prove that $\widetilde{S}(1-S)$ is continuous
    from $L^2(X)$ to $\mathcal{O}(X)$. By the results of Kohn \cite{kohn_harmonic_1963} and Boutet de
    Monvel \cite{boutet_de_monvel_hypoelliptic_1974}, there exists
    $F:L^2(X)\to L^2(X)$, such that
    $I_{L^2}=S+F\overline{\partial}_b$. Therefore, as operators on
    $L^2(X)$, one has
    \[
      \widetilde{S}=\widetilde{S}S+(F\overline{\partial}_b\widetilde{S})^*=\widetilde{S}S+\widetilde{S}(\overline{\partial}_b)^*F^*.
    \]
    As functions of the second variable, the phase and all the terms
    of the formal amplitude of $\widetilde{S}$ belong to the kernel
    of $\partial_b$, and therefore, integrating by parts, for every
    $u\in L^2(X)$, $\widetilde{S}(\overline{\partial}_b)^*u\in
    C^{\omega}(X)$.
  \item Since both $\widetilde{S}S-S$ and
    $\widetilde{S}-\widetilde{S}S$ are continuous from $L^2(X)$ to
    $\mathcal{O}(X)$, we deduce that $\widetilde{S}-S$ is continuous from
    $L^2(X)$ to $\mathcal{O}(X)$. In particular,
    $S=\widetilde{S}+(S-\widetilde{S})$ is continuous from
    $\mathcal{O}(X)$ to $\mathcal{O}(X)$, and therefore by duality $S$
    acts continuously on $\mathcal{O}'(X)$.
  \item Using item 5, we can repeat the proof of item 3 starting with $u\in
    \mathcal{O}'(X)$, and we now obtain that $(\widetilde{S}-1)S$ is
    continuous from $\mathcal{O}'(X)$ to $\mathcal{O}(X)$.
  \item Using item 5 again, we obtain that, by duality, $S-\widetilde{S}$ is
    continuous from $\mathcal{O}'(X)$ to $L^2(X)$, and therefore
    $(S-\widetilde{S})^2=S-\widetilde{S}S-S\widetilde{S}+\widetilde{S}$
    is continuous from $\mathcal{O}'(X)$ to $\mathcal{O}(X)$. Then, by
    item 6, both $\widetilde{S}S-S$ and its dual $S\widetilde{S}-S$
    are continuous from $\mathcal{O}'(X)$ to
    $\mathcal{O}(X)$. Therefore $S-\widetilde{S}$ is continuous from
    $\mathcal{O}'(X)$ to $\mathcal{O}(X)$, and the proof is complete.
\end{enumerate}
\end{proof}
An interesting consequence of Theorem \ref{prop:analytic-Szego} is the
\emph{analytic hypoellipticity} of $\overline{\partial}_b$: if $u\in
\mathcal{O}'(X)$ is such that $\overline{\partial}_bu$ is
real-analytic on an open set $V$, then $(1-S)u$ is real-analytic on
$V$; in fact, we obtain the much stronger property
\[
  WF_a((1-S)u)= WF_a(\overline{\partial}_bu).
\]
This seems to be a new result, interesting in its own right. Away from
the characteristic $\Sigma$ this property is of course the usual
ellipticity result (Proposition \ref{prop:WF_Pu}), and there are now
several proofs of microlocal analytic hypoellipticity on $\Sigma_-$ \cite{treves_analytic_1978,tartakoff_elementary_1981,sjostrand_analytic_1983}.

We are now able to generalise the construction of Section
\ref{sec:pseud-oper-1} to define the algebra of Toeplitz operators on
$X$.

\begin{prop}\label{prop:Toeplitz_general}
  Let $S(X)$ denote the space of formal analytic
  amplitudes on $X\times (0,+\infty)$. After a holomorphic extension,
  these amplitudes lead to Fourier Integral Operators acting on
  $\mathcal{F}_{\omega}(X)$, with singular kernel of the form
  \[
    I_{\psi}(a):(x,y)\mapsto
    \int_0^{+\infty}e^{-t\psi(x,y)}t^{n-1}\widetilde{a}(x,y;t)\dd
    t\qquad \qquad a\in S(X).
  \]
  \begin{enumerate}
    \item 
  These operators form an algebra under composition; the product law
  takes the form
  \[
    (a\sharp b)_{k}=\sum_{n+\ell=0}^kB_n(a_{\ell},b_{k-n-\ell}),
  \]
  where $(B_n)_{n\in \N}$ is a sequence of bilinear differential
  operator such that, for all $n\in \N$, $B_n$ is of total degree
  $2n$.

  \item 
  The principal symbol of $[I_{\psi}(a),I_{\psi}(b)]$ is $\{a,b\}$
  (using the natural symplectic structure on $X\times (0,+\infty)$).

  \item
  If the degree of $a$ is $d$, then $I_{\psi}(a)$ continuously sends
  $H^s(X)$ into $H^{s-d}(X)$ for every $s\in \R$.

  \item
    The Szeg\H{o}
    projector $S=I_{\psi}(s)$ is the unit for this algebra.
  \item Invertible
  elements are exactly those for which the principal symbol $a_0$ is bounded
  away from $0$.
  \item The space $S(X)$ is the union of Banach spaces on
    which $a,b\mapsto a\sharp b$ is continuous.
  \end{enumerate}
\end{prop}
\begin{proof}
  Except for the last two items, all claims follow from Theorem
  \ref{prop:FIOs}, Theorem \ref{prop:analytic-Szego}, the usual
  properties of Fourier Integral operators in smooth regularity, and
  the results of \cite{boutet_de_monvel_spectral_1981}. In particular,
  if $P$ is a (smooth or analytic) pseudodifferential operator on $X$,
  then $SPS$ is a (smooth or analytic) Fourier Integral Operator with
  phase $\psi$ and its principal symbol is the restriction of the
  principal symbol of $P$ to $\Sigma_+$ (identified with
  $X\times (0,+\infty)$ by the section $x\mapsto \partial \psi(x,x)$),
  so we can apply \cite{boutet_de_monvel_spectral_1981}, Chapter 2,
  Corollary 2 to compute commutators at main order.

  To prove invertibility of elliptic amplitudes and continuity of the
  product in suitable spaces of formal analytic amplitudes, we use again Proposition
  \ref{prop:Sato_conj}: the wave front set of these operators is
  $\{(z,\zeta,z,\overline{\zeta}),(z,\zeta)\in \Sigma^+\}$ and
  microlocally near any point of $\Sigma^+$, a Fourier Integral
  Operators conjugates $I_{\psi}(a)$, where $a$ is elliptic, with an operator of the form of
  Proposition \ref{prop:invesion_Toeplitz} with elliptic symbol. The
  latter can be inverted, and conjugating back we obtain, locally, an
  amplitude for the inverse. Since the composition rule is local, 
\end{proof}

One can quantize homogeneous contact transformations as in the flat
case; an important subclass of these transformations are fundamental
solutions to the Schrödinger equation for degree $1$ Toeplitz operators.
\begin{prop}
  Given a one-homogeneous symplectic transformation $\kappa:X\times
  (0,+\infty)\to X\times (0,+\infty)$, there
  exists a neighbourhood $\Omega$ of the graph of $[\kappa]:X\to X$, a function
  $\psi:\Omega\to \C$, CR-holomorphic in the first variable,
  anti-CR-holomorphic in the second variable, such that
  $\im(\psi)\equiv \dist(\cdot,{\rm Graph}([\kappa]))^2$. In particular,
  the function $\Omega\times (0,+\infty)\ni (x,y,t)\mapsto t\psi(x,y)$
  is a positive phase function, whose Lagrangian $\Lambda$ is CR-holomorphic in
  the first variable, anti-CR-holomorphic in the second variable, and
  such that $\pi(\Lambda_{\R})={\rm Graph}([\kappa])$.
\end{prop}
\begin{proof}
  We recover $\psi$ from its Lagrangian, which is unique by
  Proposition \ref{prop:completing_Lagrangians}. By design
  $\Lambda_{\R}$ is a half-line bundle over its projection, and this
  condition is open, so that $\Lambda$ satisfies the hypotheses of
  Proposition \ref{prop:FIO-rank1-phase}. This concludes the proof.
\end{proof}

\section{A few applications}
\label{sec:few-applications}

\subsection{Grauert tubes, FBI transforms, and quantum propagators}
\label{sec:appl-grau-tubes}

Let $(M,g)$ be a
compact, real-analytic Riemannian manifold. The holomorphic extension
of the exponential map gives a natural isomorphism between small
neighbourhoods of $M$ in $\widetilde{M}$ and small neighbourhoods of
the zero section in $T^*M$.

There exists \cite{guillemin_grauert_1991} a real-analytic
Kähler structure on a neighbourhood of the zero section in $T^*M$,
compatible with the complex structure (via the aforementioned
identification) and whose restriction to $M$ is $g$. The \emph{Grauert
  tube}
\[
  B_r=\{(x,\xi)\in T^*M, \|\xi\|_{g(x)}<r\}
\]
is then a strongly pseudoconvex open set for $r$ small, and its
boundary $X_r$ is a real-analytic submanifold of $T^*M$.

The Toeplitz operators on $X_r$ acquire a particular importance
because of the FBI transform (from $\mathcal{O}'(M)$ to $\mathcal{H}_{\omega}(X_r)$), defined as
follows: let $\Omega=\{(x,\xi,y)\in \overline{B_r}\times M, \dist(x,y)<r\}$; let
$\psi:\Omega\to \C$ be holomorphic in its
first factor and such that $\psi(x,\xi,x)=\tfrac{i}{2}(r-\|\xi\|_{g(x)}^2)$. This
function, when restricted to $(X_r\times M)\cap \Omega$, satisfies the following properties:
\begin{itemize}
\item $-C\dist(x,y)^2\leq \im(\psi)(x,\xi,y)\leq -c\dist(x,y)^2$ for
  some $0<c<C$.
\item $(\nabla_x,\nabla_\xi,\nabla_y)\psi(x,\xi,x)=(-\xi,-i\xi,\xi)$.
\end{itemize}
In particular, $X_r\times M\times (0,+\infty)\ni (x,\xi,y)\mapsto t\psi(x,\xi,y)$ is a positive phase function whose Lagrangian
$\Lambda$ is the only Lagrangian which is CR-holomorphic in the first
component and such that
\[
  \Lambda_\R=\{(x,-t\omega,\omega,0,y,-t\omega),t\in \R,(x,\omega)\in
  S^*M\}.
  \]
We can introduce our first guess for the FBI transform:
\[
  T_0(x,\xi,y)=\int e^{t\psi(x,\xi,y)}t^{\frac{n+1}{4}}\dd t.
\]
As in the flat case, the power of $t$ is chosen such that $T_0$ is
a bounded operator, with bounded inverse, between $L^2(M)$ and
$\ker_{L^2(X_r)}(\overline{\partial}_b)$; in other terms, $T_0^*T_0$ is an elliptic
pseudodifferential operator of degree $0$.

$T_0$ may not exactly be unitary, but we can remedy to this situation
in the following way: $T_0T_0^*$ has the same Lagrangian as $S$, and
therefore it is of the form $I_{\psi}(t_0)$ for some self-adjoint amplitude $t_0$;
there exists an inverse square root $r$ for $t_0$ (because the calculus of
these Toeplitz operators is, locally, conjugated to that of
pseudodifferential operators on flat space), and therefore, letting
\[
  T=I_{\psi}(r)T_0,
\]
then $T$ is a unitary transformation between $L^2(M)$ and
$\ker_{L^2(X_r)}(\overline{\partial}_b)$. Useful properties of $T$ are
as follows:
\begin{itemize}
\item If $A$ is an analytic pseudodifferential operator on $M$ %
  , then
  $TAT^*$ is a Toeplitz operator with principal symbol $(x,\omega,t)\mapsto
  \sigma(A)(x,-t\omega)$.
  
\item More generally, if $A$ is an analytic Fourier Integral operator
  on $M$, with Lagrangian $\Lambda_A$, then $TAT^*$ is a Fourier
  integral operator whose Lagrangian $\Lambda_{TAT^*}$ is a
  complex Lagrangian, CR-holomorphic in the first factor,
  anti-CR-holomorphic in the second factor, and such that
  \begin{align*}
    (\Lambda_{TAT^*})_{\R}=\{(x_1,-t_1\omega_1,\omega_1,0,x_2,-t_2\omega_2,\omega_2,0),&(x_1,\omega_1,x_2,\omega_2)\in
    (\Lambda_A)_\R, \\ &t_1>0, t_2>0, \\ &(x_1,\omega_1)\in S^*M, \\ &(x_2,\tfrac{t_2}{t_1}\omega_2)\in S^*M\}.
  \end{align*}
  Existence and uniqueness of this Lagrangian are guaranteed by
  Proposition \ref{prop:completing_Lagrangians}. In particular, in
  this case one has $\Sigma_+\simeq T^*M\setminus \{0\}$.
\end{itemize}

We can use this transformation to describe quantum
propagators as analytic Fourier Integral Operators.
\begin{prop}\label{prop:quantum_propag}
  Let $P$ be an analytic self-adjoint degree 1 pseudodifferential operator on
  $M$. For every $s\in \R$, the fundamental solution $e^{-isP}$ of the Schrödinger equation
  for $P$ is an analytic Fourier Integral Operator whose Lagrangian is
  the graph of the bicharacteristic flow $\phi$ of $P$ at time $s$. In particular,
  $Te^{-isP}T^*$ is a Fourier Integral operator with a singular kernel of the
  form
  \[
    (x,y)\mapsto \int_0^{+\infty}e^{it\psi(s,x,y)}a(s,x,y;t)\dd t.
  \]
  Here, $\psi(s,x,y)$ is CR-holomorphic with respect to $x$,
  anti-CR-holomorphic with respect to $y$, and satisfies $\im(\psi(x,y))\asymp
  -\dist(y,\phi_s(x))$; moreover $a$ is the realisation of a formal analytic amplitude.
\end{prop}
\begin{proof}
  We will prove the claimed description of $Te^{-isP}T^*$, from which
  the more abstract fact that $e^{-isP}=T^*(Te^{-isP}T^*)T$ is a
  Fourier Integral Operator will follow.

  By Proposition \ref{prop:repr_BG_smooth}, a pseudodifferential
  operator $Q$ on $X$ is such that $[Q,S]$ and 
  $SQS-TPT^*$ are continuous from $L^2$ to $L^2$, if and only if the
  principal symbol of $Q$, near $\Sigma_+$ is the CR-holomorphic extension of the
  principal symbol of $P$ (seen as a real-analytic function on the
  totally real manifold $\Sigma_+$).

  We know from the $C^{\infty}$ theory that $e^{-isQ}$ is, up to a
  smooth remainder, a Fourier
  Integral Operator whose Lagrangian is the graph of the
  bicharacteristic flow of the principal symbol of $Q$. Since the
  principal symbol Poisson-commutes with that of
  $(\overline{\partial}_b,\partial_b)$, the (smooth) Fourier Integral
  Operator $\Pi e^{-isQ}\Pi$ has exactly the requested
  Lagrangian. Letting $U_0(s)$ be any elliptic degree 0 analytic Fourier Integral Operator
  with the same Lagrangian (obtained by Proposition
  \ref{prop:quantized-contact}, for instance), the principal (order 1) symbols of
  $\frac{\dd}{\dd s}U_0(s)$ and $iU_0(s)TPT^*$ coincide
  (alternatively, one can compute principal symbols to show
  this). The phase of this Fourier Integral Operator is the same as
  that of $U_0(s)$.

  The operator $R(s)=U_0(s)e^{-isTPT^*}$ then solves the equation
  \[
    \frac{\dd}{\dd s}R(s)=\left[\tfrac{\dd}{\dd
        s}U_0(s)-iU_0(s)TPT^*\right]e^{-isTPT^*}.
  \]
  Introducing $V_0(s)$ an inverse for $U_0(s)$ (we already have an
  inverse if we applied Proposition \ref{prop:quantized-contact}),
  modulo a real-analytic remainder,
  \begin{equation}\label{eq:Rs}
    \frac{\dd}{\dd s}R(s)=\left[\tfrac{\dd}{\dd
        s}U_0(s)-iU_0(s)TPT^*\right]V_0(s).
  \end{equation}
  The right-hand side is a degree $0$ analytic Fourier Integral Operator with
  the same phase as $\Pi$; it is a degree $0$ Toeplitz operator in the
  sense of Proposition \ref{prop:Toeplitz_general}. Choosing a Banach
  norm of formal analytic amplitudes which contains the total symbol of $\left[\tfrac{\dd}{\dd
      s}U_0(s)-iU_0(s)TPT^*\right]V_0(s)$ and on which the formal
  product is continuous, one can give an approximate solution for
  \eqref{eq:Rs} by the Picard-Lindelöf
  theorem; thus there exists  an analytic
  Toeplitz operator $\widetilde{R(s)}$ such that, modulo an analytic remainder,
  \[
    \frac{\dd}{\dd s}\widetilde{R(s)}=\left[\tfrac{\dd}{\dd
        s}U_0(s)-iU_0(s)TPT^*\right]V_0(s).
  \]
  To conclude, we can apply the Duhamel formula (the true propagator
  $e^{isTPT^*}$ preserves real-analytic functions) and we obtain that,
  modulo an analytic remainder,
  \[
    e^{isTPT^*}=\widetilde{R(s)}^{-1}V_0(s).
  \]
  This proves the second part of the claim, and conjugating back, we
  obtain that
  $T^*\widetilde{R(s)}^{-1}V_0(s)T$ is a Fourier Integral Operator
  which coincides with $e^{isP}$ up to an analytic remainder.
\end{proof}

\subsection{From microlocal to semiclassical analysis}
\label{sec:from-micr-semicl}
In this section we broadly describe how one can obtain semiclassical results from
the description above.

In what follows, we let $X$ be the boundary of a compact, strongly pseudoconvex open set and suppose that
$X$ admits an $S^1$ action $(r_{\theta})_{\theta\in S^1}$ which preserves all the structure. Two
important examples are:
\begin{enumerate}
\item The case where $X$ is the Grauert tube around a compact
  Riemannian manifold of the form $S^1\times M$. The action of $S^1$
  on $S^1\times M$ (by rotation of the first factor) is an isometry
  and lifts to a symplectic transformation on
  $T^*(S^1\times M)$, which is also an isometry for the Kähler
  structure near the zero section.
\item The case where $X$ is the circle bundle of the dual of a
  positive line bundle over a Kähler manifold.
\end{enumerate}

Among Toeplitz operators, those whose formal amplitude is itself invariant
under the $S^1$ action commute with the $S^1$ action, and thus one can
decompose their action over the eigenmodes $\mathcal{H}_k,k\in \Z$ of
the $S^1$ action. Changing variables, the singular kernel of a
Toeplitz operator $I_{\psi}(a)$ restricted to $\mathcal{H}_k$ is then
\[
  (x,y)\mapsto
  k\int_{S^1}\int_{0}^{+\infty}e^{ik(s\psi(r_{\theta}x,y)-\theta)}a(x,y,ks)\dd
  \theta \dd s.
\]
To perform a stationary phase in the variables $(s,\theta)$, we need
to satisfy the condition
\[
  \alpha(V)\neq 0
\]
where $\alpha$ is the contact 1-form on $X$ of Proposition
\ref{prop:CR-structure} and $V$ is the tangent vector field
of the $S^1$ action. In case 2 above, this is always true because the
Reeb flow coincides with the $S^1$ action (in fact one can choose
$\psi$ with simple and explicit dependence on $\theta$, see for
instance the unlabelled equation following (24) in \cite{zelditch_pointwise_2018}). In case 1, the Reeb flow is
the geodesic flow on $S^1\times M$, so we restrict our attention to
the open set
\[
  X_+=\{(y,\eta,x,\xi)\in S^*(S^1\times M), \eta>0\}.
\]
Thus, in both cases, we are in position to apply the stationary phase lemma in the
variables $(s,\theta)$ (see for instance \cite{zelditch_szego_2000}),
and obtain a description modulo errors of size $e^{-ck}$ for some
$c>0$ (indeed, elements of analytic function spaces have exponentially
fast decaying Fourier modes). In case 2 above, we obtain semiclassical
Berezin--Toeplitz operators; in case 1, after a stereographic change
of variables $X_+\to S^1\times T^*M$, we obtain semiclassical
pseudodifferential operators.

With this point of view, one can for instance recover the expressions
of the semiclassical Szeg\H{o} kernel, and many properties of
semiclassical analytic pseudodifferential operators after conjugating
back by a semiclassical Bargmann transform, since we know that the
Toeplitz and pseudodifferential algebras are equivalent in the
analytic semiclassical case \cite{rouby_analytic_2018}.

This procedure can be generalised to other ``small parameter''
techniques, notably to study the action of (microlocal) analytic
Fourier Integral Operators on WKB states. We
hope to describe this in future work.

\appendix

\section{The algebra of symbols of pseudodifferential operators}
\label{sec:algebra-symb-pseud}

\begin{prop}\label{prop:composition-formal}
  Define the following formal analytic symbol norm on $\R^{d}_x\times \R^d_{\xi}$, for $m\in \R, \rho>0,R>0$:
  \[
    \|a\|_{S^{\rho,R}_m}=\sup_{k\in \N, \alpha\in \N^{2d}, (x,\xi)\in
      \R^{2d}}\frac{|\nabla^{\alpha}a_k(x,\xi)|(1+k+|\alpha|)^m}{\rho^{|\alpha|}R^k(|\alpha|+k)!}.
    \]
  For the star-product of left-quantization on $T^*\R^d$, there exists
  $m_0(d)$ such that
  \[
    \|a\sharp b\|_{S^{\rho,R}_m}\leq
    12\|a\|_{S^{\rho,R}_m}\|b\|_{S^{\frac{\rho}2,\frac R4}_m}
  \]
  if
  \[
    R\geq 2^{d+2}\rho^2\qquad \qquad m\geq m_0(d).
  \]
\end{prop}
\begin{proof}
  We start with the formula
  \[
    (a\sharp
    b)_k=\sum_{n=0}^k\sum_{l=0}^{k-n}\sum_{|\beta|=n}\frac{1}{\beta!}\partial^\beta_xa_l\partial_{\xi}^{\beta}b_{k-n-l}.
  \]
  Since
  \[
    \sum_{|\beta|=n}\frac{n!}{\beta!}=n^d\leq 2^{nd}
  \]
  we obtain that, for every $\alpha$ with $|\alpha|=j$,
  \begin{multline*}
    |\nabla^{\alpha}(a\sharp b)_k|\leq
    \|a\|_{S^{\rho,R}_m}\|b\|_{S^{\frac{\rho}2,\frac
        R4}_m}\frac{\rho^jR^k(j+k)!}{(j+k+1)^m}\times \\
    \sum_{n=0}^k\left(\frac{2^{d+1}\rho}{R}\right)^n\sum_{l=0}^{n-k}\sum_{j_1=0}^j\sum_{\substack{|\gamma|=j_1\\
        \gamma\leq
      \alpha}}\frac{(n+j_1+l)!(k-l+j-j_1)!j!}{4^{k-l}2^{j-j_1-n}n!(k+j)!j_1!(j-j_1)!}\left(\frac{1+j+k}{(1+n+j_1+l)(1+k-l+j-j_1)}\right)^m.
  \end{multline*}
  Let us prove that, if $0\leq j_1\leq j$ and $0\leq l+n\leq k$, then
  \[
    \frac{(n+j_1+l)!(k-l+j-j_1)!j!}{2^{k-l+j-j_1}n!(k+j)!j_1!(j-j_1)!}\leq
    1.
  \]
  If the other parameters are fixed, then
  \[
    l\mapsto
    \frac{(n+j_1+l)!(k-l+j-j_1)!j!}{2^{k-l+j-j_1}n!(k+j)!j_1!(j-j_1)!}\]
  is log-convex; there are two extremal points.

  At \underline{$l=0$}, we obtain
  \[
    \frac{(n+j_1)!(k+j-j_1)!j!}{2^{k+j-j_1}n!(k+j)!j_1!(j-j_1)!}\leq
    \frac{(n+j_1)!j!k!}{n!(k+j)!j_1!}.
  \]
  This increasing function of $j_1$ reaches a maximum at $j_1=j$,
  where we obtain
  \[
    \frac{(n+j)!k!}{n!(k+j)!}\leq 1
  \]
  since $n\leq k$.

  In the other case \underline{$l=k-n$}, we obtain
  \[
    \frac{(k+j_1)!(n+j-j_1)!j!}{2^{n+j-j_1}n!(k+j)!j_1!(j-j_1)!}\leq
    \frac{(k+j_1)!j!}{(k+j)!j_1!}\leq 1
  \]
  since $j_1\leq j$.  It remains to bound, for fixed $l,n$ such that
  $l+n\leq k$,
  \[
    \sum_{j_1=0}^j\sum_{\substack{|\gamma|=j_1\\ \gamma\leq
      \alpha}}\left(\frac{1+j+k}{(1+n+j_1+l)(1+k-l+j-j_1)}\right)^m
  \]
  To this end we use the fact that
  \[
    \#(|\gamma|=j_1,\gamma\leq \alpha)\leq \min(j_1,(j-j_1))^{d-1};
  \]
  thus, this sum is bounded by
  \[
    \sum_{j_1=0}^j\frac{(1+j+n+k)^m\min(j_1,j-j_1)^{d-1}}{(1+n+j_1+l)^m(1+k-l+j-j_1)^m}.
  \]
  Now we use \cite{deleporte_toeplitz_2018}, Lemma 2.13; for $m$ large
  enough (depending only on $d$), this sum is smaller than $3$.  To
  conclude,
  \[
    \sum_{l=0}^{n-k}\frac{1}{2^{k-n-l}}\leq 2
  \]
  and, if $R\geq 2^{d+2}\rho^2$,
  \[
    \sum_{n=0}^k\left(\frac{2^{d+1}\rho^2}{R}\right)^n\leq 2.
  \]
\end{proof}

\section{Functional spaces in analytic regularity}
\label{sec:funct-spac-analyt}

This section collects the basic properties of the spaces of
analytic functions and their duals.%
Most technical facts claimed here are proved in
\cite{hormander_introduction_1973}; see also Chapter 6 of \cite{treves_analytic_2022}.

\begin{defn}\label{def:space_hol_func}
  Let $\Omega\subset \C^n$ (or any complex paracompact manifold) be an
  open set. We define $\mathcal{O}(\Omega)$ as the space of
  holomorphic functions from $\Omega$ to $\C$, endowed with the
  topology of local uniform convergence.
\end{defn}
The topology on $\mathcal{O}'(\Omega)$ coincides with that of compactly
supported Radon measures, of which it is a closed subspace. Elements
of $\mathcal{O}'(\Omega)$ will be called \emph{analytic functionals}.

One can generalise this definition into ``germs'' of holomorphic
functions and their duals.

\begin{defn}
  Let $E\subset \C^n$ (or any complex paracompact manifold) be an open
  set. Define $\mathcal{O}(E)$ as the colimit of the
  spaces $\mathcal{O}(\Omega)$ for all $\Omega$ open and containing
  $E$ (that is to say, an element of $\mathcal{O}(E)$ is
  a holomorphic function on some open set containing $E$).
\end{defn}
If $E$ is an open subset or $\R^n$, or more generally an open subset
of a real-analytic submanifold of $\R^n$, then $\mathcal{O}(E)$ is the
space of real-analytic functions on $E$. If $E$ is a compact subspace
of $\R^n$ with non-empty interior, then $\mathcal{O}(E)$ is the space
of functions on $E$ which are real-analytic up to the boundary (and, therefore, which
extend into real-analytic functions on some open neighbourhood of
$E$).

Given two sets $E\subset F\subset \C^n$ one can naturally define a
restriction map from
$\mathcal{O}(F)$ to $\mathcal{O}(E)$. By duality, this defines a
natural map $\mathcal{O}'(E)\to \mathcal{O}'(F)$. This map can only be
injective when the restriction $\mathcal{O}(F)\to \mathcal{O}(E)$ has
dense image; this is called the \emph{Runge property}. We will only be
interested in the case where $E$ is a compact subset of $\R^n$, in
which case the Runge property is always satisfied, by elementary
approximation theory.

\begin{prop}\label{prop:real_cpts_are_Runge}(\cite{treves_analytic_2022},
  Theorem 6.2.14) Let $K$ be a compact subset of $\R^n$. Then entire functions
  are dense in $\mathcal{O}(K)$; that is to say, $K$ is Runge.
\end{prop}
The proof roughly consists in showing that, with
$\rho_N:x\mapsto=c_nN^{-n}\exp(-N|x|^2)$, given $f\in \mathcal{O}(K)$
and $g$ any smooth extension of $f$ to $\R^n$, the sequence of entire
functions $(\rho_N*g)_{N\in \N}$ converges towards $f$ in the topology
of $\mathcal{O}(K)$.

One of the main useful properties of Runge sets is the fact that one
can solve the $\overline{\partial}$ problem on them.
\begin{prop}\label{prop:solve_dbar_on_Runge}(\cite{hormander_introduction_1973}, Theorem 2.7.8) Let
  $K\Subset \C^n$ be compact and Runge. Then there exists a basis of
  neighbourhoods $(\Omega_j)_{j\in \N}$ of $K$ in $\C^n$ such that the
  following is true. For every $j\in \N$ and every
  $f=(f_1,\ldots,f_n)\in C^1(\Omega_j,\C^n)$ which is
  $\overline{\partial}$-closed, in the sense that
  \[
    \overline{\partial}_{k}f_\ell=\overline{\partial}_{\ell}f_k\qquad
    \forall 1\leq k,\ell\leq n,
  \]
  there exists $u\in C^1(\Omega_j,\C)$ such that $\overline{\partial}u=f$.
\end{prop}

The ability to prove the $\overline{\partial}$-problem on some
neighbourhoods of real compact sets allow us to describe how the
spaces $\mathcal{O}$ and $\mathcal{O}'$ behave under natural
opeartions on their domains.

\begin{prop}\label{prop:inclusion_union_O_spaces}
  Let $K_1,K_2$ be compact subsets of $\R^n$.
  \begin{enumerate}
  \item $\mathcal{O}(K_1)\cap \mathcal{O}(K_2)=\mathcal{O}(K_1\cup
    K_2)$.
  \item $\mathcal{O}(K_1)+ \mathcal{O}(K_2)=\mathcal{O}(K_1\cap
    K_2)$.
  \item $\mathcal{O}(K_1)'\cap \mathcal{O}'(K_2)=\mathcal{O}'(K_1\cap
    K_2)$.
  \item $\mathcal{O}(K_1)'+ \mathcal{O}'(K_2)=\mathcal{O}'(K_1\cup
    K_2)$.
  \end{enumerate}
\end{prop}
\begin{proof}~
  \begin{enumerate}
  \item Let $U_1,U_2$ be respective neighbourhoods of $K_1$ and $K_2$
    in $\C^n$. Then holomorphic functions on $U_1\cup U_2$ are exactly
    functions that are both holomorphic on $U_1$ and on $U_2$, because
    holomorphy is a local property.
  \item The
  inclusion from left to right is obvious and it remains to prove,
  given $f\in \mathcal{O}(K_1\cap K_2)$, that it is of the form
  $f_1+f_2$, where $f_1\in \mathcal{O}(K_1)$ and $f_2\in
  \mathcal{O}(K_2)$.

  Let $\Omega_1$ and $\Omega_2$ be small neighbourhoods in $\C^n$ of
  respectively $K_1$ and $K_2$, so that $f$ extends to
  $\widetilde{f}\in \mathcal{O}(\Omega_1\cap \Omega_2)$ (the function
  $\widetilde{f}$ may or may not correspond to the usual notion of
  holomorphic extension, depending on whether $K$ is the closure of
  its interior in $\R^n$).

  The compact sets $K_1\setminus \Omega_2$ and $K_2\setminus \Omega_1$
  do not intersect and therefore lie at positive distance from each
  other. We let $\chi\in C^{\infty}(\Omega_1\cup \Omega_2,\R)$ be such
  that $\chi=1$ on a neighbourhood $U_1$ of $K_1\setminus
  \Omega_2$ and $\chi=0$ on a neighbourhood $U_2$ of $K_2\setminus
  \Omega_1$. The one-form
  $\alpha=\overline{\partial}(\widetilde{f}\chi)$, well-defined on
  $\Omega_1\cap \Omega_2$, is equal to $0$ on $(U_1\cup U_2)\cap
  (\Omega_1\cap \Omega_2)$, and therefore can be smoothly extended by
  $0$ on $\Omega=U_1\cup U_2\cup (\Omega_1\cap \Omega_2)$.

  $\Omega$ is a neighbourhood of $K_1\cup K_2$, which is a Runge set
  by Proposition \ref{prop:real_cpts_are_Runge};
  by Proposition \ref{prop:solve_dbar_on_Runge}, this means that there
  exists a smaller neighbourhood $\Omega'\subset \Omega$ of $K_1\cup
  K_2$ and $u\in
  C^{\infty}(\Omega')$ such that $\overline{\partial}u=\alpha$.

  To conclude, the function $h_1=(1-\chi)\widetilde{f}+u$ is
  well-defined on $(U_1\cup (\Omega_1\cap \Omega_2))\cap \Omega'$,
  which is an open neighbourhood of $K_1$; it satisfies
  $\overline{\partial}h_1=\alpha-\overline{\partial}u=0$; therefore
  $h_1\in \mathcal{O}(K_1)$. In the same way,
  $h_2=\chi\widetilde{f}-u$ is well-defined and holomorphic on
  $(U_1\cup (\Omega_1\cap \Omega_2))\cap \Omega'$ which is an open
  neighbourhood of $K_2$; moreover, on $K_1\cap K_2$, one has
  $h_1+h_2=f$; this concludes this part of the proof.
\item This statement is the dual of item 2. %
\item This statement is the dual of item 1.
  \end{enumerate}
\end{proof}
Given two compacts $K_1\subset K_2$ of $\R^n$, there is no natural
restriction map from $\mathcal{O}'(K_2)$ to $\mathcal{O}'(K_1)$ (in
fact, we have a natural injective \emph{extension} map from $\mathcal{O}'(K_1)$
to $\mathcal{O}'(K_2)$). Nevertheless, there is a well-defined notion
of \emph{support} of an analytic functional, thanks to the injectivity
of this extension map and item 3 of
Proposition \ref{prop:inclusion_union_O_spaces}.

Let us remark that if we define $C^{\infty}(K)=\mathcal{E}(K)$, for $K$ any compact
of $\R^n$, following Whitney
\cite{whitney_analytic_1934}, then the equivalents of all items of
Proposition \ref{prop:inclusion_union_O_spaces} are also true, with
the small caveat that the dense maps $\mathcal{E}(\R^n)\to
\mathcal{E}(K)$ fail to be injective. We will use the following
restriction theorem on locally compact distributions: they can be
restricted to a smaller subset modulo a boundary indeterminacy.

\begin{prop}\label{prop:presheaf}
  Let $K_1\subset K$ be two compact sets. The inclusion map
  $\mathcal{E}'(K_1)\to \mathcal{E}'(K)$ has a left inverse modulo
  $\mathcal{E}'(\partial K_1)$. This right inverse does not depend on
  adding an element of $\mathcal{E}'(\partial K)$, and therefore
  defines a well-defined map
  \[
    \rho_{K,K_1}:\mathcal{E}'(K)/\mathcal{E}'(\partial K)\to
    \mathcal{E}'(K_1)/\mathcal{E}'(\partial K_1).\]
\end{prop}
\begin{proof}
  Let $K_2=\overline{K\setminus K_1}$. Recall that
  $\mathcal{E}'(K)=\mathcal{E}'(K_1)+\mathcal{E}'(K_2)$, and given
  $u\in \mathcal{E}'(K)$, choose $u_1\in \mathcal{E}'(K_1)$ such that
  $u-u_1\in \mathcal{E}'(K_2)$. The distribution $u_1$ is not unique,
  but given any other choice $v_1$, the difference $u_1-v_1\in
  \mathcal{E}'(K_1)$ is such that $(u-u_1)-(u-v_1)=v_1-u_1\in
  \mathcal{E}'(K_2)$. Therefore $u_1-v_1\in \mathcal{E}'(K_2\cap
  K_1)=\mathcal{E}'(\partial K_1)$, and the class of $u_1$ modulo
  $\mathcal{E}'(\partial K_1)$ is uniquely defined. If one had $u
  in \mathcal{E}'(K_1)$ to begin with, then one can choose $u_1=u$, so
  that this map is indeed a left inverse to the extension map.

  If $u\in \mathcal{E}'(\partial K)$, then, since $\partial K\subset
  (\partial K_1)\cup K_2$, one can choose $u_1\in
  \mathcal{E}'(\partial K_1)$
  in the lines above, and therefore $u$ is mapped to $0$; this
  concludes the proof.
\end{proof}
The same property holds when $\mathcal{E}'$ is replaced with
$\mathcal{O}'$.
The two families of spaces $\mathcal{E}'(K)/\mathcal{E}'(\partial K)$ and
$\mathcal{O}'(K)/\mathcal{O}'(\partial K)$ are quite
practical because of this well-defined restriction map, which mimics
the ability to restrict analytic or smooth functions to a smaller
set. Elements of the spaces $\mathcal{O}'(K)/\mathcal{O}'(\partial K)$ are
called hyperfunctions. Proposition \ref{prop:presheaf} and its
equivalent for $\mathcal{O}'$ means that these families of spaces (indexed by $\mathring{K}$) form \emph{pre-sheaves}.

\begin{defn}\label{def:support_analytic_functional}
  Let $K\Subset \R^n$ and let $f\in \mathcal{O}'(K)$. The support of
  $f$ is the smallest compact $K_1\subset K$ such that $f\in \mathcal{O}'(K_1)$.
\end{defn}
This notion is well-defined: first of all ``$f\in \mathcal{O}'(K_1)$''
makes sense because the extension $\mathcal{O}'(K_1)\to
\mathcal{O}'(K)$ is injective; second, if two compacts $K_1$ and $K_2$
are such that $f\in \mathcal{O}'(K_1)\cap \mathcal{O}'(K_2)$, then
$f\in \mathcal{O}'(K_1\cap K_2)$ by item 3 of
Proposition \ref{prop:inclusion_union_O_spaces}.

If $U$ is a relatively compact open set in $\R^n$ then
$\mathcal{O}(\overline{U})\subset C^{\infty}(\overline{U})$ and
therefore $\mathcal{E}'(\overline{U})\subset
\mathcal{O}'(\overline{U})$. In this case, the support of a
distribution $T\in \mathcal{E}'(\overline{U})$ coincides with its
support as an element of $\mathcal{O}'(\overline{U})$.

We now reach the main result of this appendix: one can patch together
analytic functionals defined on different compact sets which agree on
the intersection (meaning that the support of their difference lies
away from the intersection), and if they only agree on the
intersection modulo a
real-analytic function then we can patch them together modulo a
real-analytic function.

\begin{prop}\label{prop:patching_analytic_func_or_distros}
  Let $K_1,K_2$ be compact subsets of $\R^n$. Let $u_1\in
  \mathcal{O}'(K_1)$ and $u_2\in \mathcal{O}'(K_2)$ be such that ${\rm
    supp}(u_1-u_2)\cap K_1\cap K_2=\emptyset$. Then there exists $u\in
  \mathcal{O}'(K_1\cup K_2)$ such that ${\rm supp}(u-u_1)\cap K_1=\emptyset$
  and ${\rm supp}(u-u_2)\cap K_2=\emptyset$. If $u_1,u_2$ belong to
  $\mathcal{E}'$, then so does $u$.

  More generally, if there exists $f\in \mathcal{O}(K_1\cap K_2)$ such
  that ${\rm
    supp}(u_1-u_2-f)\cap K_1\cap K_2=\emptyset$, then there exists
  $u\in \mathcal{O}'(K_1\cup K_2)$, $g_1\in \mathcal{O}(K_1)$ and
  $g_2\in \mathcal{O}(K_2)$ such that ${\rm supp}(u-u_1-g_1)\cap
  K_1=\emptyset$ and ${\rm supp}(u-u_2-g_2)\cap
  K_2=\emptyset$. Again if $u_1,u_2$ belong to $\mathcal{E}'$, then so
  does $u$.
\end{prop}
\begin{proof}
  By hypothesis, there exists two compacts  $L_1\in K_1\setminus K_2$ and $L_2\in
  K_2\setminus K_1$ such that $u_1-u_2\in \mathcal{O}'(L_1\cup
  L_2)$. Using Proposition \ref{prop:inclusion_union_O_spaces}, there
  exists $f_1\in \mathcal{O}'(L_1)$ and $f_2\in \mathcal{O}'(L_2)$
  such that $u_1-u_2=f_1+f_2$.

  Then $u=u_1-f_2=u_2+f_1$ is an element of $\mathcal{O}'(K_1\cup K_2)$
  such that $u-u_1=f_2\in \mathcal{O}'(L_2)$ and $u-u_2=f_1\in
  \mathcal{O}'(L_1)$. This concludes this part of the proof, and we can
  seamlessly replace $\mathcal{O}'$ with $\mathcal{E}'$ in the lines above.

  If now $u_1-u_2\in \mathcal{O}'(L_1\cup
  L_2)+\mathcal{O}(K_1\cup K_2)$, then by writing $\mathcal{O}(K_1\cup
  K_2)=\mathcal{O}(K_1)+\mathcal{O}(K_2)$ we can correct $u_1$ and
  $u_2$ by respective elements of $\mathcal{O}(K_1)$ to reduce
  ourselves to the previous case.
\end{proof}

To conclude, we mention the generalisation of these results to
general paracompact real-analytic manifolds. Any paracompact
real-analytic manifold is an analytic submanifold of $\R^n$ for $n$
large enough \cite{grauert_levis_1958}, so that we only need to
restrict our attention to compact sets which belong to this
submanifold; thus there is no additional difficulty. Of course, when
considering the restriction map as in Proposition \ref{prop:presheaf},
the boundary is then taken with respect to the topology of the
submanifold, not with respect to the ambient $\R^n$ topology.

\nocite{*}
  
\bibliographystyle{alpha} \bibliography{main}
\end{document}